\documentclass[a4paper,11pt]{article}

\usepackage{amsmath,amssymb,amsthm}
\usepackage[margin=2cm]{geometry}
\usepackage{color}

\usepackage[stretch=10]{microtype}
\usepackage{mathtools}
\usepackage{appendix}
\usepackage[utf8]{inputenc}
\usepackage[T1]{fontenc}
\usepackage[english]{babel}
\usepackage{dsfont}
\usepackage{mathrsfs}
\usepackage{hyperref}
\usepackage{graphicx}

\newcommand\MTkillspecial[1]{
  \bgroup
  \catcode`\&=9
  \let\\\relax%
  \scantokens{#1}%
  \egroup
}
\newcommand\DeclarePairedDelimiterMultiline[3]{
  \DeclarePairedDelimiter{#1}{#2}{#3}
  \reDeclarePairedDelimiterInnerWrapper{#1}{star}{
    \mathopen{##1\vphantom{\MTkillspecial{##2}}\kern-\nulldelimiterspace\right.}
    ##2
    \mathclose{\left.\kern-\nulldelimiterspace\vphantom{\MTkillspecial{##2}}##3}}
}

\DeclarePairedDelimiterMultiline{\abs}{\lvert}{\rvert}
\DeclarePairedDelimiterMultiline{\norm}{\lVert}{\rVert}
\DeclarePairedDelimiterMultiline{\pare}{(}{)}
\newcommand{\dd}{\mathop{}\!\mathrm{d}}
\newcommand{\inter}{\overset{\circ}} 

\newcommand{\Ptop}{P_{\mathrm{top}}}
\newcommand{\Pvar}{P_{\mathrm{var}}}
\newcommand{\PGur}{P_{\mathrm{Gur}}}
\renewcommand{\epsilon}{\varepsilon}

\newcommand{\calN}{{\mathcal{N}}}
\newcommand{\calP}{{\mathcal{P}}}

\newcommand{\calF}{\mathcal{F}}
\newcommand{\calK}{{\mathcal{K}}}
\newcommand{\calH}{\mathcal{H}}

\newcommand{\calS}{\mathcal{S}}
\newcommand{\calM}{\mathcal{M}}

\newcommand{\bbN}{\mathbb{N}}

\newcommand{\R}{\mathbb{R}}

\newcommand{\bbZ}{\mathbb{Z}}
\newcommand{\bbR}{\mathbb{R}}
\newcommand{\bbP}{\mathbb{P}}

\newcommand{\ds}{\displaystyle}

\renewcommand{\phi}{\varphi}
\renewcommand{\tilde}{\widetilde}

\numberwithin{figure}{section}

\DeclareMathOperator*{\esssup}{ess\,sup}
\DeclareMathOperator{\diam}{diam}
\DeclareMathOperator{\Leb}{Leb}

\newtheorem{theo}{Theorem}[section]
\newtheorem{lemm}[theo]{Lemma}
\newtheorem{prop}[theo]{Proposition}

\newtheorem{rema}[theo]{Remark}
\newtheorem{defi}[theo]{Definition}
\newtheorem{coro}[theo]{Corollary}

\title{Pressure at infinity and strong positive recurrence in negative curvature}
\author{Sébastien Gouëzel,    Barbara Schapira, Samuel Tapie}
\date{February 26, 2023}

\begin{document}
\maketitle
\begin{center}
With an appendix by Felipe Riquelme  \end{center}

 \begin{abstract} In the context of geodesic flows of noncompact negatively curved manifolds, we propose three different definitions of entropy and pressure at infinity, through
growth of periodic orbits, critical exponents of Poincaré series, and entropy (pressure) of invariant measures. We show that these notions coincide.

Thanks to these entropy and pressure at infinity, we investigate thoroughly the notion of strong positive recurrence in this geometric context.
A potential is said to be strongly positively recurrent
when its pressure at infinity is strictly smaller than the full topological pressure. We show in particular that if a potential is strongly positively recurrent,
then it admits a finite Gibbs measure.
We also provide easy criteria allowing to build such strong positively recurrent potentials and many examples.
\end{abstract}

()
()

\section{Introduction}

The geodesic flow of a compact connected negatively curved Riemannian manifold $M$ is the typical
geometrical example of an {\em Anosov flow}. Its chaotic behavior reveals
itself in particular through the existence of infinitely many possible
different behaviors of orbits. 

A {\em Gibbs measure} is an ergodic invariant probability measure associated
with a given continuous map $F:T^1M\to \bbR$, with respect to which almost
all orbits will spend most of their time in the subsets of $T^1M$ where the
potential $F$ is large (see Section~\ref{sec:Gibbs}   for the precise
definition). In particular, 
the fact that there exists a
Gibbs measure for all Hölder-continuous maps is a quantified way to express
that numerous behaviors of orbits are indeed realized as typical trajectories
with respect to the Gibbs measures of some Hölder-continuous potentials.

When the manifold $M$ is no longer assumed to be compact, a geometric construction
developed in~\cite{PPS} allows to build good candidates for Gibbs
measures. However, due to noncompactness of $M$ and $T^1M$, these measures are not necessarily finite,
and therefore not always extremely useful.

In~\cite{PS16}, Pit and Schapira characterized the finiteness of these
measures in terms of the convergence of some geometric series.
In~\cite{ST19}, in the case of the zero potential $F=0$, building
on~\cite{PS16}, Schapira and Tapie proposed a criterion, called {\em strong
positive recurrence}, which implies the finiteness of the associated measure,
known as the {\em Bowen-Margulis-Sullivan measure}. This criterion is the
following one. If $\Gamma=\pi_1(M)$, recall that the {\em critical exponent of
$\Gamma$} is the exponential growth rate of any orbit of $\Gamma$ acting on
the universal cover $\tilde{M}$ of $M$. By a result of Otal and
Peigné~\cite{OP}, it also coincides with the topological entropy of the
geodesic flow on $T^1M$. In~\cite{ST19}, a {\em critical exponent at
infinity} $\delta^\infty_\Gamma$ is defined, and the authors prove that a
critical gap $\delta_\Gamma^\infty<\delta_\Gamma$ implies that the
Bowen-Margulis-Sullivan measure is finite. This had been previously shown by
Dal'bo, Otal and Peigné in~\cite{DOP} for \emph{geometrically finite
manifolds}, for which the critical exponent at infinity is the maximum of the
critical exponents among parabolic subgroups. In general, this critical
exponent at infinity should be seen as a kind of {\em entropy at infinity}.
Other striking applications of this critical gap have been proved
in~\cite{CDST}.

The main goal of this paper is to produce a complete study of strong positive
recurrence in negative curvature. First, in
Sections~\ref{sec:pressures-at-infinity},~\ref{sec:cinq}
and~\ref{sec:ErgoPressure}, we compare this critical exponent at infinity
with other, new and old, possible definitions of entropy at infinity and show
that they all coincide. At the same time, considering pressures and pressures
at infinity instead of entropies, we generalize this study to all Gibbs
measures studied in~\cite{PPS, PS16}. In a second part
(Section~\ref{sec:SPR}), we give a detailed study of strong positive
recurrence in negative curvature. The appendix by F. Riquelme proves
important properties of entropy, that are classical in the compact case, but
need a careful proof in the noncompact case.

Analogous results have been known for years in the context of symbolic
dynamics over a countable alphabet, see~\cite{Gurevic,
Gurevic2,GS,Sa99,Sarig01,Ruette,BBG06,BBG}.

Let us present our results with more details.

The {\em topological pressure} of a Hölder-continuous potential $F:T^1M\to
\bbR$ is a weighted version of entropy. For a dynamical system on a compact
space, there are a lot of different definitions, which all coincide, see for
example~\cite[Ch 9]{Walters} or~\cite{Bowen75}. In the noncompact setting,
some of these definitions are meaningless. In~\cite{PPS}, following the works
of~\cite{Roblin, OP} on entropy, three definitions were compared. The {\em
Gurevič Pressure} $\PGur(F)$ is the weighted exponential growth rate of the
periodic orbits of the geodesic flow which cross a fixed compact set. The {\em variational pressure}
$\Pvar(F)$
 is the supremum over all invariant probability measures of their measure-theoretic pressures, that is a weighted version of their
Kolmogorov-Sinai entropies. The {\em geometric pressure} $\delta_\Gamma(F)$,
a geometric notion specific to geodesic flows also known as {\em critical
exponent of $(\Gamma, F)$}, is the  weighted  exponential growth rate of the
orbits of the fundamental group $\Gamma$ of $M$ acting on its universal
cover~$\tilde M$.

\medskip
All the previous discussion applies to the larger setting when $\tilde M$ is
still a complete simply connected Riemannian manifold with pinched negative
curvature and bounded derivatives of the curvature, and $\Gamma$ is a
discrete group of isometry acting properly on $\tilde M$ possibly with fixed
points. In this case, the stabilizer of any point has finite order and $M =
\tilde M/\Gamma$ is a \emph{good orbifold}. As considered in~\cite{PPS}, the
\emph{unit tangent bundle} $T^1M$ is then the set of parameterized
bi-infinite geodesics on $M$ with its natural projection from $T^1 \tilde M$
and geodesic flow. A Hölder/smooth map on $M$ (resp.\ on $T^1M$) is a map on
$M$ (resp.\ $T^1M$) whose lift to $\tilde M$ (resp.\ $T^1 \tilde M$) is
Hölder/smooth.  In the sequel, all the results which we present for smooth
manifolds can be adapted verbatim to this good orbifold setting. When slight
adaptations are required for this generalization, we will specify it in the
proof. In the appendix (only), we will restrict to the case where $\Gamma$
has a subgroup of finite index without torsion. Note that since $\Gamma$ may
not be finitely generated, this is not automatic in our setting.

\medskip

{\bf Standing assumptions in the paper:} {\em We fix once and for all a
nonelementary complete connected Riemannian manifold (or good orbifold) $M$ with pinched
negative sectional curvature, and bounded first derivative of the sectional curvature.}   For all
the statements of this section, let us also fix $F:T^1M\to \bbR$ a
Hölder-continuous potential.

It has been shown in~\cite{Roblin,OP} when $F\equiv 0$ and~\cite[Thm
1.1]{PPS} for general potentials that all these pressures coincide.

\begin{theo}[Roblin, Otal-Peigné, Paulin-Pollicott-Schapira]\label{th:Variationnel}
All notions of pressure coincide:
 \begin{equation*}
\delta_\Gamma(F) = \Pvar(F) = \PGur(F).
\end{equation*}
We denote this common value by $\Ptop(F)$, and we call it the {\em topological pressure} of $F$.
\end{theo}

The terminology differs slightly from~\cite{PPS}, where $\Pvar$ was called
topological pressure. In retrospect, we consider now that the above
terminology is better.

We propose  here  three notions of pressure at infinity, whose precise
definitions will be given in Section~\ref{sec:pressures-at-infinity}. The
{\em Gurevič pressure at infinity} $\PGur^\infty(F)$ measures the weighted
exponential growth rate of periodic orbits staying most of the time outside
any given compact set. The {\em variational pressure at infinity}
$\Pvar^\infty(F)$ is the least upper bound of measure-theoretic pressures of
invariant probability measures supported mostly outside any given compact
set. The {\em geometric pressure at infinity} $\delta_\Gamma^\infty(F)$
measures the  weighted exponential growth rate of those orbits of the
fundamental group $\Gamma$ corresponding to excursions outside any given
compact set.

The first main result of this article is the following one.

\begin{theo}\label{th:AllPressionEquivalent} All notions of pressures at infinity coincide:
\[
\delta_\Gamma^\infty(F) = \Pvar^\infty(F) = \PGur^\infty(F).
\]
We denote this common value by $\Ptop^{\infty}(F)$, and we call it the {\em topological pressure at infinity} of $F$.
\end{theo}

In the special case where $F$ tends to a constant at infinity, the equality
$\delta_\Gamma^\infty(F) = \Pvar^\infty(F)$ has also been announced
in~\cite{Velozo} using different methods.

As already implicitly or explicitly noticed for example
in~\cite{EK,EKP,IRV,Riquelme-Velozo}, this pressure at infinity is deeply
related to the phenomenon of loss of mass at infinity. In the vague topology, on a
noncompact space, a sequence of probability measures   may
converge to a finite measure with smaller total mass. As proven by the above
authors, if these probability measures have a larger Kolmogorov-Sinai entropy than the entropy
at infinity, then they cannot lose the whole mass and converge to the zero
measure. In this spirit, as a corollary of
Theorem~\ref{th:PressureMassInfty}, we obtain in
Corollary~\ref{coro:PressureMassInfty} the following result.

\begin{theo}\label{theo:coro4.8} Let $(\mu_n)$ be a sequence
of invariant probability measures on $T^1M$ converging in the vague topology
to a finite measure $\mu$, with mass $0\le\norm{\mu}\le 1$. Assume that $\int
\inf(F, 0) \dd\mu_n >- \infty$ for all $n$. Then their Kolmogorov-Sinai
entropies $h_{KS}(\mu_n)$ satisfy the following inequality:
\[
\limsup_{n\to \infty} \left(\,h_{KS}(\mu_n)+\int F \dd\mu_n\,\right) \le (1-\norm{\mu}) \Ptop^\infty(F)+\norm{\mu} \Ptop(F)\,.
\]
\end{theo}

In the geometrically finite case,~\cite{IRV,Riquelme-Velozo} obtain an
improvement of the conclusion of the theorem, with $P_\mu(F)$ instead of
$\Ptop(F)$ on the right, but only for the particular class of potentials $F$
which converge to $0$ at infinity, for which $\Ptop^\infty(F) =
\Ptop^\infty(0)$.  An extension of the results of~\cite{Riquelme-Velozo} to
general manifolds has been announced in~\cite{Velozo-phd}, cf also~\cite[Thm
1.1]{Velozo}. The strategy used in these papers is different from  ours, and
does not work yet in general. It would be interesting to obtain their sharper
inequality under our weaker assumptions, see~\cite[Conjecture 5.5]{Velozo}.

\bigskip

Once Theorem~\ref{th:AllPressionEquivalent} is proven, we can say that a
potential $F$ is {\em strongly positively recurrent} (SPR) when the following
{\em pressure gap} holds:
\begin{equation*}
\Ptop^\infty(F)<\Ptop(F)\,.
\end{equation*}
We refer the reader to Section~\ref{sec:SPR} for   the notions of recurrence,
positive recurrence, strong positive recurrence.

An analogous notion of {\em pressure gap} for potentials on nonpositively
curved manifolds, with respect to the set of singular vectors instead of
infinity, has been introduced in~\cite{BCFT}.

As in~\cite[Thm 7.1]{ST19} when $F=0$, we prove the following extremely
useful property of SPR potentials.

\begin{theo}\label{theo:SPR-implies-PR}  If the potential $F$ is  strongly
positively recurrent, then it admits a finite Gibbs measure.
\end{theo}

For potentials which vanish at infinity, this has also been announced
in~\cite[Theorem 1.3]{Velozo} using a different strategy. We will show that,
on any negatively curved manifold, there exist strongly positively recurrent
potentials, see Corollary~\ref{coro:existence-pot-SPR}. This implies the
following new result.

\begin{coro}
There exists a Hölder-continuous potential on $T^1M$ which admits a finite Gibbs
measure.
\end{coro}

It may be worth pointing out that with their current proofs, all results
of~\cite{Velozo} quoted above actually rely on the existence of such a
potential with finite Gibbs measure (see Lemma~3.9 of~\cite{Velozo}).
Nevertheless to our knowledge, this fact had not been established beyond
geometrically finite manifolds.

\medskip

We also establish other useful properties. Let $m$ be a  finite or infinite
Radon measure, invariant under the geodesic flow $(g^t)$. For a given compact
subset $K$ in $ M$, and $T\ge T_0$,  consider the set $V_{T_0,T}(K)$ of
vectors $v \in T^1K$, such that for any $t\in [T_0,T]$, the vector $g^t v$
does not belong to $T^1K$. These sets $\left(V_{T_0,T}(K)\right)_{T>T_0}$
decrease when $T\to +\infty$. We say that the flow $(g^t)$ is {\em
exponentially recurrent} with respect to the measure $m$  if there exist a
compact set $K\subset M$ whose interior intersects a closed geodesic, and
constants $C,\alpha,T_0>0$ such that for all $T>T_0$,
\begin{equation*}
m(V_{T_0,T}(K))\le Ce^{-\alpha T}\,.
\end{equation*}
In Section~\ref{exp-rec}, we establish the following theorem.

\begin{theo} \label{theo:exp-rec} Assume that $F$ has finite topological pressure and
finite Gibbs measure $m^F$. Then $F$ is  strongly positively recurrent if and
only if  the geodesic flow $(g^t)$ is exponentially recurrent with respect to
the Gibbs measure $m^F$.
\end{theo}

Strong positive recurrence says that
there exists a compact subset $K$ of $M$ such that the weighted exponential
growth rate of the excursions outside $K$ is strictly smaller than the
topological pressure. We finish this work with
Theorem~\ref{theo:indep-compact}, showing that strong positive recurrence
does not really depend on the chosen compact set $K$, in the following sense:
We show in Theorem~\ref{theo:indep-compact} that if the potential $F$ is
strongly positively recurrent, then for any compact subset $K$  of $M$, as
soon as the interior of $K$ intersects a closed geodesic, this exponential
growth rate of excursions outside $K$ is strictly smaller than the
topological pressure.

\medskip

The first two Sections~\ref{sec2} and~\ref{sec:trois} contain preliminaries,
first on negatively curved geometry and dynamics, and second on thermodynamic
formalism,  in particular all different notions of pressures, and the
construction of the Gibbs measure $m^F$.

Sections~\ref{sec:pressures-at-infinity},~\ref{sec:cinq}
and~\ref{sec:ErgoPressure}  on the one hand, and  Section~\ref{sec:SPR} on
the other hand can be read independently.

Section~\ref{sec:pressures-at-infinity} contains three different definitions
of pressures at infinity. In Section~\ref{sec:cinq}, we give upper bounds on
the growth of certain sets of periodic orbits in terms of entropy and entropy
at infinity. We deduce equality of the geometric and Gurevič pressures at
infinity $\delta_\Gamma^\infty(F)$ and $\PGur^\infty(F)$. In
Section~\ref{sec:ErgoPressure}, we show that geometric and variational
pressures at infinity $\delta_\Gamma^\infty(F)$ and $\Pvar^\infty(F)$
coincide. These sections are the technical heart of the paper.

Section~\ref{sec:SPR}  is more conceptual. We investigate the notion of
strongly positively recurrent potentials in our geometric context, and prove
Theorems~\ref{theo:SPR-implies-PR} and~\ref{theo:exp-rec}.

The appendix by Felipe Riquelme (Theorem~\ref{theo:entropies-coincide}) shows
that different possible definitions of measure-theoretic entropy, the
Kolmogorov-Sinai entropy, the Brin-Katok entropy, and the Katok entropy
coincide in our geodesic flow context. This result is well known in the
compact case, but not obvious at all without compactness.

\bigskip

The authors thank warmly both Jérôme Buzzi for numerous enlightening
discussions about strong positive recurrence, and the referee for his very
detailed reading and his comments that improved a lot our text. We
acknowledge the support of the Centre Henri Lebesgue ANR-11-LABX- 0020-01 and
ANR grant CCEM (ANR-17-CE40-0034).


\section{Negative  curvature, geodesic flow}\label{sec2}

\subsection{Geometric preliminaries}\label{sec21}
Our assumptions and notations are close to those of~\cite{PPS,PS16, ST19}.

Let $(M,g)$ be a smooth complete connected noncompact Riemannian manifold
with pinched negative sectional curvature  $-b^2 \leq K_g \leq -a^2$, for
some $a,b>0$, and bounded first derivative of the sectional curvature.
 Let $\tilde M$ be its universal cover,
$\Gamma=\pi_1(M)$ its fundamental group, and $p_\Gamma:\tilde M\to M=\tilde
M/\Gamma$ the quotient map. We assume that the group $\Gamma$ is
nonelementary, i.e., that the geodesic flow admits at least three different
periodic orbits on $T^1M$. In particular $\Gamma$ contains a free group (see
for instance~\cite{Bowditch}). We denote by $T^1M$ and $T^1\tilde M$ the unit
tangent bundles of $M$ and $\tilde M$, and by $\pi:T^1M \to M$ or $\pi :
T^1\tilde M\to \tilde M$ the canonical bundle projection. By abuse of
notation, we also write $p_\Gamma:T^1\tilde M\to T^1M$ for the differential
of $p_\Gamma$.

Given any two points $x,y\in \tilde M$, the set $[x,y]\subset\tilde M$ will
denote the (unique) geodesic segment between $x$ and $y$.

 We  fix arbitrarily a point $o\in \tilde M$ that we call \emph{origin}.
The boundary at infinity $\partial \tilde M$ is the set of equivalence
classes of geodesic rays staying at bounded distance one from another. The
{\em limit set} $\Lambda_\Gamma\subset\partial \tilde M$ is the set of
accumulation points $\Lambda_\Gamma=\overline{\Gamma o}\setminus\Gamma o$ of
the orbit of $o$. As shown by Eberlein~\cite{Eberlein}, the nonwandering set
$\Omega\subset T^1M$ of the geodesic flow is the union of geodesic orbits which
admit a lift whose negative and positive endpoints belong to
$\Lambda_\Gamma$. The {\em radial limit set}
$\Lambda^{\mathrm{rad}}_\Gamma\subset\Lambda_\Gamma$ is the set of endpoints
of geodesics whose images through $p_\Gamma$ return infinitely often in some
compact set:
\[
\Lambda_{\Gamma}^{\mathrm{rad}} \coloneqq
\{\xi\in\Lambda_\Gamma,\exists C>0,\exists (\gamma_n)\in\Gamma^\bbN, \gamma_n o\to \xi, d(\gamma_n o,[o\xi))\le C \}\,.
\]

We denote by $(g^t)_{t\in \bbR}$  the geodesic flow acting on $T^1M$ or
$T^1\tilde M$. The metric $g$ induces a distance on $M$ and $\tilde M$ that
we will simply denote by $d$. We will also denote by $d$ the distance on
$T^1M$ (resp.\ on $T^1 \tilde M$) defined as follows: for all $v,w\in T^1M$
(resp.\ in $T^1\tilde M)$, let
\[
d(v,w)\coloneqq \sup_{t\in [-1, 1]} d(\pi (g^t v), \pi (g^t w)).
\]
This distance is not Riemannian but it is equivalent to the standard Sasaki
metric on $T^1M$ (resp.\ on $T^1\tilde M$), see~\cite[Chap.\ 2]{PPS} for a
discussion on the subject.

The Busemann cocycle is defined for all $ \xi\in\partial\tilde M$ and $x,y
\in \tilde M$, by
\begin{equation}\label{eq:Busemann}
\beta_\xi(x,y) = \lim_{z\in [x,\xi),\,z\to \xi} d(x,z) - d(y,z).
\end{equation}
We will sometimes also write, for all $x,y,z\in \tilde M$,
\[
\beta_z(x,y) = d(x,z) - d(y,z).
\]
The set of oriented geodesics of $\tilde M$ can be identified with
\[
\partial^2 \tilde M = (\partial \tilde M \times \partial \tilde M) \backslash {\rm Diag}\,.
\]
For all $v\in T^1\tilde M$, denote by  $v^\pm$ the negative and positive
endpoints in $\partial \tilde M$ of the geodesic tangent to $v$.
 The unit tangent bundle $T^1\tilde M$ is homeomorphic to $\partial^2\tilde M \times \bbR$ via the \emph{Hopf parametrization}
\begin{equation}\label{Hopf}
\calH : \left\{\begin{array}{ccc}
T^1\tilde M & \to & \partial^2 \tilde M \times \bbR\\
v & \mapsto & (v^-, v^+, \beta_{v^+}(o, \pi (v)))
\end{array}\right.\,.
\end{equation}
The geodesic flow acts by translation in these coordinates: for all $v =
(v^-, v^+, s)$ and   $t\in \bbR$,
\[
g^t(v^-, v^+, s) = (v^-, v^+, t+s)\,.
\]
The group $\Gamma$ acts in these coordinates by
\[
\gamma(v^-,v^+,s)=\left(\gamma v^-,\gamma v^+, s+\beta_{v^+}(\gamma^{-1}o,o)\right)\,.
\]
In terms of these Hopf coordinates, the nonwandering set $\Omega$ is
identified with $((\Lambda_\Gamma^2\backslash {\rm Diag})\times
\bbR)/\Gamma$. We denote its lift $p_\Gamma^{-1}\Omega$ by $\tilde\Omega$.

Recall that an isometry $\gamma\in \Gamma$ is {\em hyperbolic} when it admits
exactly two fixed points in $\partial\tilde M$. In this case, it acts by
translation on the geodesic joining them. The set $\mathcal{P}$ of periodic
orbits of the geodesic flow on $T^1M$ is in $1-1$ correspondence with the set
of conjugacy classes of hyperbolic elements of $\Gamma$. Indeed, a periodic
orbit $p$ with period $\ell(p)$ can be lifted to a collection $p_\Gamma^{-1}(p)$ of geodesic orbits of
$T^1\tilde M$, and each of them, once projected on $\tilde M$, is the
oriented  translation axis of a unique hyperbolic element $\gamma_p$, which acts by
translation in the positive direction on the axis, with translation length
equal to $\ell(p)$. By construction, all these elements are conjugated one to
another.

Not all elements of $\Gamma$ are hyperbolic. However, the following lemma
from~\cite[lemma 2.6]{PS16}, variant of the well known point of view, due to
Margulis, of counting elements of $\Gamma$ inside cones, will allow us to
consider only hyperbolic elements.

\begin{lemm}\label{lem:Pit-Schapira2.6}
Let $\tilde K$ be a compact subset of $\tilde M$ whose interior intersects
$\tilde \Omega$. There exist finitely many elements $g_1,\dotsc, g_k$  in
$\Gamma$  such that for every $\gamma\in\Gamma$, there exist $g_i, g_j$ among
them such that $g_j^{-1}\gamma g_i$ is hyperbolic, and its translation axis
intersects $\tilde K$.
\end{lemm}
\begin{proof}
By Lemma~\cite[lemma 2.6]{PS16} applied with $\tilde W$ the interior of $\tilde K$, there exist  finite sets $F=\{g_1,\dotsc,
g_k\}$ and  $S = \{s_1, \dotsc, s_j\}$ in $\Gamma$ such that every $\gamma \in \Gamma\setminus S$ satisfies the
conclusion of the lemma with respect to $F$. Consider a hyperbolic element $h$ whose axis
intersects $\tilde K$. Then the set $F' = \{g_1,\dotsc, g_k, s_1,\dotsc, s_j,
h\}$ works for every $\gamma \in \Gamma$. Indeed, it works for $\gamma\notin
S$ by assumption, and for $\gamma=s_i\in S$ then $s_i^{-1} \gamma h = h$ has
a translation axis intersecting $\tilde K$, with $s_i, h \in F'$.
\end{proof}

The following elementary lemma  will be used several times.

\begin{lemm}\label{lm:NegCurvTriangle}
Consider $x,y,z$ three points in a geodesic metric space $\tilde M$, and denote by $[y,z]$ a geodesic between $y$ and
$z$. Then
\[
d(y,x)+d(x,z)-2d(x,[y,z])\le d(y,z)\le d(y,x)+d(x,z)\,.
\]
\end{lemm}

We will often need more precise distance estimates, which rely on a negative
upper bound of the curvature. The next lemma follows
from~\cite[Lemma~2.5]{PPS}.

\begin{lemm}\label{lm:NegCurv4Points}
For all $D>0$ and all $\epsilon>0$, there exists $T_0 = T_0(D, \epsilon)>D$
such that if $x,x',y,y'\in \tilde M$  satisfy $d(x,x')\leq D$, $d(y,y')\leq
D$ and $d(x,y)\geq 2T_0$, then there exists $s_0\in [-T_0, T_0]$ such that,
if $v_{xy}$ (resp.\ $v_{x'y'}$) denotes the unit tangent vector based at $x$
(resp.\ $x'$) tangent to the segment $[x,y]$ (resp.\ $[x',y']$), then for all
$t\in [T_0, d(x,y) - T_0]$,
\[
d(g^t v_{xy}, g^{t + s_0} v_{x'y'})\leq \epsilon.
\]
\end{lemm}

We will also need the following lemma which allows to approximate broken geodesics by axes of hyperbolic elements.
If $x\neq y\in \tilde M$, let $v_{xy}$ denote the (oriented and unitary) tangent vector of the geodesic segment $[x,y]$ at $x$.
 If $v,w\in T^1_x\tilde M$, set $\measuredangle (v,w)\in [0, \pi]$ for their geometric angle.
  If $v\in T^1_x \tilde M$ and $w\in T^1_y \tilde M$, denote by $\measuredangle (v,w)\in [0, \pi]$
 the geometric angle between $v$ and the image
of $w$ through the parallel transport from $y$ to $x$ along $[y,x]$.

\begin{lemm}\label{lm:GeodBrisee}
For all $\theta\in (0, \pi)$, and all $\epsilon>0$, there exists $C =
C(\theta, \epsilon)>0$ such that the following holds. Let $x,y,z,b\in \tilde
M$ and $\gamma\in \Gamma$ be such that $d(x,y), d(y,z)$ and $d(z, b)$ are at
least $2C$, and $d(b, \gamma x) \leq 1/C$. Assume moreover that the angles
$\measuredangle \left(v_{yx}, v_{yz}\right)$, $\measuredangle \left(v_{zy},
v_{zb}\right)$,
 and $\measuredangle \left(\gamma v_{xy},v_{bz}\right)$ are at least $\theta$.
 Then $\gamma$ is hyperbolic, the piecewise geodesics $[x,y]\cup [y,z] \cup [z,b]$ is in the $\epsilon$-neighborhood of its axis
except in the $C$-neighborhood of the points $x,y,z$ and $b$. Moreover, the
period $T_\gamma$ of $\gamma$ satisfies
\[
T_\gamma - (6C+1) \leq d(x,y) + d(y,z) + d(z,b) \leq T_\gamma + 6C+1.
\]
\end{lemm}

\begin{figure}[ht!]\label{geodbrisee}
\begin{center}
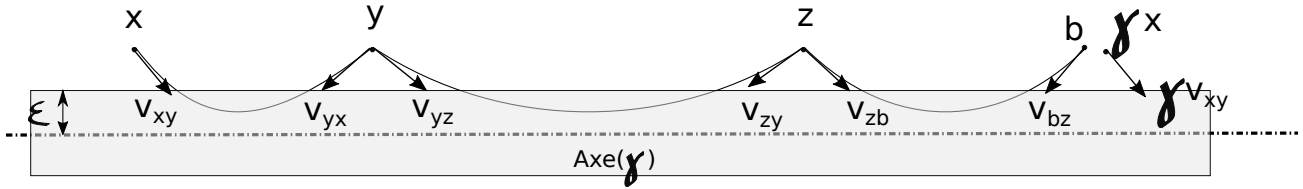
\caption{Broken geodesic close to a hyperbolic axis}
\end{center}
\end{figure}
\begin{proof}[Sketch of proof]
By the arguments presented in~\cite[p.~98]{PPS}, the geodesics from $x$ to
$b$ and from $x$ to $\gamma x$ are uniformly close to the union of segments
$[x,y]\cup [y,z]\cup[z,b]$, so that $v_{xy}$ and $v_{x,\gamma x}$ on the one
hand, and $v_{bz}$ and $v_{\gamma x,x}$ on the other hand, are uniformly
close. In particular, adjusting the constants, it implies that the angle
between $v_{\gamma x,x}$ and $\gamma v_{x,\gamma x}$ is uniformly bounded
from below by, say, $\theta/2$.

When $C$ is large enough, this prevents $\gamma$ to be parabolic.
Indeed, in this case, $v_{\gamma x,x}$ and $\gamma v_{x,\gamma x}$ would be  close to the vector  from  $\gamma x$
 to the parabolic fixed point of $\gamma$, and therefore very close one from another.

The rest of the proof is an immediate adaptation of arguments
of~\cite[p.~98]{PPS}.
\end{proof}


\subsection{Dynamical properties of the geodesic flow}


In restriction to its nonwandering set $\Omega$, the geodesic flow satisfies nice dynamical properties. It is  {\em transitive}
in the sense that  for all nonempty open sets
    $U,V\subset \Omega$, there exists $T>0$ such that $g^TU\cap
    V\neq\emptyset$.
    And it satisfies a {\em closing lemma} (see for instance Eberlein in~\cite[Prop.\ 4.5.15]{Eb96}): for every compact subset $K\subset \Omega$ and all
    $\varepsilon>0$, there exist $\eta>0$ and $T = T(K, \epsilon)>0$, such that for all $v\in K$,
    and $t>T$ such that $d(g^tv,v)\le \eta$, there exists a periodic vector
    $p$ whose period $\ell(p)$ satisfies $\abs{\ell(p)-t}\le \varepsilon$, and for
    all $0\le s\le t$, $d(g^tp,g^tv)\le \varepsilon$.


However, we will need similar properties for vectors close to $\Omega$ but
that may be wandering, and we will also need to make sure that the glued
orbit enters an a priori fixed ball. In this direction, we will use several
times the following proposition.

\begin{prop}[Connecting lemma]\label{lem:connecting}
Let $K$ and $K'$ be compact subsets of $M$ whose interiors intersect $\pi
(\Omega)$, and $\tilde K$ a compact subset of $\tilde M$   such that
$p_\Gamma(\tilde K)= K$. For all $\epsilon>0$, there exist $T_0 = T_0(\tilde
K, K', \varepsilon)>0$ and $C_0 = C_0(\tilde K, K', \varepsilon)>0$ such that
the following holds. There exists a construction that associates to any $T\ge
2T_0$ and any $v\in T^1K$ such that $g^T v\in T^1K$ a periodic orbit $\wp(v,T)$
that satisfies the following assertions.
\begin{enumerate}
\item[(1)] (Shadowing) The periodic orbit $\wp(v,T)$ has a period belonging
    to $[T, T+T_0]$, it  intersects the interior of $T^1K'$, and there
    exists a periodic vector $u$ on this periodic orbit, such that for all
    $t\in [T_0, T-T_0]$, we have $d(g^tv, g^tu )\leq \varepsilon$.
\item[(2)] (Bounded multiplicity) For  each periodic orbit $p$ with period
    $T=\ell(p)\geq 2T_0$ going through $T^1 K$, choose arbitrarily a periodic vector $v_p \in
     T^1 K \cap p$ on $p$, and denote by $\wp(v_p, \ell(p))$ the corresponding new periodic orbit  associated with  $v_p$ by our
     specific construction  in (1). Then, given any periodic
     orbit $\wp_0$, the number of periodic orbits $p$ such that $\wp(v_p,
     \ell(p))=\wp_0$ is bounded by $C_0 \ell(\wp_0)$.
\end{enumerate}
\end{prop}

\begin{rema}
\label{rmk:closing_hard} The first assertion of the above proposition is  a
standard consequence of transitivity, local product structure, and closing
lemma when $v\in \Omega$, but needs a proof otherwise.

The second assertion is more subtle than other similar statements that hold
in a compact setting. When $M$ is compact, one usually simply bounds the
number of periodic orbits that $\epsilon$-shadow a fixed orbit $\wp_0$ during
most of their period. However, when the manifold (or orbifold) $M$ is not compact, its
injectivity radius is not necessarily bounded from below. Therefore,
uniformly bounded multiplicity for the number of closed geodesics that stay
in a fixed $\varepsilon$-neighborhood of $\wp_0$ is not true in general,
notably when $\wp_0$ crosses parts of the manifold where the injectivity
radius is much smaller than $\epsilon$.  That is the reason why we consider
only those periodic orbits that are constructed through a given procedure,
detailed in the proof below, for which we are able to bound the multiplicity.
\end{rema}

In the proof and later on, we will need the following notation.  As in~\cite{PS16}, if $\tilde K\subset \tilde M$ is a compact subset, let
us denote by  $n_{\tilde K }(\wp )$ the number of lifts
$\tilde \wp $ of a given periodic orbit $\wp $ to $T^1\widetilde M$ such that $\pi(\tilde \wp )$ intersects $\tilde
K $.

\begin{proof}

The construction of the orbit $\wp(v, T)$ will be explained inside the proof
of the first assertion, and the specificities of the construction will be
used in the proof of the second assertion.

{\em Proof of Assertion (1)}. The reader may follow the proof on
Figure~\ref{connecting}. We can assume that $2\epsilon$ is smaller than $1$.
We fix once for all a vector $w$ in the intersection of $\Omega$ and the
interior of $T^1 K'$. Up to reducing $\epsilon$, we can assume that $B(\pi
(w), 2 \epsilon) \subset K'$.

By compactness of $\tilde K$, as $\Lambda_\Gamma$ is not reduced to a single point, there exists $\theta =
\theta(\tilde K)>0$ such that for all $y\in\tilde K$ and $\tilde v\in
T^1_y\tilde K$, there exists $\xi\in\Lambda_\Gamma$ such that $\measuredangle
\left( v_{y\xi},\tilde v\right) >  \theta$. As the geodesic flow is
topologically transitive on $\Omega$, and the action of $\Gamma$ on
$\Lambda_\Gamma$ is minimal, we can assume moreover that the positive geodesic orbit
on $T^1M$ associated with $(g^t v_{y\xi})_{t\ge 0}$ contains $\Omega$ in its
closure. Let $C = C(\theta, \epsilon)$ be the constant provided by
Lemma~\ref{lm:GeodBrisee}. Let $\varepsilon'=\min(\varepsilon, 1/(2C)) \leq
\epsilon$. By compactness of $T^1K\cap \Omega$, a uniform property of
transitivity holds, in the following sense.
 There exist $T_1>2C$ and $T_2>T_1+6C$, that depend only on $K,K'$ and $\epsilon'$, such that  the vector $v_{y\xi}$ can be chosen in such a way
that the projection on $T^1M$ of $g^{[2C,T_1]}(v_{y\xi})$ intersects $ B(w,\varepsilon')$ and the projection on $T^1M$ of
$g^{[T_1+6C+1,T_2]} (v_{y\xi})$ intersects once again $B(w,\varepsilon')$.

\begin{figure}[ht!]\label{connecting}
\begin{center}
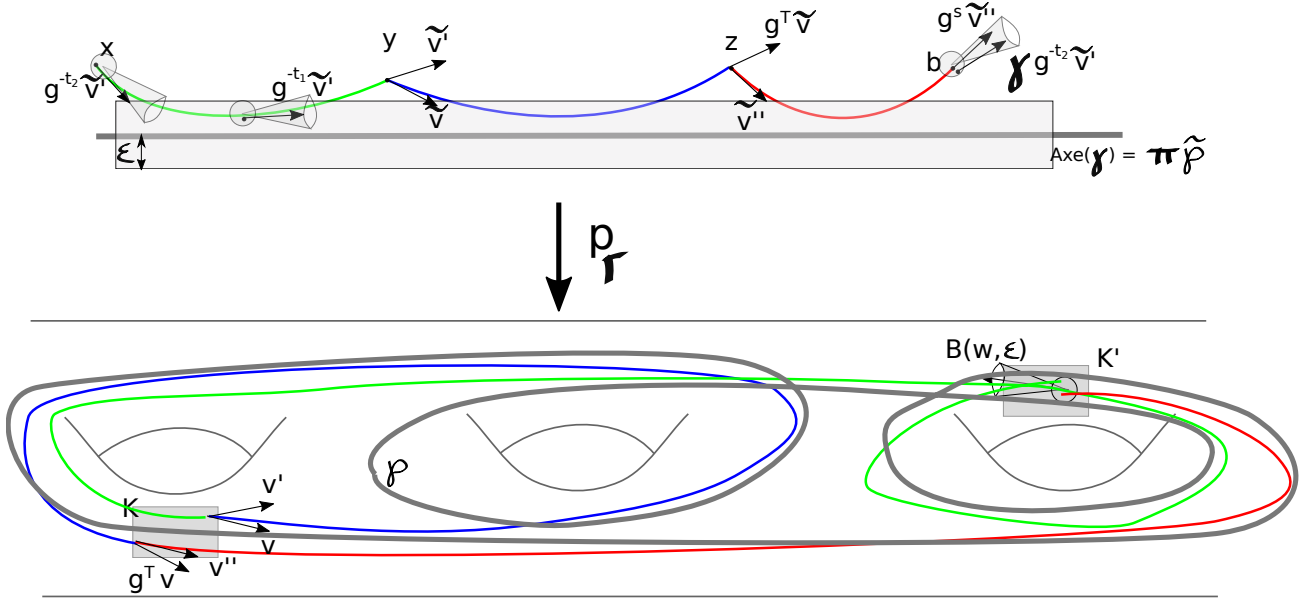
\caption{Connecting lemma}
\end{center}
\end{figure}

Let $v\in T^1K$. Set $y_0 = \pi (v)\in M$ and take $y\in \tilde K$ such that
$p_\Gamma(y) = y_0$. Let $\tilde v\in T^1_{y}\tilde M$ be such that
$p_\Gamma(\tilde v) = v$. By the above claim,  there exists
$\tilde v'=-v_{y\xi}\in T^1_{y} \tilde M$ with $\measuredangle \left(\tilde v, \tilde
v'\right) \leq \pi-\theta$ such that, with $v'=p_\Gamma(\tilde v')$, the half
orbit $(\{g^{-t}v' , t\geq 0\})$ is dense in $\Omega$, and  at two distinct
times $t_1\in[2C,T_1]$ and $t_2\in[T_1+6C+1,T_2]$, we have $g^{-t_1} v'\in
B(w, \epsilon')$, and $g^{-t_2}v'\in B(w, \epsilon')$. We will see below how
it will be important. Set $x = \pi (g^{-t_2} \tilde v')$.

\medskip
By assumption, $g^Tv\in T^1K$ for some $T\geq 0$ large enough (to be made
precise later on). Set $z = \pi( g^T \tilde v)$. By the same arguments now
applied to $-g^T\tilde v$, there exists $\tilde v''\in T^1_{z}\tilde M$ with
$\measuredangle \left(g^T\tilde v,  \tilde v''\right) \leq \pi -\theta$ such
that, if $v'' = p_\Gamma(\tilde v'')$, the half orbit $( g^tv'')_{t\geq 0}$
is dense in $\Omega$, and for some  $s \in [2C, T_1]$, $g^{s}v''\in B(w,
\epsilon')$. Let $b = \pi (g^{s }\tilde v'')$ be the base point of $g^s
\tilde v''$.

Consider now the broken geodesic path $(g^{t}g^{-t_2}\tilde v')_{0\le t\le
t_2}\cup (g^t \tilde v)_{0\le t\le T}\cup (g^t \tilde v'')_{0\le t\le s}$. It
starts from $x=\pi(g^{-t_2}\tilde v')$, has an angle at least $\theta$ at
$y=\pi(\tilde v)$, a second angle at least $\theta$ at $z=\pi(g^T\tilde v)$,
and ends at $b=\pi(g^s\tilde v'')$. Since $p_\Gamma(x)$ and $p_\Gamma(b)$ are
both in $\pi(B(w, \epsilon'))$, there exists $\gamma\in \Gamma$ such that
$d(\gamma x, b)\leq 2\epsilon'\leq 1/C$. Moreover, if $\epsilon$ is small
enough, since $g^{-t_2}v'\in B(w, \epsilon')$ and $g^s v'' \in B( w,
\epsilon')$ with $\epsilon'\leq \epsilon$, up to changing $\gamma\in\Gamma$,
the angle $\measuredangle \left(\gamma g^{-t_2} \tilde v', g^s \tilde
v''\right)$ is at most $\pi-\theta$.

Assume that $T\geq 2C$. Then Lemma~\ref{lm:GeodBrisee} applies to the
sequence of points $x,y,z,b$. Therefore, $\gamma$ is hyperbolic. Its translation axis can
be written as $\pi(\tilde{\wp})$, where we choose its lift $\tilde{\wp}$ to
$T^1 \tilde M$ oriented so that $\gamma$ acts by positive translation on it.
By Lemma~\ref{lm:GeodBrisee} again, the broken geodesic path $[x,y]\cup [y,z]
\cup [z,b]$ is in the $\epsilon$-neighborhood of $\pi (\tilde{\wp})$, except
maybe in the $C$-neighborhood of $x,y,z,b$. As we chose $ v'$ so that
$g^{-t_1} v'\in B(w, \epsilon') \subseteq B(w, \epsilon)$, with $t_1
\in [2C, d(y,x)-2C]$, the periodic orbit $\wp = p_\Gamma(\tilde{\wp})$
intersects $B(w, 2\epsilon)\subset T^1K'$ near $g^{-t_1}v'$. Moreover, since
$T_1+6C+1 \leq d(x,y)\leq T_2$ and $d(y, z)=T$ and $d(z,b)\leq T_1$, it
follows from Lemma~\ref{lm:GeodBrisee} that the translation length $\ell(\gamma)$ of
$\gamma$ satisfies
\[
(T_1+6C+1) + T -(6C+1)\leq \ell(\gamma) \leq T_2 + T + T_1 + 6C+1\,.
\]
To conclude, let $q'$ be the first point
 on the geodesic path $[y, z]$ which lies in the $\epsilon$-neighborhood
of the axis of $\gamma$, and define a point $q$ as the closest point to $q'$
on the translation axis of $\gamma$.    Let $\sigma=d(y,q')\geq 0$ so that
$q'=\pi (g^\sigma\tilde v)$. Let $\tilde u'$ be the tangent vector to the
axis of $\gamma$ at the point $q$ pointing in the same direction as $g^\sigma
\tilde v$. By construction, the vector $u \coloneqq
p_\Gamma(g^{-\sigma}\tilde u')$ satisfies that for all $\sigma \le t\le T-C$,
$d(\pi(g^{t}u),\pi(g^t v))\le\varepsilon$. With $\tau=C+1$,
the same kind of estimate holds on $T^1M$: for all $\tau\le t\le T-\tau$,
$d(g^tu,g^tv)\le\epsilon$. This shows that $u$ satisfies the conclusion of
Assertion (1), with $T_0=\max(\tau,T_1+T_2+6C+1)$.

The above procedure of construction of the periodic orbit $(g^tu)$ depends on
several arbitrary choices. We define $\wp(v, T)$ as one arbitrary periodic
orbit obtained by the above construction.
\medskip

{\em Proof of Assertion (2)}. For each periodic orbit $p$ and $v_p\in T^1K\cap p$ as in the
statement, let $\tilde v_p \in T^1 \tilde K$ be the lift of $v_p$ to the
universal cover used in the first step in order to define $\wp(v_p, \ell(p))$. Let
$\gamma_p \in \Gamma$ be the hyperbolic element whose axis is the lift of $p$
through $y_p=\pi(\tilde v_p)$, oriented in the direction of $\tilde v_p$, and
whose translation length is $\ell(p)$.

Assume that $\wp(v_p, \ell(p))$ is equal to a given periodic orbit $\wp_0$.
Then, by the construction in the first step, there exist a constant $C_1$
(depending only on $\tilde K$, $K'$ and $\epsilon$), a vector $\tilde u_p\in T^1
\tilde M$ and a lift $(\tilde \wp_{0})(p)$ of $\wp_0$, that may depend on
$p$, admitting a fundamental domain $((g^t \tilde u_p))_{0\leq t\leq
\ell(\wp_0)}$ whose projection on $\widetilde M$ is within Hausdorff distance at most $C_1$ of $[y_p,
\gamma_p y_p]$. In particular, this lift intersects the $C_1$-neighborhood
$\tilde K_{C_1}$ of $\tilde K$ as $\pi(\tilde u_p) \in \tilde K_{C_1}$.
Conversely, given a lift $\tilde \wp_0$ of $\wp_0$ intersecting $T^1\tilde
K_{C_1}$, let us show that the number of $p$ with $(\tilde \wp_{0})(p) =
\tilde \wp_0$ is uniformly bounded. The point $\pi(\tilde u_p)$ can only
belong to a compact part of $\tilde \wp_0$ (of length at most $\diam \tilde K
+ 2C_1$), hence $\pi(g^{\ell(\wp_0)} \tilde u_p)$ is also restricted to a
subset of diameter $\diam \tilde K + 2C_1$, and therefore $\gamma_p y_p$ is
also restricted to a subset of diameter $\diam \tilde K + 3C_1$. Moreover
$y_p$ belongs to the compact set $\tilde K$. For any $R>0$, there exists a
constant $A(R)$ such that, for any $x\in \tilde M$, the number of elements
$\gamma$ of $\Gamma$ with $\gamma \tilde K \cap B(x,R) \ne \emptyset$ is
bounded by $A(R)$: if this number is nonzero, one can pull back by one of
these elements to bring $B(x,R)$ to a fixed size neighborhood of $\tilde K$,
where the result is obvious by compactness. It follows that the number of
possible $\gamma_p$ is uniformly bounded by $A(\diam \tilde K + 3C_1)$, as
claimed.

We have proved that there exists a uniform constant $A$ depending only on
$\tilde K$, $K'$ and $\epsilon$ such that the number of periodic orbits $p$ with
$\wp(v_p, \ell(p)) = \wp_0$ is bounded from above by $A$ times the number $n_{\tilde K_{C_1}}(\wp_0)$  of
lifts $\tilde{\wp}_0$ of $\wp_0$ that intersect $T^1\tilde K_{C_1}$.

It remains to bound the number $n_{\tilde K_{C_1}}(\wp_0)$  of such lifts $\tilde{\wp}_0$ of $\wp_0$.
Assertion~(2) follows from the fact that there exists a constant $B=B(\tilde
K,C_1)>0$ such that for every periodic orbit $\wp_0\subset T^1M$,
\begin{equation}\label{lm:nW}
n_{\tilde K_{C_1}}(\wp_0)\leq B \times \ell(\wp_0)\,.
\end{equation}
Let us prove this bound. As $\tilde K_{C_1 + 1}$ is compact, there exists a
constant $B$ such that any point in $M$ has at most $B$ preimages under
$p_\Gamma$ in $\tilde K_{C_1+1}$. Each lift $\tilde{\wp}_0$ of $\wp_0$
intersecting $T^1 \tilde K_{C_1}$ spends a time at least $1$ in $\tilde
K_{C_1 + 1}$. Therefore,
\begin{equation*}
  n_{\tilde K_{C_1}}(\wp_0) \leq \Leb(p_\Gamma^{-1}(\wp_0) \cap \tilde K_{C_1 + 1}).
\end{equation*}
By the choice of $B$, this is bounded by $B \Leb(\wp_0) = B \ell(\wp_0)$,
proving~\eqref{lm:nW}.
\end{proof}


\section{Thermodynamical formalism}\label{sec:trois}

Entropy is a well-known measure of the exponential rate of complexity of a
dynamical system, and the measure  of maximal entropy is an important tool
in the ergodic study of hyperbolic dynamical systems.

Pressure is a weighted version of entropy, which is particularly useful for
the study of perturbations of hyperbolic systems. The notion of {\em
equilibrium state} is the weighted analogue of the measure of maximal
entropy.

In this section, for the geodesic flow of noncompact negatively curved
manifolds, we recall some well known notions and facts from~\cite{PPS}
and~\cite{PS16} on the pressure and  the construction of the {\em equilibrium
state} or {\em Gibbs measure} associated with a Hölder-continuous map
$F:T^1M\to \bbR$. This construction has a long story, initiated by the works
of Patterson~\cite{Patterson} and Sullivan~\cite{Sull} when $F=0$, by
Hamenstädt~\cite{hamenstadt} and Ledrappier~\cite{Ledrappier}. We refer
to~\cite{PPS} for detailed historical background and proofs of the assertions
in this paragraph. We follow here mainly~\cite[Chap 3.]{PPS}
and~\cite{schapira2004}, and~\cite{PS16}.

\subsection{Pressures of Hölder-continuous potentials}

\label{sssec:Pressure}


Let $F:T^1M\to \bbR$ be a Hölder-continuous map in the following sense: there
exist  $0<\beta\le 1$ and $C>0$ such that for all $v, w\in T^1M$ with $d(v,w)
\leq 1$, we have
\[
\abs*{ F(v) - F(w) } \leq C d(v,w)^\beta.
\]
Such a map $F$ will be said $(\beta, C)$-Hölder-continuous.  Let $\tilde F = F \circ
p_\Gamma$ be the $\Gamma$-invariant lift of $F$ to $T^1\tilde M$.

For $x\neq y\in \tilde M$, recall the notation
\[
 \int_x^y \tilde F\coloneqq \int_0^{d(x,y)}\tilde F(g^t v_{x,y})\,dt \,.
\]

Lemma 3.2 of~\cite{PPS} and the remark (ii) page 34 which follows this lemma
easily imply the following statement.
\begin{lemm}\label{lm:hold-potential}  Let $F:T^1M\to \bbR$ be a $(\beta, C_F)$-Hölder-continuous map
on $T^1M$, and $\tilde F$ its $\Gamma$-invariant lift. There exists a
constant $c_1>0$ depending only on the upper bound of the curvature and the
Hölder constants $\beta,C_F$, with the following property. Let $D\geq 1$, and
consider points $x,y,x',y'\in \tilde M$ with $d(x,y)\leq D$ and $d(x',y')
\leq D$. Then
\[
\abs*{\int_x^{x'} \tilde F-\int_y^{y'}\tilde F}\le
  c_1 e^D + D(\abs{\tilde F(v_{xx'})} + \abs{\tilde F(g^{d(x,x')}v_{xx'})})\,,
\]
where $v_{xx'}$ is the tangent vector at $x$ to the geodesic segment from $x$
to $x'$.

This bound applies in particular when $x$ and $y$ are picked in a compact
subset $\tilde K$ of $T^1 M$ with diameter at most $D$, and $x'$ and $y'$ are
picked in $\gamma \tilde K$ for some $\gamma \in \Gamma$. In this situation,
one gets an upper bound $c_1 e^D + 2D \max_{v \in T^1 \tilde K} \abs{F(v)}$
which only depends on $\tilde K$.
\end{lemm}
\begin{proof}
By Lemma 3.2 of~\cite{PPS} and the remark (ii) page 34 which follows this
lemma, we have
\begin{equation}
\label{eq:kjvcmlkxwjvm}
  \abs*{\int_x^{x'} \tilde F-\int_y^{y'}\tilde F}\le
  c_1 e^D + D\max_{\pi^{-1} (B(x,D))} \abs{\tilde F} + D\max_{\pi^{-1} (B(x',D))} \abs{\tilde F}\,,
\end{equation}
for some constant $c_1$. Moreover, on the ball $\pi^{-1} (B(x,D))$ one has
the inequality $\abs{\tilde F(v) - \tilde F(v_{xx'})} \leq C(\tilde F) D$ as
$\tilde F$ is Hölder-continuous and therefore Lipschitz on large scales. One
can therefore bound $D\max_{\pi^{-1} (B(x,D))} \abs{\tilde F}$ with
$D\abs{\tilde F(v_{xx'})} + C(\tilde F) D^2$, and then bound the second term
with $C' e^D$. The last term in~\eqref{eq:kjvcmlkxwjvm} is handled similarly.
\end{proof}



There are several natural definitions of pressure, that all coincide, as
proven in~\cite[Theorems~4.7 and~6.1]{PPS}, see
Theorem~\ref{th:Variationnel}. We recall here these three definitions.

\subsubsection{Geometric pressure as a critical exponent}

Recall that some point $o\in\tilde M$ has been chosen once and for all. The
Poincaré series associated with $(\Gamma,F)$ is defined by
\[
P_{\Gamma,o,F}(s)=  \sum_{\gamma\in \Gamma }e^{-sd(o,\gamma o)+\int_o^{\gamma o} \tilde F} \,.
\]
The following lemma is elementary, see for instance~\cite[p.\ 34-35]{PPS}.
\begin{lemm}[Geometric pressure]\label{lem:geom-pressure} The above series admits a {\em critical exponent}
$\delta_\Gamma(F)\in \bbR\cup \{+\infty\}$ defined by the fact that for all
$s>\delta_\Gamma(F)$ (resp.\ $s<\delta_\Gamma(F)$), the series
$P_{\Gamma,o,F}(s)$ converges (resp.\ diverges).
 Moreover, $\delta_\Gamma(F)$ does not depend on the choice of $o$ and satisfies for any $c>0$,
\[
\delta_\Gamma(F)=\limsup_{T\to +\infty}\frac{1}{T}\log\sum_{\gamma\in \Gamma, T-c\le d(o,\gamma o)\le T } e^{\int_o^{\gamma o} \tilde F}\,.
\]
We   call $\delta_\Gamma(F)$ the \emph{critical exponent} of $(\Gamma, F)$ or
the \emph{geometric pressure} of $F$.
\end{lemm}

As $\Gamma$ is nonelementary, one can show (see~\cite[Lemma 3.3]{PPS}) that
$\delta_{\Gamma}(F)>-\infty$. Moreover, observe that $\delta_\Gamma(F)$ is
finite as soon as $F$ is bounded from above. In~\cite[Thm 4.3]{PPS}, it has
been shown that the above limsup is in fact a true limit if $c$ is large
enough. In what follows, we will never require $F$ to be bounded from above, but
we will sometimes assume that $\delta_\Gamma(F)$ is finite.

\medskip

\subsubsection{Variational pressure }\label{ssec:ErgoPressure}

Let $\calM_1$ be the set of Borel probability measures on $T^1M$ invariant
under the geodesic flow, and $\calM_{1, \mathrm{erg}}$ the subset of ergodic
probability measures. For a given Hölder-continuous potential $F:T^1M\to
\bbR$, consider their subsets $\mathcal{M}_1^F$ and
$\mathcal{M}_{1,\mathrm{erg}}^F$ of probability measures with $\int
F^-\dd\mu<\infty$, where $F^-=-\inf(F,0)$ is the negative part of $F$. Given
$\mu\in\calM_1$,  we denote by $h_{KS}(\mu)=h_{KS}(g^1,\mu)$ its
\emph{Kolmogorov-Sinai entropy}, or \emph{measure-theoretic entropy} with
respect to $g^1$ (see  the appendix for the definition).

\begin{defi}\label{def:pressure}
The \emph{variational pressure} of $F$ is defined by
\[
\Pvar(F)=\sup_{\mu\in\calM^F_1} \left( h_{KS}(\mu)+\int F\dd\mu\,\right)\,
=
 \sup_{\mu\in\calM^F_{1, \mathrm{erg}}} \left( h_{KS}(\mu)+\int F\dd\mu\,\right)\,.
\]
\end{defi}


\subsubsection{Growth of periodic geodesics and \texorpdfstring{Gurevič}{Gurevic} pressure}

\label{subsubsec:Gurevic}

We denote by $\mathcal{P}$ (resp.\ $\mathcal{P}'$) the set of periodic
(resp.\ primitive periodic) orbits of the geodesic flow. Let now $K$ be a
compact subset of $M$ whose interior intersects at least a closed geodesic,
and $c>0$ be fixed. Let us denote by $\calP_K$ (resp.\ $\calP_K(t)$,
$\calP_{K}(t-c,t)$) the set of periodic orbits $p\in\mathcal{P}$ of the
geodesic flow whose projection $\pi(p)$ on $M$ intersects $K$ (resp.\ such
that $\ell(p)\le t$, \ $\ell(p)\in (t-c, t]$). The subsets $\calP'_K$,
$\calP'_K(t)$, $\calP_K'(t-c,t)$ of $\calP'$ are defined similarly.

Denote by $\int_pF$ the integral of $F$ over $(g^tv_p)_{0\le t\le \ell(p)}$ for any $v_p$ on $p$. By~\cite[Thm 4.7]{PPS}, the definition below makes sense.
\begin{defi}[Gurevič pressure]
For any compact subset $K$ of $M$ whose interior intersects a closed geodesic
and any $c>0$, the \emph{Gurevič pressure} of $F$ is defined by
\[
\PGur(F)
=
 \limsup_{T \to +\infty} \frac 1 T \log \sum_{p \in \calP_K(T-c,T)} e^{\int_p F}\,.
\]
It does not depend on $K$ nor $c$. Moreover, when $\PGur(F)>0$, then
\[
\PGur(F) = \limsup_{T \to +\infty} \frac 1 T \log \sum_{p \in \calP_K( T)} e^{\int_p F}\,.
\]
\end{defi}
 Gurevič was the first to introduce this definition (for the potential $F = 0$) in the context of
symbolic dynamics, see~\cite{Gurevic}. The equality $\PGur(F)=\Pvar(F)$ has
been proven in~\cite{Bowen} for compact manifolds and $F=0$, in~\cite{BR} for
compact manifolds and Hölder-continuous potentials. The equality
$\delta_\Gamma(F)=\PGur(F)$ is due to Ledrappier~\cite{Ledrappier} in the
compact case.

In the noncompact case, when $F\equiv 0$, Sullivan~\cite{Sull84} and
Otal-Peigné~\cite{OP} proved that $\delta_\Gamma=\Pvar$, and
Roblin~\cite{Roblin} proved that $\PGur=\delta_\Gamma$. The equality between
the three notions of pressures for general Hölder-continuous potentials on
noncompact manifolds is done in~\cite[Thm.~4.7 and Thm.~6.1]{PPS}.


\subsection{Patterson-Sullivan-Gibbs construction}\label{ssec:PattSull}

Let $F:T^1M\to\bbR$ be a Hölder-continuous potential with finite topological pressure. As will be seen in
Paragraph~\ref{sec:Gibbs}, the construction of a good invariant measure
associated with $F$ will use the product structure $\Omega\simeq
 ((\Lambda_\Gamma^2\backslash {\rm Diag})\times \bbR)/\Gamma$. The main step
is the definition of a good measure  $\nu^F$ on $\Lambda_\Gamma$, that we will call a {\em
Patterson-Sullivan-Gibbs measure}. We recall it below with more care
than usually done, because we will need in Section~\ref{sec:SPR-implique-PR}
to deal with technical points of the construction.

As stated in Lemma~\ref{lem:geom-pressure}, the Poincaré series $P_{\Gamma,o,F}(s)$ converges when $s>\delta_{\Gamma}(F)$ and diverges when $s<\delta_\Gamma(F)$.
We say that $(\Gamma, F)$ is {\em divergent} if this series diverges at
$s=\delta_{\Gamma}(F)$, and {\em convergent} if the series converges.

Following the famous Patterson trick, see~\cite{Patterson}, when $(\Gamma,F)$
is convergent, we  choose a positive nondecreasing map $h : \bbR^+ \to
\bbR^+$ with subexponential growth such that for all $\eta>0$, there exists
$C_\eta\ge 1$ such that
\begin{equation}
\label{eq:patterson_strong}
\forall r\geq 0, \quad \forall t\geq 0, \quad h(t+r)\leq C_\eta e^{\eta t}h(r)\,,
\end{equation}
and the series
\[
 \tilde P_{\Gamma,F}(o,s)
=
 \sum_{\gamma\in \Gamma}
h(d(o,\gamma o))\,e^{-sd(o,\gamma o)+\int_o^{\gamma o}\tilde F}
\] has the
same critical exponent $\delta_\Gamma(F)$, but diverges at the critical
exponent $\delta_\Gamma(F)$. The article~\cite{Patterson} provides the
construction of such a nondecreasing function $h$, except
that~\eqref{eq:patterson_strong} is replaced by the following property: for all $\eta>0$, there
exists $r_\eta>0$ such that
\begin{equation*}
\forall r\geq r_\eta, \quad \forall t\geq 0, \quad h(t+r)\leq e^{\eta t}h(r)\,.
\end{equation*}
We claim that this property implies~\eqref{eq:patterson_strong}. Indeed, the
result is obvious for $r\geq r_\eta$, while for $r \leq r_\eta$ one may write
\begin{equation*}
h(t+r) \leq h(t+r_\eta) \leq e^{\eta t} h(r_\eta) = \frac{h(r_\eta)}{h(0)}e^{\eta t} h(0)
\leq \frac{h(r_\eta)}{h(0)}e^{\eta t} h(r).
\end{equation*}
Hence,~\eqref{eq:patterson_strong} follows with $C_\eta = h(r_\eta)/h(0)$.

Define now for all  $s>\delta_\Gamma(F)$ a probability measure on $\tilde
M\cup\partial\tilde M$ by
\begin{equation*}
\nu^{F,s} = \frac{1}{\tilde P_{\Gamma,F}(o,s)} \sum_{\gamma\in \Gamma}h(d(o,\gamma o))e^{-sd(o,\gamma o)+\int_o^{\gamma o} \tilde F}\mathscr D_{\gamma o}\,,
\end{equation*}
where $\mathscr D_x$ stands for the Dirac mass at $x$.

 By compactness of $\tilde M\cup\partial\tilde M$, we can choose a decreasing sequence
$s_k\to \delta_{\Gamma}(F)$ such that $\nu^{F,s_k}$ converges to a
probability measure $\nu^F$. As $\tilde P_{\Gamma,o,F}$ diverges at
$s=\delta_{\Gamma,F}$, we deduce that $\nu^F$ is supported on
$\Lambda_\Gamma\subset\partial \tilde M$.

For all $x,y\in\tilde M$ and $\xi\in\partial\tilde M$, recall the following
notation from~\cite[sec 2.2.1]{schapira2004} (with an opposite sign
convention compared to~\cite{PPS})
\[
\rho_\xi^F(x,y)=\lim_{z\in[x,\xi),\,z\to\xi}\int_x^z\tilde F-\int_y^z\tilde F\,.
\]
Observe that $\rho_\xi^0=0$ and more generally, when $F\equiv c$ is constant,
$\rho^c=c\times \beta $, where $\beta$ is the usual Busemann cocycle defined
in Equation~\eqref{eq:Busemann}.

The measure $\nu^F$ satisfies the following crucial property. For all
$\gamma\in \Gamma$, and $\nu^F$-almost all $\xi\in\partial \tilde M$,
\begin{equation}\label{eq:GammaInvPS}
\frac{d\gamma_*\nu^F}{\dd\nu^F}(\xi)=e^{\delta_\Gamma(F) \beta_\xi(o,\gamma o)-\rho_\xi^F(o,\gamma o)}\,.
\end{equation}

%


As a consequence of~\eqref{eq:GammaInvPS}, one gets the following key
property, proved in~\cite{mohsen}. Recall that for a given set
$A\subset\tilde M$, the {\em shadow} $\mathcal{O}_x(A)$ of $A$ viewed from a
point $x\in\tilde M$ is by definition the set of points $y\in\tilde
M\cup\partial\tilde M$ such that the geodesic interval $[x,y]$ intersects the
set $A$.

\begin{prop}[Shadow Lemma] There exists $R_0>0$ such that for every given $R\ge R_0$,
there exists a constant $C>0$ such that for all $\gamma\in\Gamma$,
\[
\frac{1}{C}\,e^{-\delta_\Gamma(F) d(o,\gamma o)+\int_o^{\gamma o} \tilde F}\le
\nu^F(\mathcal{O}_o(B(\gamma o,R))\le C e^{-\delta_\Gamma(F) d(o,\gamma o)+\int_o^{\gamma o} \tilde F}\,.
\]
\end{prop}

Observe that the probability measure $\nu^F$ constructed above is not unique a priori,
but it will be unique in all interesting cases, see Section~\ref{ssec:SPR}
for details.

In fact, we will need a shadow lemma for the family of measures $\nu^{F,s}$,
for $s>\delta_\Gamma(F)$. As the uniformity of the constants in the
statements with respect to $s>\delta_\Gamma(F)$ will be crucial, we provide a
detailed proof.

For $A,B\subset\tilde M$ two sets, we will use the {\em enlarged shadow}
$\mathcal{O}_B(A)=\bigcup_{x\in B} \mathcal{O}_x(A)$, i.e.,  the set of
points $y\in\tilde M\cup\partial\tilde M$ such that there exists some $x\in
B$ such that the geodesic interval $[x,y]$ intersects~$A$.

\begin{lemm}[Orbital Shadow Lemma]\label{lem:orbital-shadow-lemma}
For every compact subset $\tilde K$ of $\tilde M$, there exist $r>0$ and
$\tau>0$ with the following property.
\begin{enumerate}
\item \emph{upper bound:} For every $\eta>0$, there exists $c_\eta>0$ such
    that for all $\delta_\Gamma(F)<s\le \delta_\Gamma(F)+\tau$ and
    $\gamma\in\Gamma$ with $d(o,\gamma o)\ge r$, we have
\[
\nu^{F,s}(\mathcal{O}_{\tilde K}(\gamma \tilde K))
\le
 c_\eta e^{-(s-\eta)d(o,\gamma o)+\int_o^{\gamma o}\tilde F}.
\]
\item \emph{lower bound:} Assume additionally that $\tilde K$ contains
    $B(o, R_1)$, where $R_1=R_1(F)$ is a fixed large constant. Then there
    exists $C$   such that for all $\delta_\Gamma(F)<s\le
    \delta_\Gamma(F)+\tau$ and $\gamma\in\Gamma$ with $d(o,\gamma o)\ge r$,
    we have
\[
\frac{1}{ C} e^{-sd(o,\gamma o)+\int_o^{\gamma o}\tilde F}
\le
\nu^{F,s}(\mathcal{O}_o(\gamma \tilde K)).
\]
\end{enumerate}
\end{lemm}

\begin{proof}
By convexity of the distance in nonpositive curvature, if
$D=\textrm{diam}(\tilde K)+d(o, \tilde K)$  and $\tilde L$ is the $D$-neighborhood of $\tilde K$, then for all
$\gamma\in\Gamma$, we have
\[
\mathcal{O}_{\tilde K}(\gamma\tilde K)\subset \mathcal{O}_o(\gamma\tilde L)\,.
\]
Therefore, upon replacing $\tilde K$ with $\tilde L$ in the upper bound, it
suffices to prove it for the shadow $\mathcal{O}_o(\gamma\tilde K)$. This also shows that without loss of generality, we can assume that $o\in\widetilde K$.

We follow the classical proof of the Shadow lemma, with $\nu^{F,s}$ on
$\tilde M$ instead of $\nu^F$ on $\partial \tilde M$. By definition, for all
$y\in \Gamma o$ and $\alpha\in\Gamma$, we have
\[
\frac{d(\alpha_*\nu^{F,s})}{\dd\nu^{F,s}}(y)=\frac{h(d(\alpha o,y))}{h(d(o,y))}e^{-s(d(\alpha o,y)-d(o,y))+\int_{\alpha o}^y\tilde F-\int_o^y\tilde F}\,.
\]
We deduce that
\begin{align*}
  \nu^{F,s}\left(\mathcal{O}_o(\gamma\tilde K)\right)
  & = \gamma^{-1}_*\nu^{F,s}(\mathcal{O}_{\gamma^{-1}o}(\tilde K))
  \\&
  = \int_{\mathcal{O}_{\gamma^{-1}o}(\tilde K)} \frac{h(d(\gamma^{-1}o,y))}{h(d(o,y))}
    e^{-s(d(\gamma^{-1}o,y)-d(o,y))+\int_{\gamma^{-1}o}^y\tilde F-\int_o^y\tilde F}\dd\nu^{F,s}\,.
\end{align*}

The triangular inequality gives $d(\gamma^{-1}o,y)\le
d(\gamma^{-1}o,o)+d(o,y)$. Moreover, since $o\in \tilde K$ and
$y\in\mathcal{O}_{\gamma^{-1}o}(\tilde K)$, by
Lemma~\ref{lm:NegCurvTriangle}, we have $d(\gamma^{-1}o,y)\ge
d(\gamma^{-1}o,o)+d(o,y)-2D$. In particular, with $r=2D$, if $d(\gamma^{-1}o, o)\ge r$, we get $d(\gamma^{-1}o,y) \geq d(o, y)$.

By  construction, the map $h$ is nondecreasing and for all $\eta>0$, there
exists $C_\eta>0$ such that for $\rho\ge 0$, $t\ge 0$, $h(t+\rho)\le C_\eta
e^{\eta t}h(\rho)$. Thus, independently of $s>\delta_{\Gamma}(F)$, we have
\[
1\le \frac{h(d(\gamma^{-1}o,y))}{h(d(o,y))}\le \frac{h(d(\gamma^{-1}o,o)+d(o,y))}{h(d(o,y))}
\le C_\eta e^{\eta d(\gamma^{-1}o,o)}\,.
\]

By Lemma~\ref{lm:hold-potential}, there exists a positive constant
$C(F,\tilde K)$, such that uniformly  in
$y\in\mathcal{O}_{\gamma^{-1}o}(\tilde K)$, we have $\displaystyle
\abs*{\int^y_{\gamma^{-1}o}\tilde F-\int_o^y\tilde
F-\int_{\gamma^{-1}o}^o\tilde F}\le C(F,\tilde K)\,.$ We deduce that, when $s
< \delta_\Gamma(F)+1$,
\begin{align*}
  \nu^{F,s}\left(\mathcal{O}_o(\gamma\tilde K)\right)&
  \le C_\eta e^{2Ds+C(F,\tilde K)} e^{-(s-\eta)d(\gamma^{-1}o,o)+\int_{\gamma^{-1}o}^o\tilde F} \times \nu^{F,s}(\mathcal{O}_{\gamma^{-1}o}(\tilde K))
  \\& \le C_\eta e^{2D(\delta_\Gamma(F)+1)+C(F,\tilde K)} e^{-(s-\eta)d(\gamma^{-1}o,o)+\int_{\gamma^{-1}o}^o\tilde F}\,.
\end{align*}
This concludes the proof of the upper bound.

For the lower bound, we have
\[
\nu^{F,s}\left(\mathcal{O}_o(\gamma\tilde K)\right)\ge e^{-C(F,\tilde K)} e^{-sd(\gamma^{-1}o,o)+\int_{\gamma^{-1}o}^o\tilde F} \times \nu^{F,s}\left(\mathcal{O}_{\gamma^{-1}o}(\tilde K )\right)\,.
\]
The crucial point is to get a lower bound of the measure on the right hand
side. More precisely, as we assume that $\tilde K$ contains a ball centered at $o$ with large radius,
we wish to find such a radius $R>0$ and $\tau>0$ such that uniformly in $\gamma\in\Gamma$
and    $\delta_\Gamma(F)<s<\delta_\Gamma(F)+\tau$, the measure
$\nu^{F,s}\left(\mathcal{O}_{\gamma^{-1}o}(B(o,R))\right)$ has a  positive
lower bound. It would follow immediately if we knew that for some $R>0$,
uniformly in $y\in\tilde M$ and $\delta_\Gamma(F)<s<\delta_\Gamma(F)+\tau$,
the measure $\nu^{F,s}\left(\mathcal{O}_{y}(B(o,R))\right)$ has a  positive
lower bound.
We follow the usual argument which concludes the proof of the classical
Shadow Lemma.  Imagine by contradiction that
there exist $s_n\to \delta_\Gamma(F)$,
$R_n\to \infty$ and $y_n\to y_\infty\in \tilde M\cup \partial \tilde M$ such
that $\nu^{F,s_n}\left(\mathcal{O}_{y_n}(\tilde B(o,R_n))\right)\to 0$.


There exists a subsequence $s_{n_k}$ such that $\nu^{F,s_{n_k}}$ converges to
some probability measure $\nu'$ on the boundary which is supported on the full limit set
$\Lambda_\Gamma$. This measure is not a single Dirac mass at $y_\infty$, by
non-elementarity. By regularity, we can find an open neighborhood $U$ of
$y_\infty$ with $\nu'(U)=1-\alpha < 1$. Since $\nu^{F,s_{n_k}}$ converges
weakly to $\nu'$ and $U$ is open, this entails $\nu^{F,s_{n_k}}(U) \leq 1
-\alpha/2$ for large enough $k$. For large enough $n$, the complement of $U$
is contained in $\mathcal{O}_{y_n}(\tilde B(o,R_n))$ as $y_n\to y_\infty$ and
$R_n\to \infty$. This gives $\nu^{F,s_{n_k}}(\mathcal{O}_{y_{n_k}}(\tilde
B(o,R_{n_k}))) \geq \alpha/2$, a contradiction.
\end{proof}



\subsection{Gibbs measures}\label{sec:Gibbs}

Let $F: T^1M\to \bbR$ be a Hölder-continuous potential with finite topological pressure,
and let $\nu^F$ be a Patterson-Sullivan measure associated with $F$, as
constructed in the previous paragraph.

Denote by $\iota:T^1M\to T^1M$ the involution $v\mapsto -v$, and let
$\nu^{F\circ \iota}$ be a Patterson-Sullivan measure associated with $F\circ
\iota$. Hopf coordinates allow us to define a Radon measure on $T^1\tilde M$
by the formula

\begin{equation}\label{Gibbs-product}
\dd\tilde m^F(v)   =
 e^{\delta_\Gamma(F)\beta_{v^-}(o,\pi(v))-\rho^{F\circ\iota}_{v^-}(o,\pi(v))+\delta_\Gamma(F)\beta_{v^+}(o,\pi(v))-\rho^F_{v^+}(o,\pi(v))} \,
\dd\nu^{F\circ\iota}(v_-) \dd\nu^F(v_+)\dd t\,.
\end{equation}
 By construction, $\tilde m^F$ is invariant under the geodesic flow and it follows from~\eqref{eq:GammaInvPS} that
it is invariant under the action  of $\Gamma$  on $T^1\tilde M$, so that it
induces a Radon measure $m^F$ on $T^1M$.

The following crucial result  was shown in~\cite{OP} for $F=0$ and
in~\cite[Chap.\ 6]{PPS} in general.

\begin{theo}[\cite{OP}--\cite{PPS}] \label{theo:Gibbs}
Let $F: T^1M\to \bbR$ be a Hölder-continuous potential with finite
topological pressure. Then the following alternative holds. If a measure
$m^F$ on $T^1M$ given by the Patterson-Sullivan-Gibbs construction is finite
and if, once normalized into a probability measure, it belongs to
$\calM_1^F$, then it is the unique probability measure realizing the supremum
in the variational principle:
\[
P(F)=\sup_{m\in\mathcal{M}^F_{1}}\left(\, h_{KS}(m)+\int_{T^1M} F\dd m\,\right)\,= h_{KS}\left(\frac{m^F}{\norm{m^F}}\right)+\int_{T^1M}F\frac{\dd m^F}{\norm{m^F}}\,.
\]
Otherwise, there is no probability measure
realizing this supremum.
\end{theo}

We will also need the following result, called {\em
Hopf-Tsuji-Sullivan-Roblin} Theorem,  see~\cite[Theorem 5.3]{PPS} for a more
complete statement and a proof.

\begin{theo}[Hopf-Tsuji-Sullivan-Roblin theorem,~\cite{PPS}]\label{theo:HTS}
Let $F: T^1M\to \bbR$ be a Hölder-continuous potential with finite
topological pressure, and let $\nu^F$ and $m^F$ be associated with $F$ as
above. The following assertions are equivalent.
\begin{enumerate}
\item The pair $(\Gamma,F)$ is divergent, i.e., the Poincaré series
    $P_{\Gamma,o,F}(s)$ diverges at the critical exponent
    $\delta_\Gamma(F)$;
\item the measure $\nu^F$ gives positive measure to the radial limit set:
    $\nu^F(\Lambda_{\Gamma}^{\mathrm{rad}})>0$;
\item the measure $\nu^F$ gives full measure to the radial limit set:
    $\nu^F(\Lambda_{\Gamma}^{\mathrm{rad}})=1$;
\item the measure $m^F$ is conservative for the action of the geodesic flow
    on $T^1M$;
\item the measure $m^F$ is ergodic and conservative for the action of the
    geodesic flow on $T^1M$.
\end{enumerate}
\end{theo}

Together with the above Hopf-Tsuji-Sullivan-Roblin Theorem, the Poincaré recurrence Theorem implies the following crucial observation:\\
\centerline{ When the measure $m^F$ is finite, it is ergodic and
conservative.}







\section{Pressures at infinity}\label{sec:pressures-at-infinity}

In this section, we recall first the notion of {\em fundamental group outside
a compact set} introduced in~\cite{PS16}. Then, to each of the three  notions
of pressures recalled in section~\ref{sssec:Pressure}, we  associate  a
natural notion of \emph{pressure at infinity}.


\subsection{Fundamental group outside a given compact set}

For any compact set $\tilde K\subset\tilde M$,  as in~\cite{PS16, ST19,CDST}
we define the {\em fundamental group outside $\tilde K$}, denoted by
$\Gamma_{\tilde K}$, as
\[
\Gamma_{\tilde K}=\left\{\gamma\in \Gamma, \exists x,y\in \tilde K,\;   [x,\gamma y]\cap \Gamma {\tilde K}\subset {\tilde K}\cup \gamma {\tilde K} \right\}\,.
\]
Considering the last point on such a geodesic segment in $\tilde K$, and the
first point in $\gamma {\tilde K}$, it follows that this set can equivalently
be written as
\[
\Gamma_{\tilde K}=\left\{\gamma\in \Gamma, \exists x,y\in \tilde K,\;   [x,\gamma y]\cap \Gamma {\tilde K}=\{x, \gamma y\} \right\}\,.
\]
This subset of $\Gamma$ corresponds to long excursions of geodesics outside
of $K\coloneqq p_\Gamma(\widetilde K)$. We stress that this is not a subgroup
in general, see examples in~\cite[Section 7]{ST19}.

Recall from~\cite[Prop.\ 7.9]{ST19} and~\cite[Prop.\ 7.7]{ST19} the following
results.

\begin{prop}\label{prop:comparison-fund-group-outside-compacts}
\begin{enumerate}
\item Let $\tilde K\subset \tilde M$ be a compact subset, and
    $\alpha\in\Gamma$. Then $\Gamma_{\alpha\tilde K}=\alpha\Gamma_{\tilde
    K}\alpha^{-1}$.
\item If $\tilde K_1$ and $\tilde K_2$ are compact subsets of $\tilde M$
    such that $\tilde K_1$ is included in the interior of $\tilde K_2$,
    then there exist finitely many $\alpha_1, \dotsc,\alpha_k\in\Gamma$
such that $\ds \Gamma_{\tilde K_2}\subset
\bigcup_{i,j=1}^k\alpha_i\Gamma_{\tilde K_1}\alpha_j^{-1}$.
\end{enumerate}
\end{prop}

In some circumstances, it may be useful to consider different Riemannian
structures $(M,g_0)$
 and $(M,g)$ on the same orbifold, and compare
their fundamental groups outside a given compact set, denoted by
$\Gamma^{g_0}_{\tilde K}$ and $\Gamma^g_{\tilde K}$ in order to avoid confusions. The
following proposition follows from the definition.

\begin{prop}\label{prop:ST19-7.7}
Let $\tilde K \subset \tilde M$ be a compact subset. Let
$g_0$ and $g$ be two complete Riemannian metrics with pinched negative curvature and bounded derivatives of the curvature  that coincide outside
$p_\Gamma(\tilde K )$. Then
\[
\Gamma_{\tilde K}^g=\Gamma_{\tilde K}^{g_0}\,.
\]
\end{prop}

\subsection{Critical exponent at infinity}

Consider the associated restricted Poincaré series
\[
P_{\Gamma_{\tilde K}}(s,F)=\sum_{\gamma\in \Gamma_{\tilde K}}e^{-sd(o,\gamma o)+\int_o^{\gamma o} \tilde F}   \,.
\]
Its critical exponent  $\delta_{\Gamma_{\tilde K}}(F)\in [-\infty,+\infty]$, satisfies
for all $c>0$
\[
\delta_{\Gamma_{\tilde K}}(F)=
\limsup_{t\to +\infty} \frac{1}{t}\log \sum_{\gamma\in \Gamma_{\tilde K}, t-c\le d(o,\gamma o)\le t}e^{\int_o^{\gamma o} \tilde F}\,.
\] We call it the \emph{critical exponent} or \emph{geometric pressure of $F$ outside $\tilde K$}. By construction,
\[
\delta_{\Gamma_{\tilde K}}(F) \leq \delta_\Gamma(F)\,.
\]

\begin{defi}
The \emph{critical exponent at infinity} or {\em geometric pressure at
infinity} of $F$ is defined as
\[
\delta_\Gamma^\infty(F)=\inf_{\tilde K} \delta_{\Gamma_{\tilde K}}(F)\,,
\]
where the infimum is taken over all compact sets ${\tilde K}\subset\tilde M$.
\end{defi}

An immediate corollary of
Proposition~\ref{prop:comparison-fund-group-outside-compacts} is the
following result.

\begin{coro}\label{coro:comparison-crit-expo-outside-compacts}
Let $F : T^1 M \to \R$ be a Hölder-continuous potential.
\begin{enumerate}
\item Let $\tilde K\subset \tilde M$ be a compact subset, and
    $\alpha\in\Gamma$. Then $\delta_{\Gamma_{\alpha\tilde
    K}}(F)=\delta_{\Gamma_{\tilde K}}(F)$.
\item If $\tilde K_1$ and $\tilde K_2$ are compact subsets of $\tilde M$
    such that $\tilde K_1$ is included in the interior of $\tilde K_2$,
    then
\[
\delta_{\Gamma_{\tilde K_2}}(F)\le\delta_{\Gamma_{\tilde K_1}}(F)\,.
\]
\end{enumerate}
\end{coro}

Corollary~\ref{coro:comparison-crit-expo-outside-compacts} implies  for any
Hölder-continuous potential $F$ the very convenient following fact:
\begin{equation*}
\delta_\Gamma^\infty(F) = \lim_{R\to +\infty} \delta_{\Gamma_{B(o, R)}}(F)\,.
\end{equation*}

It is worth noting that this critical exponent at infinity can be equal to
$-\infty$, in particular in the trivial situations described in the following
lemma, where {\em all} potentials have critical exponent at infinity equal to
$-\infty$.

\begin{lemm} Let $M$ be a compact or convex-cocompact Riemannian manifold with pinched negative curvature.
Then, for every  Hölder-continuous potential $F: T^1M\to \bbR$,
\[
\delta_\Gamma^\infty(F) = -\infty.
\]
\end{lemm}

\begin{proof} By~\cite[Prop.\ 7.17]{ST19}, for
${\tilde K}\subset {\tilde M}$ large enough, the set $\Gamma_{\tilde K}$ is
finite. It immediately implies
\begin{equation*}
\delta_{\Gamma}^\infty(F) \leq \delta_{\Gamma_{\tilde K}}(F) = -\infty\,.
\qedhere
\end{equation*}
\end{proof}

We refer to Corollary~\ref{coro:exposant-infini} for more interesting
situations where $\delta_\Gamma^\infty(0)\ge 0$ and there exists a
Hölder-continuous map $F:T^1M\to \bbR$ with
$\delta_\Gamma^{\infty}(F)=-\infty$.

\medskip
\subsection{Variational pressure at infinity}

Recall that the \emph{vague topology} on the space of Radon measures on $T^1
M$ is the weak-* topology on the space of Radon measures viewed as the dual
of the space $C_c(T^1M)$ of continuous maps with compact support on $T^1M$
(or equivalently of its closure $C_0(T^1M)$ with respect to the supremum
norm). A sequence of probability measures $(\mu_n)_{n\in \mathbb N}$
converges to $0$ for the vague topology if and only if for every map
$\varphi\in C_c(T^1M)$, it satisfies $\displaystyle \lim_{n\to +\infty}
\int\varphi\dd\mu_n =0$. We write this $\mu_n\overset{\ast}{\rightharpoonup}
0$. This provides the following other natural notion of pressure at infinity.

\begin{defi}
Let $F$ be a Hölder-continuous potential on $T^1M$. The \emph{variational
pressure at infinity} of $F$ is
\begin{align*}
\Pvar^{\infty}(F)  & =
\sup \left\{ \limsup_{n\to +\infty} \left( h_{KS}(\mu_n) + \int_{T^1M} F \dd\mu_n \right)\; ; \; (\mu_n)_{n\in \bbN} \in {(\mathcal{M}_{1}^F)}^\bbN \text{ s.t.\ } \mu_n \overset{\ast}{\rightharpoonup} 0\right\} \\
& =
\lim_{\epsilon \to 0} \inf_{ {K}\subset   M, K \text{ compact}}\sup
\left\{ h_{KS}(\mu) + \int_{T^1 M} F \dd\mu \; ; \; \mu\in \mathcal{M}_{  1}^F \text{  s.t.\ } \mu(T^1{K}) \leq \epsilon\right\}\\
& = \inf_{ {K} \subset   M, K \text{ compact}} \lim_{\epsilon \to 0}\,\, \sup
 \left\{ h_{KS}(\mu) + \int_{T^1 M} F \dd\mu \; ; \; \mu\in \mathcal{M}_{  1}^F \text{  s.t.\ } \mu(T^1{K}) \leq \epsilon\right\}.
\end{align*}
\end{defi}

Let us check that these three definitions coincide.
\begin{proof}
The limit in $\epsilon$ in the last two lines is a decreasing limit, i.e.,
an infimum, so it commutes with the infimum over compact subsets $K$. Hence, it suffices to
show that the quantity on the first line, say $A$, coincides with the
quantity on the second line, say $B$. If a sequence $\mu_n$ realizes the
supremum in $A$, then for any $\epsilon>0$ and for any compact subset $K$, one
has eventually $\mu_n(T^1 K) \leq \epsilon$ by definition of the vague
convergence to $0$. Therefore, $A \leq B$. Conversely, consider sequences
$\epsilon_n$ and $K_n$ realizing the infimum in $B$. Since decreasing
$\epsilon_n$ and increasing $K_n$ can only make the infimum smaller, it
follows that $\epsilon'_n = \min(\epsilon_n, 1/n)$ and $K'_n = K_n \cup B(o,
n)$ also realize the infimum in $B$. We get a sequence of measures $\mu_n\in
\mathcal{M}_{  1}^F$ with $\mu_n(T^1 K'_n)\leq \epsilon'_n$ and
$h_{KS}(\mu_n) + \int_{T^1 M} F \dd\mu_n \to B$. Since $T^1 K'_n$ increases
to cover the whole space and $\epsilon'_n$ tends to $0$, we have $\mu_n
\overset{\ast}{\rightharpoonup} 0$. Therefore, $B \leq A$.
\end{proof}

From a dynamical point of view, it would be more natural and apparently more
general to consider all compact subsets $\calK$ of $T^1M$, instead of
restricting to unit tangent bundles $\calK=T^1K$ of compact subsets of $M$.
However, the equality between the three above quantities shows that it would
not bring anything to the definition.

In the case $F\equiv 0$, in the context of symbolic dynamics, this definition
already appeared in different works, see for example~\cite{GS,
Ruette,BBG06,BBG}.

One can consider a variation around the above definition, requiring
additionally that all the measures $\mu_n$ are ergodic. We will denote this
pressure by $P_{\mathrm{var}, \mathrm{erg}}^{\infty}(F)$. We will see in
Corollary~\ref{cor:Perg_infty} that it coincides with $\Pvar^{\infty}(F)$, as
a byproduct of the proof of Theorem~\ref{th:AllPressionEquivalent}.


\medskip

\subsection{\texorpdfstring{Gurevič}{Gurevic} pressure at infinity}
To the Gurevič pressure is naturally associated a notion of {\em Gurevič
pressure at infinity}, when considering only periodic orbits that spend an
arbitrarily small proportion of their period in a given compact subset. This
only makes sense for compact subsets on $T^1M$   whose interior intersects
the non-wandering set $\Omega$. As in the preceding sections, we consider
only compact subsets $K$ of $M$, so that we require that the
interior of $K$, denoted by $\inter{K}$, intersects the projection
$\pi(\Omega)$ of the nonwandering set on $M$. We recall that $\calP_{K} (T-c,
T)$, defined in Subsection~\ref{subsubsec:Gurevic}, is the set of periodic
orbits intersecting $K$ with length in the interval $[T-c, T]$.

\begin{defi}
Let $F$ be a Hölder-continuous potential on $T^1M$. For any $c > 0$, the
\emph{Gurevič pressure at infinity} of $F$ is
\begin{align*}
\PGur^{\infty}(F) & =
\inf_{ \substack{K \subset  M, K \text{ compact}\\{\inter{K} \cap \pi(\Omega)\ne \emptyset }}}\,\, \lim_{\alpha\to 0}\,\,\limsup_{T\to +\infty}
\frac 1 T \log \sum_{p \in  \calP_{K} (T-c, T) \;; \; \ell(p\cap T^1{K} )<\alpha \ell(p)} e^{\int_p F}\\
 & =
\lim_{\alpha\to 0}\,\,  \inf_{ \substack{K \subset  M, K \text{ compact}\\{\inter{K} \cap \pi( \Omega) \ne \emptyset }}}\,\, \limsup_{T\to +\infty}
 \frac 1 T \log \sum_{p \in \calP_{K} (T-c, T) \;; \; \ell(p\cap T^1{K})<\alpha \ell(p)} e^{\int_p F}.
\end{align*}
It does not depend on $c$.
\end{defi}

It is not completely obvious from the definition what happens when one
increases a compact subset $K'$ to a larger compact subset $K$. Since one may consider
orbits that intersect $K$ but not $K'$, one is allowed more orbits. However,
the condition $\ell(p\cap T^1{K} )<\alpha \ell(p)$ becomes more restrictive
for $K$ than for $K'$, allowing less orbits. These two effects pull in
different directions. It turns out that the latter effect, allowing less
orbits, is stronger. We formulate this statement with a third compact subset
$K''$ as we will need it later on in this form, but for the previous
discussion you may take $K' = K''$.

\begin{prop}
\label{prop:gurevic_subset} Consider three compact subsets $K'', K', K$ of $M$ such
that the interior of $K''$ intersects a closed geodesic, and $K'$ is
contained in the interior of $K$. Then, for $\alpha>0$,
\begin{multline*}
  \limsup_{T\to +\infty}
  \frac 1 T \log \sum_{p \in  \calP_{K} (T-c, T) \;; \; \ell(p\cap T^1{K} )<\alpha \ell(p)} e^{\int_p F}
  \leq
  \limsup_{T\to +\infty}
\frac 1 T \log \sum_{p \in  \calP_{K''} (T-c, T) \;; \; \ell(p\cap T^1{K'} )<2\alpha \ell(p)} e^{\int_p F}.
\end{multline*}
\end{prop}
Therefore, the infimum in the definition of the Gurevič pressure may be
realized by taking an increasing sequence of balls, just like in
Corollary~\ref{coro:comparison-crit-expo-outside-compacts}:
\[
 P_{Gur}^\infty(F)=
\lim_{R\to\infty} \,\, \lim_{\alpha\to 0}\,\,\limsup_{T\to +\infty}
\frac 1 T \log \sum_{p \in  \calP_{B(o,R)} (T-c, T) \;; \; \ell(p\cap T^1{B(o,R)} )<\alpha \ell(p)} e^{\int_p F}\,.
\]

\begin{proof} Consider a periodic orbit $p$ of length $\ell(p)\in [T-c, T]$
starting from $v \in T^1K$, parameterized by $[0,\ell(p)]$. Fix  also
$\epsilon>0$. By the first assertion of Proposition~\ref{lem:connecting}
there is another periodic orbit $p'$, of length $\ell(p') \in [\ell(p),
\ell(p)+T_0]$ for a constant $T_0$ depending on $K$ and $K''$ and $\epsilon$,
parameterized by $[0,\ell(p')]$, following $p$ within $\epsilon$ during the
interval of time $[T_0,\ell(p)-T_0]$, and intersecting $T^1 K''$.
Lemma~\ref{lm:hold-potential} shows that there exists a constant $C'$ such
that $\abs{\int_p F - \int_{p'}F} \leq C'$. Moreover, by Assertion~2 of
Proposition~\ref{lem:connecting}, there exists $C''$ such that the
multiplicity of the map $p \mapsto p'$ is bounded by $C'' T$ if $T$ is large
enough.

If $\epsilon$ is such that the $\epsilon$-neighborhood of $K'$ is included in
$K$, then the times at which $p'$ belongs to $T^1 K'$ are of two kind: either
they are in $[T_0, \ell(p')-2T_0]$, and then the corresponding point on $p$
belongs to $T^1 K$, or they are not. Hence, $\ell(p' \cap T^1 K') \leq 3T_0 +
\ell(p \cap T^1 K)$. Taking into account the multiplicity, we obtain
\begin{equation*}
\sum_{p \in  \calP_{K} (T-c, T) \;; \; \ell(p\cap T^1{K} )<\alpha \ell(p)}
e^{\int_p F}
\leq C'' T\sum_{p' \in \calP_{K''} (T-c, T+T_0) \;; \; \ell(p'\cap T^1 K')<3T_0 + \alpha \ell(p')}
e^{C' + \int_{p'} F}\,.
\end{equation*}
When $T$ is large enough, we have $3T_0 + \alpha \ell(p') < 2 \alpha
\ell(p')$. As the interval $[T-c,T+T_0]$ is the union of at most
$\frac{T_0}{c}+2$ intervals of length at most $c$, taking a limsup, we obtain
\begin{multline*}
  \limsup_{T\to +\infty}
  \frac 1 T \log \sum_{p \in  \calP_{K} (T-c, T) \;; \; \ell(p\cap T^1{K} )<\alpha \ell(p)} e^{\int_p F}
  \\
  \leq
  \limsup_{T\to +\infty}
\frac 1 T \log \sum_{p' \in  \calP_{K''} (T-c, T) \;; \; \ell(p'\cap T^1{K'} )<2\alpha \ell(p')} e^{\int_{p'} F}\,.
\qedhere
\end{multline*}
\end{proof}

%


\subsection{Pressure at infinity is invariant under compact perturbations}

In this paragraph, we will show that the critical exponent at infinity is
invariant under any compact perturbation of the potential or of the underlying
metric.

\begin{prop}\label{prop:CompactPerturbPotential}
Let $F : T^1M \to \mathbb R$ be a  Hölder-continuous map with finite
pressure, let $A : T^1M\to \bbR$ be a Hölder-continuous map, and let $\tilde
K \subset \tilde M$ be a compact subset such that~$A$ vanishes outside
$p_\Gamma(T^1 \tilde K)$. Then
\[
\delta_{\Gamma_{\tilde K}}(F + A) = \delta_{\Gamma_{\tilde K}}(F).
\]
In particular,
\[
\delta_\Gamma^\infty(F+A) = \delta_\Gamma^\infty(F)\,.
\]
\end{prop}
\begin{proof}
By definition, for all $\gamma\in\Gamma_{\tilde K}$, there exist
$x,y\in\tilde K$ such that the geodesic segment $[x,\gamma y]$ satisfies
$[x,\gamma y]\cap \Gamma \tilde K=\{x,\gamma y\}$.
 We deduce that
\[
\int_x^{\gamma y} (\tilde F+A) =\int_x^{\gamma y}\tilde F \,.
\]
By Lemma~\ref{lm:hold-potential}, we deduce that
\[
\abs*{\int_o^{\gamma o} (\tilde F+A) -\int_o^{\gamma o}\tilde F}\le  2C(F,\tilde K,A)\,.
\]
By definition of $\delta_{\Gamma_{\tilde K}}(F)$ and $\delta_{\Gamma_{\tilde
K}}(F+A)$, the result follows immediately.
\end{proof}

In the next proposition, we  consider two negatively curved Riemannian
metrics $g_0$ and $g$ on $M$ such that there exists $C>0$ satisfying at every point of $M$
\begin{equation}\label{eq:qisom}
\frac 1 C g_0 \leq g \leq C g_0
\end{equation}
and still denote by $g_0$ and $g$ their lifts to $\tilde M$. For a given
potential $F:TM\to\bbR$, denote by $\delta_{\Gamma_{\tilde K},g_0}(F)$,
$\delta_{\Gamma_{\tilde K},g}(F)$, $\delta_{\Gamma,g_0}^\infty(F),
\delta_{\Gamma,g}^\infty(F)$ the associated critical exponents for the
restriction of $F$ to the unit tangent bundles for $g$ and $g_0$
respectively. It follows from~\eqref{eq:qisom} that being Hölder-continuous
does not depend on the metric one considers.

\begin{prop}\label{prop:CompactPerturbMetric}
Let $(M,g_0)$ be a Riemannian manifold with pinched negative curvature, and
$g$ be another negatively curved metric on $M$. Let $F : T M \to \mathbb R$
be a Hölder-continuous potential. Let $\tilde K\subset \tilde M$ be a compact
set such that $g$ and $g_0$ coincide outside of $p_\Gamma(\tilde K)$. Then
\[
\delta_{\Gamma_{\tilde K},g_0}(F) = \delta_{\Gamma_{\tilde K},g}(F).
\]
In particular, $\delta_{\Gamma,g_0}^\infty(F) = \delta_{\Gamma,
g}^\infty(F)$.
\end{prop}
\begin{proof} When necessary, denote by
$[a,b]^g$ or $[a,b]^{g_0}$ the geodesic segment of the metric $g$ (resp.\
$g_0$) between $a$ and $b$. By Proposition~\ref{prop:ST19-7.7}, we have
$\Gamma_{\tilde K}^{g_0}=\Gamma_{\tilde K}^{g}$. Let $\gamma\in\Gamma_{\tilde
K}$. There exist $x,y\in\tilde K$ such that $[x,\gamma y]^{g_0}\cap \Gamma
\tilde K = \{x, \gamma y \}$.

Outside $\Gamma \tilde K$, the metrics $g_0$ and $g$ coincide, so
that~\eqref{eq:qisom} is satisfied, the segments $[x,\gamma y]^g$ and
$[x,\gamma y]^{g_0}$ are the same, and the integrals of $F$ coincide:
$\int_{[x,\gamma y]^g}\tilde F=\int_{[x,\gamma y]^{g_0}}\tilde F$.

Moreover,  by compactness, there exists $D>0$ depending on $\tilde K$, $g_0$
and $g$, such that for both metrics, $d^{g_0}(x,o)\le D$, $d^{g}(x,o)\le D$,
$d^{g_0}(y,o)\le D$, and $d^g(y,o)\le D$. Therefore, using
Lemma~\ref{lm:hold-potential}, there exists a constant $C$ depending on $D$
and $\sup_{\tilde K}(\tilde F)$ such that for both metrics, we have
\[
\abs*{\int_{[o,\gamma o]^g}\tilde F-\int_{[x,\gamma y]^g}\tilde F}\le C\quad
\text{and}\quad \abs*{\int_{[o,\gamma o]^{g_0}}\tilde F-\int_{[x,\gamma y]^{g_0}}\tilde F}\le C\,.
\]
The result follows by definition of the geometric pressure outside
$\widetilde K$.
\end{proof}

Compact perturbations of a given potential do not change the critical
exponent at infinity, but modify the pressure, as shown in the next
proposition. This kind of statement is very useful and relatively classical.
Similar statements in symbolic dynamics or on geometrically finite manifolds,
or for potentials converging to $0$ at infinity can be found for example
in~\cite{IRV,Riquelme-Velozo}.

\begin{prop}\label{prop:PressureBump}
Let $F:T^1M\to\bbR$ be a Hölder-continuous potential, and $A : T^1M \to [0,
+\infty)$ a nonnegative Hölder-continuous map with compact support. The map
\[
\lambda\in\bbR\to \delta_\Gamma(F+\lambda A)
\]
is Lipschitz-continuous,  convex, nondecreasing, and as soon as the interior
of the support of $A$ intersects the nonwandering set $\Omega$, we have $\ds
\lim_{\lambda\to\infty}\delta_\Gamma(F+\lambda A)=+\infty $.
\end{prop}

\begin{proof} The fact that it is Lipschitz-continuous is an immediate consequence of the definition,
and that it is nondecreasing is obvious as $A\ge 0$. Convexity follows from
the variational principle (Theorem~\ref{th:Variationnel}) because it is a
supremum of affine maps.

Now, if the interior of  the support of  $A$ intersects $\Omega$, there will be at least an
invariant probability measure $\mu$ with compact support (supported by a
periodic orbit intersecting the interior of  the support of $A$ for example) such that $\int A \dd\mu>0$. By
the variational principle,
\[
\delta_\Gamma(F+\lambda A)\ge h_{KS}(\mu)+\int F\dd\mu +\lambda\int A \dd\mu\,,
\] and the latter quantity goes to $+\infty$ when $\lambda\to +\infty$. The result follows.
\end{proof}

The combination of Propositions~\ref{prop:CompactPerturbPotential}
and~\ref{prop:PressureBump} provides the following corollary, which will
become relevant in Section~\ref{sec:SPR}.

\begin{coro}\label{coro:existence-pot-SPR}
Let $F$ and $A : T^1M\to \bbR$ be two Hölder-continuous potentials. Assume
that $F$ has finite geometric pressure at infinity, and that $A$ is nonnegative,
compactly supported, and not everywhere zero on the non-wandering set. Then
for $\lambda>0$  large enough, we have
\[
\delta_\Gamma(F+\lambda A) > \delta_\Gamma^\infty(F + \lambda A).
\]
\end{coro}


\subsection{Infinite pressure}

In this paragraph, we prove that if the geometric pressure of a potential is infinite,
then its pressure at infinity is also infinite. This is not surprising:
everything coming from a compact set is finite, so if the pressure is
infinite the major contribution has to come from the complement of compact
sets, and therefore the pressure outside any compact set should also be
infinite. However, the proof is not completely trivial. It will involve
careful splittings of orbits and subadditivity, two themes that will also
show up in later proofs. One may think of this proof as a warm-up for the
next sections.

\begin{prop}
\label{prop:infinite_pressure} Let $F:T^1M\to\bbR$ be a Hölder-continuous
potential with $\delta_\Gamma(F) = +\infty$. Then $\delta^\infty_\Gamma(F) =
+\infty$.
\end{prop}
\begin{proof}
We will prove the contrapositive, namely, if there exists a compact set
$\tilde K$ of $\tilde M$ with $\delta_{\Gamma_{\tilde K}}(F) < \infty$ then
$\delta_\Gamma(F) < \infty$. Adding $o$ to $\tilde K$ if necessary, we can
assume $o\in \tilde K$. Fix some $s
> \delta_{\Gamma_{\tilde K}}(F)$. Let $D$ be the diameter of $\tilde
K$.

Let $u_n = \sum_{\gamma\in\Gamma : d(o, \gamma o) \in (n-1, n]} e^{\int_o^{\gamma o}
\tilde F}$. We claim that there exists $C>0$ such that, for all $n\in\mathbb{N}$,
\begin{equation}
\label{eq:un_le}
  u_n \leq C\sum_{\substack{a,b\in\mathbb{N}\\1 \leq a, b \leq n-1 \\ \abs{a+b-n} \leq C}} u_a u_b + C e^{sn}.
\end{equation}
The proof of this inequality is purely geometrical. On the other hand, the
proof that this inequality implies the proposition is purely analytical. We
postpone the geometrical proof of~\eqref{eq:un_le} and explain how to deduce
the result assuming this inequality, by a subadditivity argument. Extend
$u_n$ by $0$ for $n\in(-\infty, -1]$, and define a new sequence $v_n =
\sum_{n-C}^{n+C} u_i$. It satisfies the inequality
\begin{equation}
\label{eq:vn_le}
  v_n \leq C_1 \sum_{\substack{1\leq a', b'\leq n-1 \\ a'+b' = n}} v_{a'} v_{b'} + C_1 e^{sn},
\end{equation}
for some $C_1$. To get this inequality, bound each $u_i$ appearing in $v_n$
using~\eqref{eq:un_le}, and notice that the $a,b$ in the upper bound
satisfy $n-2C \leq a+b\leq n+2C$ and will therefore appear in one of the
products $v_{a'} v_{b'}$ for $a'+b'=n$. We will prove that this sequence $v_n$ grows
at most exponentially fast, from which the same result follows for $u_n$, as
desired. For small $z>0$, define $B(z) = \sum_{n\geq 1} C_1 e^{sn} z^n$ and
$V_N(z) = \sum_{n=1}^N v_n z^n$. The inequality~\eqref{eq:vn_le} gives
\begin{equation}
\label{eq:Vnz_le}
  V_N(z) \leq B(z) + C_1 V_{N-1}(z)^2.
\end{equation}
The function $B$ is smooth at $0$. Let $t$ be strictly larger than its
derivative at $0$. Fix $z$ positive and small enough so that $B(z) + C_1
(tz)^2 < t z$, which is possible since the function on the left has
derivative $<t$. We claim that $V_N(z) \leq t z$ for all $N$. This is obvious
for $N=0$ as $V_0=0$, and the choice of $z$ and the
inequality~\eqref{eq:Vnz_le} imply that, if it holds at $N-1$, then it holds
at $N$, concluding the proof by induction. In particular, $v_n z^n \leq
V_n(z) \leq tz$. This proves that $v_n$ grows at most exponentially.

\medskip

It remains to show~\eqref{eq:un_le}, using geometry. Let $A>0$ be large
enough ($A > D+1$ will suffice). Take $\gamma$ with $d(o, \gamma o) \in (n-1,
n]$. We consider two different cases: either $[o,\gamma o] \setminus (B(o, A)
\cup B(\gamma o, A))$ does intersect $\Gamma \tilde K$ (we say that $\gamma$
is recurrent -- this terminology is local to this proof), or it does not. The
former will give rise to the first term in~\eqref{eq:un_le}, the latter to
the second term.

We start with the non-recurrent $\gamma$'s. If $[o,\gamma o] \subset B(o, A)
\cup B(\gamma o, A)$, then $d(o, \gamma o)$ is uniformly bounded, so is $n$,
and the formula~\eqref{eq:un_le} is obvious for these finitely many $n$'s by
taking $C$ large enough. Assume now that $n$ is large. Consider the last
point $x$ on $[o,\gamma o] \cap B(o, A) \cap \Gamma \tilde K$, and the first
point $y$ on $[o,\gamma o] \cap B(\gamma o, A) \cap \Gamma \tilde K$. Take
$\gamma_x\in\Gamma$ such that $x \in \gamma_x \tilde K$, and $\gamma_y\in\Gamma$ such that
$y\in \gamma \gamma_y \tilde K$. Note that $\gamma_x$ and $\gamma_y$ belong
to a finite set $\calF_A$ (depending on $A$), made of these elements of
$\Gamma$ that move $o$ by at most $A + D$. Moreover, $\gamma' = \gamma_x^{-1}
\gamma \gamma_y$ belongs to $\Gamma_{\tilde{K}}$ since $[x,y]\cap \Gamma
\tilde K = \{x, y\}$ by construction.

Applying Lemma~\ref{lm:hold-potential} to the compact set $\bigcup_{g\in
\calF_A} g \tilde K$, we obtain a constant $C$ such that
\begin{equation*}
  \int_o^{\gamma o} \tilde F \leq  \int_{\gamma_x o}^{\gamma \gamma_y o} \tilde F + C
  = \int_o^{\gamma' o} \tilde F + C.
\end{equation*}
Finally, the contribution of the non-recurrent $\gamma$'s to $u_n$ is bounded from above
by
\begin{equation*}
  \sum_{\gamma_x, \gamma_y \in \calF_A}
\sum_{\substack{\gamma' \in \Gamma_{\tilde{K}}\\ d(o, \gamma' o) \in (n-1 - 2A-2D, n+2A+2D]}}e^{\int_o^{\gamma' o} \tilde F + C}.
\end{equation*}
The sum over $\gamma_x$ and $\gamma_y$ gives a finite multiplicity, and the
sum over $\gamma'$ is bounded by $C(A) e^{n s}$ for some constant $C(A)>0$ since $s
> \delta_{\Gamma_{\tilde K}}(F)$. This is compatible with the second term
in the upper bound of~\eqref{eq:un_le}.

We turn to the contribution to $u_n$ of the recurrent $\gamma$'s. For such a
$\gamma$, there is a point $x$ in $[o,\gamma o]\cap \Gamma \tilde K \setminus
(B(o, A) \cup B(\gamma o, A))$. Write $x = \gamma' x'$ with $x' \in \tilde
K$. Consider the integer $a$ such that $d(o, \gamma' o) \in (a-1, a]$. It
satisfies $A - D \leq a <n-A+D+1<n$, so if $A$ is large enough one has $1\le a\le n-1$. Let
$\gamma'' = {\gamma'}^{-1} \gamma $, so that $\gamma =  \gamma' \gamma''$.
The integer $b$ such that $d(o, \gamma'' o) \in (b-1, b]$ satisfies also $0<A - D \leq b <n-A+D+1<n$. Moreover,
\begin{multline*}
  a + b = d(o, \gamma' o) + d(o, \gamma'' o) \pm 2
  = d(o, \gamma' o) + d(\gamma' o, \gamma o) \pm 2
  \\
  = d(o, x) + d(x, \gamma o) \pm (2+2D)
  = d(o, \gamma o) \pm (2+2D)
  = n \pm (3+2D).
\end{multline*}
This shows that $1\le a,b\le n-1$ and $\abs{a+b-n} \leq 3+2D$. Finally, applying twice
Lemma~\ref{lm:hold-potential}, we obtain the existence of a constant $C'$ such
that
\begin{equation*}
  \abs*{\int_o^{\gamma o} \tilde F - \int_o^{\gamma' o} \tilde F - \int_o^{\gamma'' o}\tilde F} \leq C'.
\end{equation*}
Altogether, this shows that the contribution of recurrent $\gamma$'s to $u_n$
is bounded by the first term of the right hand side of~\eqref{eq:un_le}.
\end{proof}

\section{\texorpdfstring{Gurevič}{Gurevic} and geometric pressure at infinity coincide}\label{sec:cinq}

In this section, we will study and count the possible excursions of periodic
orbits outside large compact sets, and first deduce the inequality
\[
\PGur^\infty(F)\le \delta_\Gamma^\infty(F).
\]
The arguments we develop here will also be instrumental in the proof of the
inequality $\Pvar^\infty(F)\le \delta_\Gamma^\infty(F)$ in
Section~\ref{sec:ErgoPressure}.

These inequalities are the heart of Theorem~\ref{th:AllPressionEquivalent}.
The reverse inequalities $\PGur^\infty(F)\ge \delta_\Gamma^\infty(F)$ and
$\Pvar^\infty(F)\ge \delta_\Gamma^\infty(F)$ are simpler, and will be proven
respectively in Sections~\ref{sec:GurPressure}
and~\ref{subsec:Pvar_infty_easy}.

Let us explain why the above inequalities are the most surprising and
difficult. A major difference between the definition of
$\delta_\Gamma^\infty(F)$ and the two others is that $\PGur^\infty(F)$ and
$\Pvar^\infty(F)$ take into account trajectories (respectively periodic /
typical) that spend most of the time outside a given large compact set, but
can however come back inside this compact set several times, whereas
$\delta_\Gamma^\infty(F)$ considers trajectories that start and finish in a
given compact set, but never come back in the meantime. Thus, there are
apparently much more trajectories considered in the first two definitions.
However, in the next two sections, culminating in
Corollaries~\ref{cor:half-thm-Gur-geom} and~\ref{coro:PressureMassInfty}, we
prove that the above inequalities hold.

The strategy developed below is to cut a given trajectory, which comes back
several times inside a given compact set, but spends a small proportion of
time inside, into several excursions, and to prove precise upper bounds
presented below.


\subsection{Excursions of closed geodesics outside compact sets}

In this section, we study periodic orbits that intersect (the unit tangent
bundle of) a fixed compact subset $K\subset M$, but which spend most of their time
away from the $R$-neighborhood $K_R$  of $K$.

For all compact subsets $K_1 \subset K_2\subset M$  and   $0 < \alpha \leq 1$,
we define
\begin{equation*}
\begin{split}
\calP(K_1, K_2, \alpha) = {} &
\Bigl\{p \text{ periodic orbit} \; ; \; p \cap T^1 K_1 \ne\emptyset,\;
  \; \ell(p \cap T^1K_2) \leq \alpha \ell(p)\Bigr\}
\end{split}
\end{equation*}
and
\begin{equation}\label{eq:Periodic2K-time}
\calP(K_1, K_2, \alpha;T,T') =
\left\{p  \in
\calP(K_1, K_2, \alpha) ,\; T\le \ell(p)\le T'\right\}.
\end{equation}
Given a Hölder-continuous potential $F$, we define for all $T, T'>0$,
\[
\calN_F(  K_1, K_2, \alpha;T,T') =
\sum_{p \in \calP(K_1, K_2, \alpha; T, T')} e^{\int_p F}.
\]

\begin{theo}\label{th:CountExcursion}
Let $K\subset M$ be a compact subset, and $\tilde K\subset\tilde M$ be a
compact subset such that $p_\Gamma(\tilde K)=K$. Let $T_0>0$. Let $F : T^1M
\to \bbR$ be a Hölder-continuous potential with $\delta_{\Gamma_{\tilde
K}}(F)
> -\infty$. Let $\eta >0$. For all $0 < \alpha \leq 1$ and $R\geq 2$,
there exists a positive number $\psi = \psi(\tilde K, F, \eta, \alpha/R)$
such that
\begin{align*}
\limsup_{T\to +\infty} \frac 1 T \log \calN_F(K, K_R, \alpha; T, T+T_0)
&\leq (1-\alpha)\delta_{\Gamma_{\tilde K}}(F) + \alpha\delta_\Gamma(F) + \eta + \psi.
\end{align*}
Moreover, when $\tilde K, F$ and $\eta$ are fixed, $\psi(\tilde K, F, \eta,
\alpha/R)$ tends monotonically to $0$ when $\alpha/R$ tends to $0$.
\end{theo}

\begin{rema}
\label{rema:delta_eq_neg_inf} \rm When $\delta_{\Gamma_{\tilde K}}(F) =
-\infty$, the statement should be modified, replacing on the right hand side
$\delta_{\Gamma_{\tilde K}}(F)$ with an arbitrary real number $d$, and
allowing $\psi$ to depend on $d$. The same proof applies.
\end{rema}


Letting $R\to +\infty$, $\eta\to 0$ and at last $K$ exhaust $M$ and
$\alpha\to 0$, we deduce the following corollary.

\begin{coro}\label{cor:half-thm-Gur-geom}
Under the same assumptions on $M$ and $F$ as in Theorem~\ref{th:CountExcursion}, we have
\[
\PGur^\infty(F)\le \delta_\Gamma^\infty(F)\,.
\]
\end{coro}
\begin{proof}
If $\delta_\Gamma(F)$ is infinite, then $\delta_\Gamma^\infty(F)$ is also
infinite by Proposition~\ref{prop:infinite_pressure}, and the result is
obvious. We can therefore assume $\delta_\Gamma(F)<\infty$. We will also
assume $\delta_\Gamma^\infty(F) > -\infty$, as the case
$\delta_\Gamma^\infty(F) = -\infty$ can be proved similarly using
Remark~\ref{rema:delta_eq_neg_inf}.

Let $\eta>0$. We have to find a compact subset $\tilde L$ of $\tilde M$ whose interior
intersects $\pi(\tilde \Omega)$, and $\alpha>0$, such that, with $L=p_\Gamma(\tilde L)$, the exponential
growth rate of $\sum_{p \in  \calP_{L} (T, T+1) \;; \; \ell(p\cap T^1{L}
)<\alpha \ell(p)} e^{\int_p F}$ is at most $\delta_\Gamma^\infty(F) + 3\eta$.
Fix a large compact set $\tilde K$ with $\delta_{\Gamma_{\tilde K}}(F) \leq
\delta_\Gamma^\infty(F) + \eta$. We denote respectively by $\tilde K_2$ and
$\tilde K_3$ the neighborhoods of size $2$ and $3$ of $\tilde K$. We wish to
apply Theorem~\ref{th:CountExcursion} with $R=3$, and set $\tilde L=\tilde
K_3$, and $L=p_\Gamma(\tilde K_3)$.

There is a difficulty coming from the fact that the definition of the Gurevič
pressure involves all periodic orbits going through $\tilde L=\tilde K_3$,
while Theorem~\ref{th:CountExcursion} only takes into account those that,
additionally, enter $\tilde K$. This difficulty is solved using
Proposition~\ref{prop:gurevic_subset} applied with  $K$ instead of $K'' $,
$K_2$ instead of $K'$ and $K_3= p_\Gamma(\tilde K_3)  $ instead of $K$. This
proposition yields that the exponential growth rate of $ \sum_{p \in
\calP_{L} (T, T+1) \;; \; \ell(p\cap T^1L )<\alpha \ell(p)} e^{\int_p F}$ is
bounded by that of $ \sum_{p \in \calP_{K} (T, T+1) \;; \; \ell(p\cap
T^1K_2)<2\alpha \ell(p)} e^{\int_p F}$. The latter can be estimated thanks to
Theorem~\ref{th:CountExcursion} applied to $R=2$, $T_0=1$ and $2\alpha$: this
growth rate is bounded by $(1-2\alpha) \delta_{\Gamma_{\tilde K}}(F) +
2\alpha \delta_\Gamma(F) + \eta + \psi(\alpha)$, where $\psi(\alpha)$ tends
to $0$ with $\alpha$. This quantity converges to $\delta_{\Gamma_{\tilde
K}}(F) + \eta \leq \delta_\Gamma^\infty(F) + 2\eta$ when $\alpha$ tends to
$0$, so that for some $\alpha>0$ it is strictly smaller than $
\delta_\Gamma^\infty(F) + 3\eta$.
\end{proof}


The strategy of the proof of Theorem~\ref{th:CountExcursion} is as follows. A
periodic orbit will be cut into two kinds of segments, those which stay in
the given compact set $K$, and the excursions outside this compact set. The
weighted growth of the excursions should be controlled by the exponent
$\delta_{\Gamma_K}(F)$ multiplied by the proportion of time spent outside
$K$, and the weighted growth of the segments inside $K$ should be controlled
by $\delta_\Gamma(F)$ multiplied by the proportion of time spent in $K$.
However, to succeed to get such a control, we need to avoid the situation
with several very short excursions in a very close neighborhood of $K$. For
this reason, we need to play with two compact sets, $K$ and its
$R$-neighborhood $K_R$.

\begin{proof}[Proof of Theorem~\ref{th:CountExcursion}] As in the above proof, we can assume
that $\delta_\Gamma(F)$ and $\delta_\Gamma^\infty(F)$ are finite.
Let $\tilde K\subset\tilde M$   be a compact set and $\tilde K_R\subset
\tilde M$ be its $R$-neighborhood, and set $K=p_\Gamma(\tilde  K) $,
$K_R=p_\Gamma(\tilde  K_R)$. Let $D$ be the diameter of $K$. Any geodesic
segment joining the boundary of $\tilde K$ and the boundary of $\tilde K_R$
has length at least $R$ and at most $D+2R$. Let also $D'=D'(K, T_0)$ be
larger than the diameter of $K \cup \{o\}$, $1$ and $T_0$.

Consider a periodic orbit $p\in \calP(K, K_R, \alpha)$ with $\ell(p)\in [T,
T+ T_0]$. By assumption, $\pi(p)\cap K \neq \emptyset$.  We will divide it
into long excursions, i.e., those excursions  outside both $K$ and $K_R$, of
total length at least $(1-\alpha)\ell(p)$, and periods of time of total
length at most $\alpha\ell(p)$ where it stays inside $K_R$.

Since $p$ intersects $T^1 K$, we can choose a lift $c$ of $p$ that intersect
$T^1 \tilde K$. Let $g$ be a hyperbolic isometry whose translation axis is $c$, and
whose translation length is $\ell(p)$, and which translates in the direction
given by the orientation of $p$.

Define inductively points $a_i, b_i$ on $c$ as follows. Choose first a point
$a_0$ on $c$ inside $\tilde K$. Consider on  the geodesic segment $[a_0,
g.a_0]$ of $c$ the first points $b_0,a_1 \in\Gamma\partial\tilde K$ such that the open interval
$(b_0,a_1)$ does not intersect $ \Gamma \tilde K$ and $(b_0,a_1)\cap\tilde M\setminus
\Gamma \tilde K_R\neq\emptyset$. The interval $(b_0,a_1)$ projects through
$p_\Gamma$ into a {\em long excursion}, i.e., an excursion outside $K$ which
also goes outside $K_R$. Inductively, we define $(b_1, a_2), \dotsc,
(b_{N-1},a_N)$ by the properties that $b_i,a_{i+1}$ are the first points of
$[a_i, g.a_0]$ which lie in $\Gamma\partial\tilde K$ and satisfy
$(b_i,a_{i+1})\cap \Gamma \tilde K=\emptyset$ and $(b_i,a_{i+1})\cap\tilde
M\setminus \Gamma \tilde K_R\neq\emptyset$. In other terms, the intervals
$(b_i,a_{i+1})$, $0\le i\le N-1$, are the connected components
 of $[a_0, g.a_0]\setminus\Gamma\tilde K$ that intersect $\tilde M\setminus \Gamma\tilde K_R$,
whereas the segments $[a_i, b_{i}]$ are included in $\Gamma\tilde K_R$.
Finally, set $b_N=g.a_0$.

\begin{figure}[ht!]\label{groupactions}
\begin{center}
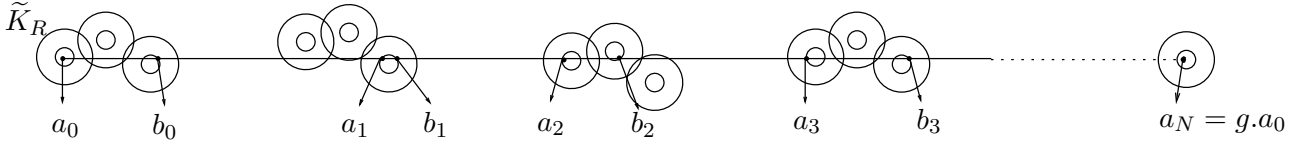
\caption{Long excursions outside $\tilde K$ and $\tilde K_R$ }
\end{center}
\end{figure}
For all $0\le i\le N$, choose elements $\gamma_i^{\pm}\in \Gamma$ such that
$a_i\in \gamma_i^-\tilde K$ and $b_i\in \gamma_i^+\tilde K$. As $\tilde K$ is
compact and the action of $\Gamma$ is proper, for each $i$, there are only
finitely many choices of such elements $\gamma_i^\pm$. Without loss of
generality, set $\gamma_0^-=\operatorname{Id}$ and $\gamma_N^+=g$.

\medskip

Choose some $\varepsilon>0$.
  The following elementary observations are crucial for the sequel.

\begin{enumerate}
\item As $\bigcup_{0\le i\le N}[a_i, b_i]\subset \Gamma  \tilde K_R$, by
    definition of $\calP(K, K_R, \alpha)$ and since $T \leq \ell(p) \leq T
    + T_0$, we have
\[
\ell(p\cap T^1K)\le\sum_{i = 0}^N d(a_i, b_i) \leq \alpha(T + T_0) \leq \alpha T + D'.
\]

\item For all $i\in \{0, \dotsc, N-1\}$,   we have $(b_i, a_{i+1})\subset
    \tilde M \setminus\Gamma \tilde K $. Moreover, the length of
    $(b_i,a_{i+1})\cap \Gamma\tilde K_R$ is at least $2R$ and
    $\bigcup_{i}[b_i,a_{i+1}]$ does not intersect the interior of
$\Gamma\tilde K$, so that by definition of $\calP(K, K_R, \alpha)$,
\begin{equation}\label{eqn:estimate-on-time}
(1-\alpha) T+2RN \leq \sum_{i = 0}^{N-1} d(b_i, a_{i+1}) \leq  T+ T_0 \leq T + D',
\end{equation}
and therefore, for $\nu\coloneqq \frac{1}{2R}\left(\alpha T+D'\right)$, we
have
\begin{equation}\label{eqn:nb-excursions}
N\le  \nu\,.
\end{equation}
For large enough $T$, we have
\begin{equation}
\label{eq:nu_le}
  \nu \leq \frac{\alpha}{R}T.
\end{equation}

\item Write $\psi_i = (\gamma_i^-)^{-1}\gamma_i^+ \in \Gamma$ for all $i =
    0, \dotsc, N$. We have $\abs{d(o, \psi_i o) - d(a_i, b_i)}\le 2D'$, so
    that
\[
\sum_{i = 0}^N d(o, \psi_i o) \leq \alpha (T+T_0) + 2(N+1)D' \leq \alpha T + 5N D'\,.
\]
Let $s_i$ be the unique integer such that  $d(o,\psi_i o)\le s_i
<d(o,\psi_{i} o)+1$. Then
\begin{equation*}
s_0+ \dotsb + s_N \le\alpha T + 5ND'+N+1\le  \alpha T + 7ND'\,.
\end{equation*}

\item By definition of $\Gamma_{\tilde K}$, for all $i = 0, \dotsc, N-1$, since $\left((\gamma_i^+)^{-1}b_i,\varphi_i(\gamma_{i+1}^-)^{-1}a_{i+1}\right)$ does not intersect $\Gamma \widetilde K$
    we have $\phi_i =  (\gamma_i^+)^{-1}\gamma_{i+1}^-\in \Gamma_{\tilde
    K}$. Moreover, $\abs{d(o, \phi_i o) - d(b_i, a_{i+1})}\le 2D'$. Let
    $t_i$ be the unique integer such that $d(o,\varphi_i o)\le t_i
    <d(o,\varphi_i o)+1$.

\item As $\sum_{i=0}^N d(a_i,b_i)+\sum_{i=0}^{N-1}
    d(b_i,a_{i+1})=d(a_0,b_N)=\ell(p)\in[T,T+T_0]$, we get
\[
\abs*{\sum_{i = 0}^{N} d(o, \psi_i o) + \sum_{i=0}^{N-1} d(o, \phi_i o) - T} \leq T_0 + (4N+2)D'
\]
and therefore
\[
\abs*{\sum_{i = 0}^{N}s_i + \sum_{i=0}^{N-1} t_i - T} \leq T_0 + (4N+2)D' + (2N+1) \leq 10 ND'\,.
\]
\item By~\eqref{eqn:estimate-on-time}, as $d(b_i,a_{i+1})-2D'\le t_i\le
    d(b_i,a_{i+1})+2D'+1$, we get
\begin{equation*}
(1-\alpha)T-2ND'\le  \sum_{i=0}^{N-1} t_i \leq T+4ND'.
\end{equation*}

\item Since $M$ has pinched negative sectional curvature and $F$ is
    $(\beta, C_F)$-Hölder-continuous, Lemma~\ref{lm:hold-potential} applied
    to the compact set $\tilde K \cup \{o\}$ ensures that there exists a
    constant $C(F,\tilde K)$ depending only on the upper bound of the
    curvature, on $\tilde K$ and the Hölder-continuous constants of $F$
    such that for all $i = 0, \dotsc, N$,
\[
\abs*{\int_{a_i}^{b_i}\tilde F - \int_o^{\psi_i o}\tilde  F }\le C(F,\tilde K)\,.
\]


\item Similarly, for all $i = 0, \dotsc, N-1$,
\[
\abs*{\int_{b_i}^{a_{i+1}} \tilde F - \int_o^{\phi_i o} \tilde F} \le C(F,\tilde K)\,.
\]

\item As $\int_p F = \int_{a_0}^{g a_0} \tilde F$, and bounding $2N+1$ from above with
    $3\nu$, we deduce
\begin{equation}
\label{eq:int_pF_le}\sum_{i=0}^N \int_{o}^{\psi_i o}\tilde F +
\sum_{i = 0}^{N-1} \int_o^{\phi_i  o}\tilde F -3 C(F, \tilde K)\nu \leq
\int_p F \leq \sum_{i=0}^N \int_{o}^{\psi_i o}\tilde F +
\sum_{i = 0}^{N-1} \int_o^{\phi_i  o}\tilde F +3 C(F, \tilde K)\nu.
\end{equation}
\end{enumerate}

\medskip

For all $t\in \bbN$, set
\[
\Gamma(t-1,t) = \{ \gamma\in \Gamma\; ;\; d(o, \gamma o)\in (t-1, t]\} \quad
\text{ and } \quad \Gamma_{\tilde K}(t-1,t) = \Gamma(t-1,t) \cap \Gamma_{\tilde K}.
\]
We also write
\[
Q_{F, \Gamma}(t) = \sum_{\gamma\in \Gamma(t-1,t)} e^{\int_o^{\gamma o} \tilde F} \quad
\text{ and } \quad Q_{F, \Gamma_{\tilde K}}(t) = \sum_{\gamma\in \Gamma_{\tilde K}(t-1,t)} e^{\int_o^{\gamma o}\tilde F} \,.
\]

To each periodic orbit $p\in \calP(K, K_R, \alpha)$ with $\ell(p)\in [T, T+
T_0]$, we have associated a hyperbolic isometry $g\in \Gamma$ whose axis
intersects $\tilde K$ and projects through $p_\Gamma$ onto $\pi(p)$ and with
translation length equal to $\ell(p)$. Then, to each such element $g$ we have
associated by the previous construction finite sequences $\phi_0,
\dotsc,\phi_{N-1}$ in $\Gamma_{\tilde K}$ and $\psi_0, \dotsc, \psi_N\in \Gamma$. As
one can recover $g$ (and then $p$, which is the projection of the translation axis of
$g$) from these sequences by the formula $g = \psi_0 \phi_0 \psi_1 \dotsm
\phi_{N-1}\psi_N$, this association is injective.

Let us now bound $\calN_F(K, K_R, \alpha;T, T+T_0)$. Summing the exponentials of the
bounds~\eqref{eq:int_pF_le} over all the periodic orbits in $\calP(K, K_R,
\alpha;T,T+T_0)$, we get the inequality
\begin{equation}
\label{eq:CountExcursion1}
\begin{split}
\calN_F(K, & K_R, \alpha;T, T+T_0  ) \leq
 e^{3 C(F, \tilde K)\nu}
\sum_{N = 0}^{\nu(\alpha,T,T_0,R)} \sum_{\substack{t_0, \dotsc, t_{N-1}, s_0,\dotsc, s_N \in \mathbb{N} \\ \abs*{\sum s_i + \sum t_i - T}\leq 10ND'\\ \sum t_i \geq (1-\alpha)T - 2ND'}}
\\ &
 Q_{F, \Gamma}(s_0) \cdot Q_{F, \Gamma_{\tilde K}}(t_0)\cdot Q_{F, \Gamma}(s_1)\cdot Q_{F, \Gamma_{\tilde K}}(t_1)\dotsm Q_{F, \Gamma_{\tilde K}}(t_{N-1}) \cdot Q_{F, \Gamma}(s_N).
\end{split}
\end{equation}

The following lemma is a straightforward consequence of the definition of the
critical exponents $\delta_\Gamma(F)$ and $\delta_{\Gamma_{\tilde K}}(F)$.

\begin{lemm}\label{lem:majoration-seulement}
For all $\eta>0$, there exists $C_\eta = C_\eta(\tilde K, F, \eta) \geq 1$
such that for all $t>0$, we have
\[
Q_{F,\Gamma}(t)\leq C_\eta e^{\delta_\Gamma(F) t + \eta t} \quad
\text{ and } \quad Q_{F,\Gamma_{\tilde K}}(t) \leq C_\eta e^{\delta_{\Gamma_{\tilde K}}(F)t + \eta t}.
\]
\end{lemm}
We can write the second bound as $Q_{F,\Gamma_{\tilde K}}(t) \leq C_\eta
e^{(\delta_{\Gamma_{\tilde K}}(F) - \delta_\Gamma(F))t + \delta_\Gamma(F) t +
\eta t}$. Multiplying these bounds, we get
\begin{align*}
 & Q_{F, \Gamma}(s_0) \cdot Q_{F, \Gamma_{\tilde K}}(t_0)\cdot Q_{F, \Gamma}(s_1)\cdot Q_{F, \Gamma_{\tilde K}}(t_1)\dotsm Q_{F, \Gamma_{\tilde K}}(t_{N-1}) \cdot Q_{F, \Gamma}(s_N) \\
  & \leq C_\eta^{2N+1} \exp\pare*{(\delta_\Gamma(F) +\eta) (\sum s_i + \sum t_i) + (\delta_{\Gamma_{\tilde K}}(F) - \delta_\Gamma(F))(\sum t_i)} \\
  & \leq C_\eta^{3N} \exp\pare*{(\delta_\Gamma(F) +\eta) T + (\abs{\delta_{\Gamma}(F)} + \eta) 10ND' + (\delta_{\Gamma_{\tilde K}}(F) - \delta_\Gamma(F))((1-\alpha)T-2ND')} \\
   & = C_\eta^{3N} \exp\pare*{ \left(\alpha \delta_\Gamma(F) + (1-\alpha) \delta_{\Gamma_{\tilde K}}(F) + \eta\right)T + \left(\abs{\delta_{\Gamma}(F)} + \eta + \delta_\Gamma(F) - \delta_{\Gamma_{\tilde K}}(F)\right)  10ND'}.
\end{align*}
Note that this bound does not depend anymore on the choice of the $s_i$ and
$t_i$. In order to bound~\eqref{eq:CountExcursion1}, one should take into account a
multiplicity given by the number of possible choices for these integers.

The following combinatorial standard estimate will control the number of
possible choices.

\begin{lemm}\label{lm:Combi1}
Let $\tau, \kappa\in \mathbb N$ be integers with $\kappa<\tau$. The number of
ordered integer decompositions of $\tau$ of length $\kappa$, i.e., the number
of $(u_1,\dotsc, u_\kappa)\in \mathbb N^\kappa$ such that $u_i\ge 0$ and
$u_1+ \dotsb + u_{\kappa} \le \tau$, is equal to
\[
 \binom{\tau+\kappa}{\kappa} = \frac{(\tau+\kappa)!}{\kappa!\tau!}.
\]
\end{lemm}

Then $(s_0, t_0, s_1,\dotsc, s_N)$ forms an ordered integer decomposition of some integer
$\tau \le T+10ND'$, with $\kappa=2N+1$. Their number is thus bounded by $\binom{T+10ND'+2N+1}{2N+1}$.
Recall that by~\eqref{eqn:nb-excursions}, we have $N\leq \nu$, which is
bounded by $T/2$ for large $T$, so that $T+10ND' + 2N+1 \leq 8D'T$ and $2N+1
\leq 3\nu \leq 8D'\nu$. We get $\binom{T+10ND'+2N+1}{2N+1} \leq
\binom{8D'T}{2N+1}\leq \binom{8D'T}{8D'\nu}$ thanks to monotonicity
properties of binomial coefficients. Summing over all the values of $N$, we
obtain the estimate
\begin{multline*}
  \calN_F(K, K_R, \alpha;T, T+T_0)
  \leq (\nu + 1) \cdot \binom{8D'T}{8D'\nu} e^{3C(F, \tilde{K}) \nu}\cdot  C_\eta^{3\nu} \\
  \exp\pare*{(\alpha \delta_\Gamma(F) + (1-\alpha) \delta_{\Gamma_{\tilde K}}(F) + \eta)T
    + (\abs{\delta_{\Gamma}(F)} + \eta + \delta_\Gamma(F) - \delta_{\Gamma_{\tilde K}}(F)) 10\nu D'}\,.
\end{multline*}

To conclude the proof, we should estimate the exponential growth rate of the
various terms in this expression when $T$ tends to infinity. Recall that $\nu
\leq \alpha T/R$ by~\eqref{eq:nu_le}. Stirling's formula $n! \sim \sqrt{2\pi
n}(n/e)^n$ implies that the exponential growth rate of $\binom{8D'T}{8D'\nu}
\leq \binom{8D'T}{8D'T \cdot \alpha/R}$ is bounded by $-\rho \log \rho -
(1-\rho)\log (1-\rho)$ for $\rho = \alpha/R$. Finally, the exponential growth
rate of $\calN_F(K, K_R, \alpha;T, T+T_0)$ is bounded by
\begin{multline*}
  \alpha \delta_\Gamma(F) + (1-\alpha) \delta_{\Gamma_{\tilde K}}(F) + \eta -\rho \log \rho - (1-\rho)\log (1-\rho)
  \\
   + \pare*{3C(F, \tilde{K}) + 3 \log C_\eta + 10D'(\abs{\delta_{\Gamma}(F)} + \eta + \delta_\Gamma(F) - \delta_{\Gamma_{\tilde K}}(F))}
  \frac{\alpha}{R}\,.
\end{multline*}
This concludes the proof of the theorem.
\end{proof}



\subsection{\texorpdfstring{Gurevič}{Gurevic} and geometric pressures at infinity coincide}\label{sec:GurPressure}

This paragraph is devoted to the proof of the following part of
Theorem~\ref{th:AllPressionEquivalent}.

\begin{theo}\label{th:PressureGeod}
For all  Hölder-continuous potentials $F: T^1M \to \bbR$ with finite
pressure, we have
\[
\PGur^\infty(F) = \delta_\Gamma^\infty(F).
\]
\end{theo}

By Corollary~\ref{cor:half-thm-Gur-geom}, it is enough to prove the
inequality $\PGur^\infty(F) \geq \delta_\Gamma^\infty(F)$.
\begin{proof}
The set of periodic orbits of the geodesic flow (counted with locally bounded
multiplicities in the orbifold case) is in $1-1$ correspondence with the set
of conjugacy classes of hyperbolic elements of $\Gamma$. Let us recall how.
Given a periodic orbit $p\subset T^1M$, its preimage $p_\Gamma^{-1}(p)\subset
T^1\tilde M$ is a countable union of orbits of the geodesic flow on
$T^1\tilde M$. Each of these orbits projects on $\tilde M$ to the translation
axis of a hyperbolic element of $\Gamma$, which is unique (modulo the
pointwise stabilizer of their axis) when requiring that this element
translates along the axis with translation length equal to $\ell(p)$, and in
the direction given by the direction of $(g^t)_{t>0}$ on this orbit. The
number of conjugacy classes of hyperbolic elements (modulo the pointwise
stabilizer of their axis) associated with $p$ in this way is equal to the
multiplicity of $p$.

Let $K\subset M$ be a compact subset whose interior intersects a closed
geodesic, and containing the projection $p_\Gamma(o)$. Let $\tilde K$ be a
compact subset of $\tilde M$ which contains $o$, such that $p_\Gamma(\tilde
K)=K$, and whose interior intersects $\tilde\Omega$. Let $N$ be the maximal
multiplicity of $p_\Gamma$ on $\tilde K$. Let $D$ be the diameter of $\tilde K$.

With the notation of~\eqref{eq:Periodic2K-time}, set
\[
 \calP(K, \alpha) \coloneqq \calP(K,K, \alpha)  \quad  \text{ and}  \quad
 \calP(K, \alpha; T,T') \coloneqq \calP(K,K, \alpha; T,T') \,.
\]

First, by Lemma~\ref{lem:Pit-Schapira2.6}, there exists a finite set
$\mathcal{G} = \{g_1,\dotsc, g_k\} \subset \Gamma$, such that, for all
$\gamma\in\Gamma_{\tilde K}$, there exist $g_i, g_j$ (not necessarily unique)
such that $g_i^{-1}\gamma g_j$ is hyperbolic with a translation axis which intersects
$\tilde K$. Let $p_\gamma$ be the associated periodic orbit (it depends on
the choice of $g_i,g_j$ but this is not a problem). As the axis of
$g_i^{-1}\gamma g_j$ intersects $\tilde K$, we deduce that
\[
\abs*{\ell(p_\gamma)-d(o,g_i^{-1}\gamma g_j o)}\le 2D\,.
\]
By the triangular inequality, we deduce that
\[
\abs{d(o,\gamma o)-\ell(p_\gamma)}\le 2D + 2\max_{1\le i\le k}(d(o,g_i o))\,.
\]
Similarly, thanks to Lemma~\ref{lm:hold-potential}, and using the fact that
$\tilde F$ is bounded on the $\delta$-neighborhood of $\Gamma\tilde K$, with
$\delta=\max_{1\le i\le k} (d(o,g_i o))$, we deduce that there exists  a constant $C=C(F,
\tilde K, g_1, \dotsc, g_k)$ such that
\[
\abs*{\int_o^{\gamma o} \tilde F-\int_{p_\gamma} F} \le C\,.
\]

Choose now some $R>1$, and let $\tilde K_R$ be the $R$-neighborhood of
$\tilde K$. For $\gamma\in\Gamma_{\tilde K_R}$, there exist $a \in \tilde
K_R$ and $b\in \gamma\tilde K_R$ such that the geodesic segment $[a,b]$ only
meets $\Gamma \tilde K_R$ at its endpoints. Using the above notations, we
assume that $g_i^{-1}\gamma g_j$ is hyperbolic with associated periodic orbit
$p_\gamma$. The point $g_i o$ is at  distance at most $\delta$ from
$o$, which is at distance at most $D$ from $a$, and the point $g_j o$
is at  distance at most $\delta$ from $\gamma g_j o$ which is at
  distance at most $D$ from $b$. Therefore, by
Lemma~\ref{lm:NegCurv4Points}, there exists a constant $T_0>0$ depending on
$\delta$, $D$ and the bounds on the curvature, such that, when removing
segments of length  $T_0$ at the beginning and the end of $[g_io,\gamma g_jo
]$, the middle segment is in a neighborhood of radius less than $1/2$ from
the geodesic segment $[a,b]$.

On the other hand, the periodic orbit $p_\gamma$ associated with
$g_i^{-1}\gamma g_j$ admits a translation axis which intersects $\tilde K$. Let
$x\in\tilde K$ be a point on this axis and $g_i^{-1}\gamma g_j x\in
g_i^{-1}\gamma g_j\tilde K$ its image by $g_i^{-1}\gamma g_j$. By
Lemma~\ref{lm:NegCurv4Points},   when removing segments of length $T_0$ at
the beginning and the end of the segment $[x, g_i^{-1}\gamma g_j x]$, the
middle segment is in a neighborhood of size less than $1/2$ of the geodesic
segment $[o,g_i^{-1}\gamma g_j o]$.

The triangular inequality implies that, after removing segments of length
$2T_0$ at the beginning and at the end of the geodesic segment $[g_i x,
\gamma g_j x]$, this segment is at distance at most $1/2$ of $[g_i o,\gamma
g_j o]$, and therefore, at distance at most $1$ from $[a,b]$. In particular,
as $\gamma\in \Gamma_{\tilde K_R}$, and $R\ge 1$, after removing segments of
length $2T_0+R$ at the beginning and the end of $[g_i x, \gamma g_j x]$, this
segment spends the rest of the time outside $\Gamma \tilde K$.

We deduce that the time spent by $p_\gamma$ inside $K$ is at most $4T_0+2R$.
In particular, when $\ell(p_\gamma)\ge \frac{4T_0+2R}{\alpha}$, the periodic
orbit $p_\gamma$ spends a proportion of time at most $\alpha$ inside $K$. As
$\abs{d(o,\gamma o)-\ell(p_\gamma)}\le 2D + 2\delta$, it implies that as soon
as $d(o,\gamma o)\ge 2D + 2\delta+ \frac{4T_0+2R}{\alpha}$, then $p_\gamma$
belongs to $\calP(K, \alpha)$. In particular, when $T>1+ 2D+2\delta
+\frac{4T_0+2R}{\alpha}$, the above considerations show that for
$\gamma\in\Gamma_{\tilde K _R}(T-1,T)$, the associated periodic orbit
$p_\gamma$ belongs to $\calP(K,\alpha,T-1-2D-2\delta, T+2D+2\delta)$.

The translation axis of $g_i^{-1}\gamma g_j$ is a lift of $p_\gamma$ that intersects
$\tilde K$. The number of such lifts is at most linear in $\ell(p_\gamma)$,
by~\eqref{lm:nW}. Therefore, the multiplicity of the above map $\gamma\mapsto
p_\gamma$ is at most linear in $\ell(p_\gamma)$.

The above considerations imply that there exist constants $C$ and $\tau$
depending only on $K,\tilde K, D,\alpha,F$ such that for $T>0$ large enough,
and all $R>1$,
\[
\sum_{\gamma\in \Gamma_{{\tilde K}_R},\, T-1 \leq d(o, \gamma o)\leq T} e^{\int_o^{\gamma o}\tilde F}  \le
C\times T\times \sum_{p\in  \mathcal{P}(K, \alpha,T-1-\tau,T+\tau)} e^{\int_{p}F}\,.
\]
Taking $\frac{1}{T}\log $ of the above inequality, and letting $T\to +\infty$, and then letting $R\to +\infty$ and $\alpha \to 0$ gives $\PGur^\infty(F) \geq \delta_\Gamma^\infty(F)$.
\end{proof}


\section{Variational and geometric pressures at infinity coincide}\label{sec:ErgoPressure}

This section is devoted to the proof of the equality between geometric and
variational pressures at infinity.

\begin{theo}\label{th:ErgoPressure}
Let $F: T^1M \to \bbR$ be a Hölder-continuous potential. Then
\[
\delta_{\Gamma}^\infty(F) = \Pvar^\infty(F).
\]
\end{theo}

The first paragraph contains the proof of the easier inequality
$\delta_\Gamma^\infty(F)\le \Pvar^\infty(F)$. The harder inequality
$\Pvar^\infty(F)\le \delta_\Gamma^\infty(F)$ will follow from
Section~\ref{sec:cinq}, after some reductions. First, in
Section~\ref{ssec:ErgodicPreliminaries}, we introduce a notion of pressure,
that we call {\em Katok pressure} in reference to the Katok entropy
introduced in~\cite{Katok80}. We show that the variational pressure is
bounded from above by this new pressure, involving spanning sets. Using the
closing lemma, in Section~\ref{sec:escape-of-mass}, we study escape of mass
of sequences of probability measures, and relate this new pressure to the
Gurevič pressure (which involves weighted growth of periodic orbits), and
conclude the proof of the inequality $\Pvar^\infty(F)\le
\delta_\Gamma^\infty(F)$ thanks to Theorem~\ref{th:CountExcursion}.

\subsection{The first inequality}
\label{subsec:Pvar_infty_easy}

This paragraph is devoted to the proof of  the easier inequality
$\delta_\Gamma^\infty(F)\le \Pvar^\infty(F)$. We deal first with the
exceptional situation where $\delta_\Gamma(F)=\infty$.

\begin{lemm}
\label{lem:Pvar_infty_infty} Under the assumptions of
Theorem~\ref{th:ErgoPressure}, if we assume $\delta_\Gamma(F)=\infty$, then
for any compact subset $K$ in $  M$ and any $C, \epsilon>0$, there exists $\mu
\in \calM_{1, \mathrm{erg}}^F$ such that $\mu(T^1K)<\epsilon$ and
$h_{KS}(\mu) + \int F \dd\mu > C$.
\end{lemm}
\begin{proof}
The entropy $h_{KS}(\mu)$ of any invariant measure $\mu\in\calM_{1,
\mathrm{erg}}^F$ is nonnegative.
Therefore,  it suffices to find a measure $\mu\in\calM_{1,
\mathrm{erg}}^F$ with
$\mu(T^1K)<\epsilon$ and $\int F \dd\mu>C$. By Theorem~\ref{th:Variationnel},
$\Pvar(F) = \infty$.

Let $R=R(C, K)$ and $C'=C'(C, K, R)$ be two large enough constants, to be
determined later on in the proof.   The equality $\Pvar(F) = \infty$ ensures the
existence of a measure $\nu \in \calM_1^F$ with $\int F \dd\nu > C'$, for
arbitrarily large $C'>0$. Taking an ergodic component of $\nu$ if necessary,
we can assume that $\nu$ is ergodic. If $\nu(T^1 K)<\varepsilon$, we are done
choosing $\mu=\nu$ and $C'=C$.

Otherwise, consider a $\nu$-typical vector $v$ in $T^1K$. By Birkhoff ergodic
Theorem and Poincaré recurrence Theorem, one can find an arbitrarily large
$T>0$ such that   $1/T \int_0^T F(g^t v)\dd t> C'$ and $g^T v \in K$.

Let $K_1$ (resp.\ $K_R$) be the neighborhood of size $1$ (resp.\ $R$) of $K$.
Consider the open set $\{t\in [0,T], \,\,g^t v \notin K_1\}$.   Inside this
set, consider those connected components that contain some $t$ such that $g^t
v$ does not belong to $K_R$. These components have length at least $2(R-1)$.
If $C'$ is large enough so that $\abs{F}<C'$ on $K_R$, we claim that there
exists such a component $(a,b)$ such that $\int_a^b F(g^t v) > C' (b-a)$.
Indeed, otherwise, one would get $\int_0^T F(g^t v) \leq C' T$ by summing the
contributions of these big connected components, and integrating the bound
$\abs{F}\leq C'$ on the remaining points.

Set $w = g^a v$, and $\tau=b-a$. The piece of orbit $(g^tw)_{0\le t\le \tau}$
has length larger than $2(R-1)$, its projection to $M$ starts and ends in $\partial K_1$ and remains
outside  $K_1$ in between, and it
satisfies $\int_0^{\tau} F(g^t w) \dd t \geq
\tau C'$.

Using the connecting lemma~\ref{lem:connecting} in the compact subset $K_1$,
we get a closed orbit $(g^t w')_{0\leq t \leq \tau + s}$, with $s \leq T_0$,
for some $T_0$ depending only on $K_1$, which stays at distance at most $1/2$
of the orbit of $w$ for $T_0 \leq t\leq \tau - T_0$. In particular, $g^tw'$
can belong to $K$ only for $t\leq T_0$ or $\tau - T_0 \le t \le \tau+s$.
Define the measure $\mu$ as the uniform probability measure along this periodic orbit. If
$R$ has been chosen large enough compared to $T_0$, we deduce $\mu(T^1 K)\le
\frac{3T_0}{2(R-1)}\le \varepsilon$. Let us now check that $\int F \dd\mu$ is
large. First, $\abs*{\int_0^\tau F(g^t w')\dd t- \int_0^\tau F(g^t w)}$ is
bounded by a constant $C_0$ depending only on $K$, by
Lemma~\ref{lm:hold-potential}. Second, $\int_\tau^{\tau + s} F(g^t w')$ is
bounded from below by a constant $-C_1$ depending only on $K$, as $s$ is
bounded by $T_0$ and $F$ is bounded on the $(T_0+2)$-neighborhood of $K$. We
get
\begin{equation*}
  \int_0^{\tau+s} F(g^t w') \dd t \geq \int_0^\tau F(g^t w) \dd t - C_0 - C_1
  \geq C' \tau - C_0 - C_1.
\end{equation*}
If $C'=C'(K, C, R)$ is large enough, this is at least $C (\tau + s)$, as
desired.
\end{proof}

\begin{prop}\label{prop:first-inequality}  Under the assumptions of Theorem~\ref{th:ErgoPressure},
let $F$ be a Hölder-continuous map. Then $\delta_\Gamma^\infty(F)\le
\Pvar^\infty(F)$.
\end{prop}
\begin{proof}
If $\delta_\Gamma(F) = \infty$,   Lemma~\ref{lem:Pvar_infty_infty} shows that
one can find a sequence of measures $\mu_n\in \calM_1^F$ tending weakly to
$0$ such that $h_{KS}(\mu_n) + \int_{T^1M} F \dd\mu_n$ tends to infinity.
Therefore, $\Pvar^\infty(F)=\infty$, and the result is obvious. If
$\delta_\Gamma^\infty(F) = -\infty$, the result is also obvious.

Assume now that $\delta_\Gamma(F) < \infty$ and $\delta_\Gamma^\infty(F) >
-\infty$. Choose for every $R\in\bbN\setminus\{0\}$ a Hölder-continuous map
$0\le \chi_R\le 1$ which approximates ${\bf 1}_{T^1p_\Gamma B(o,R)}$ on
$T^1M$: $\chi_R\equiv 1$ on $T^1\left(p_\Gamma B(o,R-1)\right)$ and
$\chi_R\equiv 0$ outside $T^1\left(p_\Gamma B(o,R)\right)$. Define
$F_{n,R}=F-n\chi_R$, for all $n\in\bbN$, and note that $F_{n,R}=F$ outside
$T^1p_\Gamma B(o,R)$ so that
$\delta_{\Gamma_{B(o,R)}}(F)=\delta_{\Gamma_{B(o,R)}}(F_{n,R})$ by
Proposition~\ref{prop:CompactPerturbPotential}. As a consequence,
\[
\delta_{\Gamma}(F_{n,R})\ge \delta_{\Gamma_{B(o,R)}}(F_{n,R})=\delta_{\Gamma_{B(o,R)}}(F)\ge \delta_\Gamma^\infty(F)\,.
\]

By the variational principle~\cite[Thm 1.1]{PPS}, we can find for all
$\epsilon>0$ a measure $\mu_{n,R,\epsilon}\in\mathcal{M}^{F_{n,R}}_1$, such
that
\[
h_{KS}(\mu_{n,R,\varepsilon})+\int_{T^1M}F_{n,R}d\mu_{n,R,\varepsilon}>\delta_\Gamma(F_{n,R})-\varepsilon\ge \delta_\Gamma^\infty(F)-\varepsilon\,.
\]

Since $F_{n,R}=F$ outside of a compact subset, $\mu_{n,R,\epsilon}$ also
belongs to $\mathcal{M}^{F}_1$. Therefore,
\begin{align*}
  \delta_\Gamma(F)& \ge h_{KS}(\mu_{n,R,\varepsilon})+\int_{T^1M}F \dd\mu_{n,R,\varepsilon}
  \\&
  \ge
  n\mu_{n,R,\epsilon} (T^1 p_\Gamma B(o,R-1))+ h_{KS}(\mu_{n,R,\varepsilon})+\int_{T^1M}F_{n,R} \dd\mu_{n,R,\varepsilon}
  \\&
  \ge n\mu_{n,R,\epsilon} (T^1 p_\Gamma B(o,R-1)) + \delta_\Gamma^\infty(F)-\varepsilon\,.
\end{align*}

Choose any sequence $\varepsilon_k\to 0$, $R_k\to\infty$, $n_k\to \infty$,
and $\mu_k=\mu_{n_k,R_k,\varepsilon_k}$. As $\delta_\Gamma(F)<\infty$, we get
from the above on the one hand that for all $R>0$,
\[
\limsup_{k\to \infty}\mu_k(T^1p_\Gamma(o,R ))=0\,,
\]
and on the other hand that
\[
\liminf_{k\to \infty} \left( h_{KS}(\mu_k)+\int F\dd\mu_k\right)\ge \delta_\Gamma^\infty(F)\,.
\]
This proves that
\[
\Pvar^\infty(F)\ge
\delta_\Gamma^\infty(F)\,.
\qedhere
\]
\end{proof}

\begin{rema}
\label{rema:Pvarerg} Since the proof only needs ergodic measures, it even
proves the slightly stronger result
\[
\delta_\Gamma^\infty(F) \leq
P_{\mathrm{var}, \mathrm{erg}}^\infty(F) \leq
P_{\mathrm{var}}^\infty(F) \,.
\]
\end{rema}

\subsection{Katok pressure}\label{ssec:ErgodicPreliminaries}

The proof of Theorem~\ref{th:PressureMassInfty} will rely on the following
notion of pressure, extending to general potentials a notion of entropy
introduced by A. Katok in~\cite{Katok80} in the case $F = 0$.

For all $v\in T^1\tilde M$ and $\epsilon,T>0$, the \emph{dynamical ball}
$B(v, \epsilon;-T, T)$ is defined by
\[
B(v, \epsilon; -T,T) = \{w\in T^1\tilde M \; ; \; \forall t\in [-T, T], d(g^t v, g^tw) \leq \epsilon\}.
\]
As in~\cite{PPS}, it is more convenient to deal with symmetric dynamical
balls. Recall from~\cite[Lemma 3.14]{PPS} that for all $0<\varepsilon\le
\varepsilon'$, there exists $T_{\varepsilon,\varepsilon'}\ge 0$, such that
for all $v\in T^1\tilde M$ and $T>0$, we have
\begin{equation*}
B(v,\varepsilon';-T-T_{\varepsilon,\varepsilon'}, T+T_{\varepsilon,\varepsilon'})
\subset B(v,\varepsilon;-T,T)\subset B(v,\varepsilon';-T,T).
\end{equation*}
As in~\cite[Rem 3.1]{ST19}, on $T^1M$, we define two kinds of dynamical
balls for $v\in T^1M, \varepsilon, T>0$: the small dynamical ball
$B_\Gamma(v,\varepsilon;-T,T)=p_\Gamma(B(\tilde v,\varepsilon;-T,T))$, where $\tilde v\in T^1\tilde M$ is a lift of $v\in T^1M$ and the
big dynamical ball
\begin{equation}\label{eqn:dyn-ball}
B_{\mathrm{dyn}}(v,\varepsilon;-T,T)=\{w\in T^1 M \; ; \; \forall t\in [-T, T], d(g^t v, g^tw) \leq \epsilon\}\supset B_\Gamma(v,\varepsilon;-T,T).
\end{equation}
Both balls coincide as soon as the injectivity radius of $M$ is bounded from
below and $\epsilon$ is small enough. More generally, if along the geodesic
$(g^t v)_{-T\le t\le T}$, the injectivity radius at all points $\pi(g^t v)$
is larger than $\varepsilon$, then

\begin{equation}\label{eqn:equality-dyn-balls}
B_{\mathrm{dyn}}(v,\varepsilon;-T,T)= B_\Gamma(v,\varepsilon;-T,T)\,.
\end{equation}

We will mainly use the small dynamical balls, that are more convenient in our
geometric context, but less natural from the dynamical point of view.

Given a probability measure $\mu$ on $T^1M$, $\delta\in (0,1)$ and
$\epsilon,T>0$,
 we will say that a set $V \subset T^1M$ is \emph{$(\mu, \delta, \epsilon;-T, T)$-spanning}, respectively \emph{dynamically-$(\mu, \delta, \epsilon;-T, T)$-spanning},
if
\[
\mu\left(\bigcup_{v\in V} B_\Gamma(v, \epsilon;-T, T) \right)\geq \delta\,,\quad\text{respectively}\quad
\mu\left(\bigcup_{v\in V} B_{\mathrm{dyn}}(v, \epsilon;-T, T) \right)\geq \delta\,.
\]
Of course, a $(\mu, \delta, \epsilon;-T, T)$-spanning set is also
dynamically-$(\mu, \delta, \epsilon;-T, T)$-spanning.

Let $F: T^1M\to \bbR$ be a Hölder-continuous potential. Let
$\mu\in\mathcal{M}^F_{1, \mathrm{erg}}$ be an ergodic probability measure on
$T^1M$, invariant under the geodesic flow, such that $\int F^-\dd\mu<\infty$.

\begin{defi} Set
\[
S_F (\mu, \delta, \epsilon;-T, T) = \inf   \sum_{v\in V} e^{\int_{-T}^T F(g^tv) \dd t},
\]
where the infimum is taken over all $V\subset T^1M$ that are $(\mu, \delta,
\epsilon;-T, T)$-spanning. Similarly define
$S^{\mathrm{dyn}}_F(\mu,\delta,\varepsilon;-T,T)$ as the infimum of the same
quantity over all dynamically- $(\mu, \delta, \epsilon;-T, T)$-spanning sets.

The \emph{Katok pressure} of $F$ with respect to $\mu$ at level $\delta$ is
defined by
\[
P^\Gamma_{\mathrm{Katok}}(\mu, F, \delta) = \sup_{\epsilon > 0} \limsup_{T\to +\infty}\frac{1}{2T} \log S_F(\mu, \delta, \epsilon; -T, T)\,.
\]
Similarly, define
\[
P^{\mathrm{dyn}}_{\mathrm{Katok}}(\mu, F, \delta) = \sup_{\epsilon>0}\limsup_{T\to +\infty}\frac{1}{2T} \log S^{\mathrm{dyn}}_F(\mu, \delta, \epsilon; -T, T)\,.
\]
The \emph{Katok pressure} of $F$ with respect to $\mu$, respectively the
\emph{dynamical Katok pressure}, is
\[
P^\Gamma_{\mathrm{Katok}}(\mu, F) = \inf_{\delta\in (0,1)}P^\Gamma_{\mathrm{Katok}}(\mu, F, \delta),
\]
respectively
\[
P^{\mathrm{dyn}}_{\mathrm{Katok}}(\mu, F) = \inf_{\delta\in (0,1)}P^{\mathrm{dyn}}_{\mathrm{Katok}}(\mu, F, \delta)\,.
\]
\end{defi}


Comparison between the two kinds of dynamical balls in~\eqref{eqn:dyn-ball}
implies the following inequality:
\[
P^{\mathrm{dyn}}_{\mathrm{Katok}}(\mu,F)\le P^\Gamma_{\mathrm{Katok}}(\mu,F)\,.
\]
The first and main inequality of Proposition~\ref{prop:EntropyKatok} below was
shown in~\cite{Katok80}. Compactness was assumed, but his proof~\cite[(1.4)
p.~144]{Katok80} does not use the compactness of the underlying manifold. The
second inequality below follows obviously from the above considerations.

\begin{prop}[Katok~\cite{Katok80}] \label{prop:EntropyKatok}  Let $\mu$  be a $g^t$-invariant ergodic probability measure. Then for all
$\delta>0$,
\[
h_{KS}(\mu) \leq h_{\mathrm{Katok}}(\mu)= P^{\mathrm{dyn}}_{\mathrm{Katok}}(\mu,0)\le P^\Gamma_{\mathrm{Katok}}(\mu,0)\,.
\]
\end{prop}

The appendix by F. Riquelme  shows that, in the case of geodesic flows on
manifolds in negative curvature, these entropies coincide, even in our
non-compact setting, cf.\ Theorem~\ref{theo:entropies-coincide}.

In the sequel, we will always work with small dynamical balls and the
associated Katok pressure $P^\Gamma_{\mathrm{Katok}}(\mu,F)$. Assume that
$\mu$ is ergodic.

For all $A\subset T^1M$, all $\delta\in (0,1)$ and all $\epsilon,T>0$, we
define
\[
S_{F,A}(\mu, \delta, \epsilon;-T, T) =
\inf_{V\subset A \; (\mu, \delta, \epsilon; -T,T)\text{-spanning}} \sum_{v\in V} e^{\int_{-T}^T F(g^tv) \dd t}
\]
and
\[
P_{\mathrm{Katok}}^A(\mu, F,\delta) = \sup_{\epsilon>0}\limsup_{T\to +\infty}\frac{1}{2T} \log S_{F,A}(\mu, \delta, \epsilon; -T, T)\,.
\]

The following lemma is elementary but crucial in the sequel.
\begin{lemm}\label{lem:KatokrestreintaA} Under the assumptions of Theorem~\ref{th:ErgoPressure}, let
$\mu\in\mathcal{M}_{1,\mathrm{erg}}^F$ be an ergodic invariant measure. As
soon as $\mu(A)>\delta$ we have
\begin{equation*}
P^\Gamma_{\mathrm{Katok}}(\mu,F, \delta) \leq P_{\mathrm{Katok}}^A(\mu,F,\delta).
\end{equation*}
Moreover, if $\mu(A)\geq 1 - \frac \delta 6$, and $F$ is bounded on $A$, then
\begin{equation}\label{eq:PKatokTypical2}
P^\Gamma_{\mathrm{Katok}}(\mu,F, \delta) \geq P_{\mathrm{Katok}}^A(\mu,F,\frac \delta 2).
\end{equation}
\end{lemm}

\begin{proof} The first inequality is immediate from the definition.

For the second one, let $A' = A \cap g^{-T}A\cap g^T A$. It satisfies
$\mu(A') \geq 1-\delta/2$. Consider $V$ a
$(\mu,\delta,\varepsilon;-T,T)$-spanning set. As $\mu(\bigcup_{v\in V}
B_\Gamma(v,\varepsilon;-T,T))\ge \delta$, we get $\mu(A'\cap \bigcup_{v\in V}
B_\Gamma(v,\varepsilon;-T,T))\ge \delta/2$. For every $v\in V$ such that
$\mu(A'\cap B_\Gamma(v,\varepsilon;-T,T))>0$, choose an element $v'$ in the
intersection $A'\cap B_\Gamma(v,\varepsilon;-T,T)$, and let $V'$ be the set
of all such $v'$. By construction, $V'\subset A$ is a
$(\mu,\delta/2,2\varepsilon;-T,T)$-spanning set.

As $F$ is Hölder-continuous, for $v\in V$ such that $\mu(A'\cap
B_\Gamma(v,\varepsilon;-T,T))>0$ and $v'\in A'\cap
B_\Gamma(v,\varepsilon;-T,T)$, the integrals $\int_{-T}^T F\circ g^t v \dd t$
and $\int_{-T}^T F\circ g^t v'\dd t$ differ at most by an additive constant
depending on the Hölder constants of $F$, and its $L^\infty$-norm on the
$\epsilon$-neighborhood of $A$, but not on $T$. Indeed, as $F$ is bounded on
$A$, and Hölder-continuous, it is also bounded on the
$\varepsilon$-neighborhood of $A$. Moreover, by definition of $A'$, $g^{\pm
T}v'\in A$, so that $g^{\pm T}v$ belong to the $\varepsilon$-neighborhood of
$A$. Moreover, $d(g^Tv,g^Tv')\le\varepsilon$ and
$d(g^{-T}v,g^{-T}v')\le\varepsilon$. Thus, Lemma~\ref{lm:hold-potential}
applies and gives the desired bound.


Therefore, up to a multiplicative constant, $\sum_{v\in V}e^{\int_{-T}^T
F(g^tv) \dd t}$  is greater than $\sum_{v'\in V'} e^{\int_{-T}^T F(g^tv')\dd
t}$. Up to this multiplicative constant, $S_F(\mu, \delta, \epsilon;-T, T)$
is greater than $S_{F,A}(\mu, \delta/2, 2\epsilon;-T, T)$. Taking the limsup
of $1/(2T)\log$ of these quantities leads to the second inequality.
\end{proof}

Since the Katok pressure is defined by taking an infimum over all $(\mu,
\delta, \epsilon;-T, T)$-spanning sets, we deduce the following useful
statement.

\begin{lemm}\label{WeakPKatok} Under the assumptions of Theorem~\ref{th:ErgoPressure}, let
$\mu\in\mathcal{M}_{1,\mathrm{erg}}^F$ be an ergodic probability measure. Let
$\delta>0$ and $\epsilon>0$ be fixed, and for all $T>0$, let $A_T\subset
T^1M$ be a set such that $\mu(A_T)>\delta$. Then
\[
P^\Gamma_{\mathrm{Katok}}(\mu, F) \leq  \limsup_{T\to +\infty}\frac 1 {2T} \log S_{F,A_T}(\mu, \delta, \epsilon; -T, T).
\]
\end{lemm}

We will use the following analogue of Proposition~\ref{prop:EntropyKatok} for
general potentials.

\begin{prop}\label{prop:PressureKatok} Under the assumptions of Theorem~\ref{th:ErgoPressure},
let $F: T^1M\to \bbR$ be a Hölder-continuous map, and $\mu\in \mathcal{M}_{1,
\mathrm{erg}}^F$ be an ergodic probability measure  on $T^1M$ such that $\int
F^-\dd\mu<\infty$. Then
\[
h_{KS}(\mu) + \int_{T^1M} F \dd\mu \leq P^\Gamma_{\mathrm{Katok}}(\mu, F)\,.
\]
\end{prop}

\begin{proof}
Let $\mu\in \calM_{1, \mathrm{erg}}^F$ be an ergodic probability measure and $F$ a Hölder-continuous
potential. Let $\delta\in (0,1)$ be fixed.

For all $\eta>0$ and $T>0$, set
\[
G_{T, \eta}(F) = \left\{v\in T^1M \; ; \; \forall t\geq T,\; \abs*{\frac 1 {2t} \int_{-t}^t F(g^s v) ds - \int F \dd\mu }\leq \eta\right\}.
\]
Birkhoff ergodic theorem implies that for all $\eta>0$, we have
$\displaystyle \lim_{T\to +\infty}\mu(G_{T, \eta}(F)) = 1$. Therefore there
exist $T_0>0$ and a compact subset $A_{\delta, \eta}\subset G_{T_0, \eta}(F)$
such that $\mu(A_{\delta, \eta})> 1 - \frac \delta 6$. Therefore,
by~\eqref{eq:PKatokTypical2},
\begin{equation}\label{eq:PKatokMin1}
P^\Gamma_{\mathrm{Katok}}(\mu,F, \delta) \geq P_{\mathrm{Katok}}^{A_{\delta, \eta}}(\mu,F,\frac \delta 2) =
\limsup_{T\to +\infty}\frac{1}{2T} \log \inf_{V\subset A_{\delta, \eta} \; (\mu, \delta/2, \epsilon;-T, T)\text{-spanning}} \sum_{v\in V} e^{\int_{-T}^T F(g^tv) \dd t}.
\end{equation}
Let $\calS_T\subset A_{\delta, \eta}$ be a finite  $(\mu,
\delta/2,\epsilon;-T, T)$-spanning set. As $A_{\delta, \eta} \subset G_{T_0,
\eta}(F)$ and thanks to the definition of $G_{T_0, \eta}(F)$, we have for $T
\ge T_0$
\[
\sum_{v\in \calS_T} e^{\int_{-T}^T F(g^tv) \dd t} \geq e^{2T (\int F \dd\mu - \eta)}\# \calS_T
\geq  e^{2T (\int F \dd\mu - \eta)} \inf \# V,
\]
the infimum being taken over all $(\mu, \delta/2, \epsilon, T)$-spanning sets
$V$.

Minimizing over $\calS_T$, Equation~\eqref{eq:PKatokMin1} leads to
\[
P^\Gamma_{\mathrm{Katok}}(\mu,F, \delta) \geq \int F \dd\mu - \eta + P^\Gamma_{\mathrm{Katok}}(\mu, 0, \delta/2)\,.
\]
Together with Proposition~\ref{prop:EntropyKatok}, this concludes the proof
of Proposition~\ref{prop:PressureKatok} since $\delta\in (0,1)$ and $\eta>0$
can be arbitrarily small.
\end{proof}

\medskip


\subsection{Escape of mass and pressure at infinity}\label{sec:escape-of-mass}

This paragraph is devoted to the proof of the following result, of
independent interest, which implies Corollary~\ref{coro:PressureMassInfty}, a
key step in the proof of Theorem~\ref{th:ErgoPressure}.

\begin{theo}\label{th:PressureMassInfty}
Let $K\subset M$ be a compact set whose interior intersects $\pi(\Omega)$,
and let $\tilde K\subset\tilde M$ be a compact subset such that
$p_\Gamma(\tilde K)=K$. Let $F : T^1 \to \bbR$ be a Hölder-continuous
potential with $\delta_{\Gamma_{\tilde K}}(F)
> -\infty$. Let $\eta >0$. For all $0< \alpha \leq 1$ and $R\geq 4$,
there exists a positive number $\psi = \psi(\tilde K, F, \eta, \alpha/R)$
with the following property. For every
$\mu\in\mathcal{M}_{1, \mathrm{erg}}^F$ with $\mu(T^1K_R)\le \alpha$, we have
\begin{equation*}
h_{KS}(\mu) + \int_{T^1M} F \dd\mu
\leq (1-\alpha)\delta_{\Gamma_{\tilde K}}(F) + \alpha\delta_\Gamma(F) + \eta + \psi.
\end{equation*}
Moreover, when $\tilde K, F$ and $\eta$ are fixed, $\psi(\tilde K, F, \eta,
\alpha/R)$ tends monotonically to $0$ when $\alpha/R$ tends to $0$.
\end{theo}

Letting $K$ grow to exhaust $ M$, we deduce the following corollary, which
provides the second half of Theorem~\ref{th:ErgoPressure} (the first
inequality $\delta_\Gamma^\infty(F)\le \Pvar^\infty(F)$ has been proved in
Proposition~\ref{prop:first-inequality}).
\begin{coro}\label{coro:PressureMassInfty}
Let $F$ be a Hölder-continuous potential on $T^1M$. Let $(\mu_n)_{n\geq
0}\in(\mathcal{M}_{1}^F)^{\bbN}$ be
 a sequence of probability measures which converges in the vague topology to a measure $\mu$. Then
\[
\limsup_{n\to +\infty} \left( h_{KS}(\mu_n) + \int F \dd\mu_n  \right) \leq (1-\norm{\mu})\delta_\Gamma^\infty(F) + \norm{\mu} \delta_\Gamma(F).
\]
In particular, when $\mu_n \overset{\ast}{\rightharpoonup} 0$, then $\displaystyle
\limsup_{n\to +\infty} \left( h_{KS}(\mu_n) + \int F \dd\mu_n \right) \leq
\delta_\Gamma^\infty(F)$, so that
\[
\Pvar^\infty(F) \leq \delta_\Gamma^\infty(F).
\]
\end{coro}
\begin{proof}
When $\delta_\Gamma(F)=\infty$, then $\delta_\Gamma^\infty(F)=\infty$ by
Proposition~\ref{prop:infinite_pressure}, and the result is obvious. We can
therefore assume that $\delta_\Gamma^\infty(F)<\infty$. We will deal with the
case $\delta_\Gamma^\infty(F)>-\infty$, as the case
$\delta_\Gamma^\infty(F)=-\infty$ can be treated similarly.

Let $\epsilon>0$. Let $K$ be a large compact subset of $M$, and $\tilde K$ a
compact subset of $\tilde M$ satisfying $p_\Gamma(\tilde K)=K$ and
$\delta_{\Gamma_{\tilde K}}(F) \leq \delta^\infty_\Gamma(F)+\epsilon$ and
$\norm{\mu}\leq \mu(T^1 K)+\epsilon$. There are only countably many values of
$r$ for which $\mu(\partial T^1 K_r)$ has positive measure as these sets are
disjoint. Therefore, we can pick $r$ such that $\mu(\partial T^1 K_r) = 0$.
Replacing $K$ with $K_r$, we can assume $\mu(\partial T^1 K)=0$.

We apply Theorem~\ref{th:PressureMassInfty} to $\eta=\epsilon$, obtaining a
function $\psi$. Let $R$ be large enough so that $\psi(1/R)\leq \epsilon$. We
can also ensure that $\mu(\partial T^1 K_R)=0$. For large enough $n$, we have
$\mu_n(T^1 K) \geq \mu(T^1 K)-\epsilon$ and $\mu_n(T^1 K_R) \leq \mu(T^1
K_R)+\epsilon \leq \norm{\mu}+\epsilon$. In particular, $\mu_n(T^1 K_R) \geq
\mu_n(T^1 K) \geq \norm{\mu}-2\epsilon$. Let us estimate $h_{KS}(\mu_n) +
\int F \dd\mu_n$ for such an $n$, fixed from now on.

We can write $\mu_n$ as an average of ergodic measures: $\mu_n = \int_\Omega
 \nu_\omega \dd\bbP(\omega)$, where all the $\nu_\omega$ are invariant
probability measures for $g_t$. Since $\infty>\int F^- \dd\mu_n = \int (\int
F^- \dd\nu_\omega) \dd\bbP(\omega)$, almost all the measures $\nu_\omega$
belong to $\calM_{1,\mathrm{erg}}^F$. The entropy of a convex combination of
probability measures is given by~\cite[Proposition 4.3.16 (2)]{HK}. We can
therefore apply Theorem~\ref{th:PressureMassInfty} to each of the
$\nu_\omega$ (with $\alpha = \nu_\omega(T^1 K_R)$) and then average with
respect to $\bbP$, yielding
\begin{align*}
  h_{KS}(\mu_n) + \int F \dd\mu_n
  &= \int \pare*{h_{KS}(\nu_\omega) + \int F \dd\nu_\omega} \dd\bbP(\omega)
  \\& \leq \int \pare*{(1- \nu_\omega(T^1 K_R)) \delta_{\Gamma_{\tilde K}}(F) +
\nu_\omega(T^1 K_R) \delta_\Gamma(F) + \epsilon + \psi(1/R)} \dd\bbP(\omega)
  \\& = (1 - \mu_n(T^1 K_R)) \delta_{\Gamma_{\tilde K}}(F) +
\mu_n(T^1 K_R) \delta_\Gamma(F) + \epsilon + \psi(1/R)
  \\& \leq (1 - \norm{\mu} + 2\epsilon)(\delta^\infty_\Gamma(F)+\epsilon) + (\norm{\mu} + \epsilon) \delta_\Gamma(F) + 2\epsilon.
\end{align*}
As $\epsilon$ is arbitrary, this gives the conclusion.
\end{proof}

Let us point out that when $F=0$  and $M$ is geometrically finite, under the
same hypotheses,   a stronger version of
Corollary~\ref{coro:PressureMassInfty} appears
in~\cite[Thm.~1.1]{Riquelme-Velozo}:
\[
\limsup_{n\to +\infty} h_{KS}(\mu_n) \leq (1-\norm{\mu})\delta_\Gamma^\infty(0) + \norm{\mu} h_{KS}\left(\frac \mu {\norm{\mu}}\right).
\]
In~\cite{Velozo-phd,Velozo}, Velozo announces an analogous inequality for
pressure on general negatively curved manifolds,   in the case of potentials
going to $0$ at infinity. Our approach is valid for all Hölder-continuous
potentials, but gives a weaker inequality. However, it provides enough
information for our purposes.

\begin{coro}
\label{cor:Perg_infty} The pressures $\Pvar^\infty(F)$ and $P_{\mathrm{var},
\mathrm{erg}}^\infty(F)$ are equal.
\end{coro}

\begin{proof}
We have obviously the inequality $P_{\mathrm{var}, \mathrm{erg}}^\infty(F)
\leq \Pvar^\infty(F)$. Moreover, $\Pvar^\infty(F) \leq
\delta_\Gamma^\infty(F)$ by Corollary~\ref{coro:PressureMassInfty}. Finally,
Remark~\ref{rema:Pvarerg} gives the inequality $\delta_\Gamma^\infty(F)\leq
P_{\mathrm{var}, \mathrm{erg}}^\infty(F)$. Together, these inequalities show
that all these quantities coincide.
\end{proof}

\begin{proof}[Proof of Theorem~\ref{th:PressureMassInfty}]
As the result is obvious if $\delta_\Gamma(F)=\infty$, we may assume that
$\delta_\Gamma(F)<\infty$. Let $K\subset  M$ be a compact subset, $R>0$, and
$K_R$ the $R$-neighborhood of $K$.
Let $\eta>0$.

Let $\mu\in\mathcal{M}^F_{1, \mathrm{erg}}$ be an ergodic probability measure
on $T^1M$, and $0<\alpha\leq 1$ such that $\mu(T^1K_R) \leq \alpha$. Let
$\epsilon>0$ be small enough (how small exactly will be prescribed at the end
of the proof).

Let $A $ be a large compact subset containing $K_R$, with $\mu(T^1A)>
1-\epsilon$. Let $T_0$ be the constant given by Assertion 1 of   Proposition~\ref{lem:connecting} (Connecting lemma) applied with $A$ and $K$ in the role of $K$ and $K'$. Define
\begin{equation*}
A_T= \Bigl\{w\in T^1A, \, \abs*{\frac{1}{2T}\int_{-T}^T F\circ g^t w\dd t-\int F\dd\mu}\le \varepsilon\quad\text{and}\quad
  \frac{1}{2T}\int_{-T}^T {\bf 1}_{T^1K_R}(g^t w)\dd t\le \alpha+\epsilon\,\Bigr\}\,.
\end{equation*}
By Birkhoff ergodic Theorem, there exists $T_1>0$ such that for $T\ge T_1$,
$\mu(A_T)\ge 1-\epsilon$. Then
\[
\mu(A_T\cap g^{T+T_0} T^1A\cap g^{-T-T_0} T^1A)\ge 1-3\epsilon\,.
\]
The strategy is to bound
\[
h_{KS}(\mu)+\int F\dd\mu
\]
from above, in terms of periodic orbits, and use
Theorem~\ref{th:CountExcursion} to prove Theorem~\ref{th:PressureMassInfty}.

Consider a maximal $(\varepsilon;-T,T)$-separated subset $V$ of
$\mathcal{A}_{T}=A_T\cap g^{T+T_0} T^1 A\cap g^{-T-T_0} T^1 A$ in the sense
that the small dynamical balls $B_\Gamma(v;\varepsilon,-T,T)$, for $v\in V$,
are pairwise disjoint. By maximality,
$\mathcal{A}_{T} \subseteq \bigcup_{v \in V} B_\Gamma(v,2\epsilon;-T,T)$.
Therefore, $V$ is also a $(\mu, \delta, 2\epsilon; -T,T)$ spanning set for
any $\delta\leq 1-3\varepsilon$. Proposition~\ref{prop:PressureKatok} and
Lemma~\ref{WeakPKatok} ensure that $h_{KS}(\mu) + \int_{T^1M} F \dd\mu$ is
bounded from above by the exponential growth rate of the sums $\sum_{v\in V}
e^{\int_{-T}^T F(g^t v)\,dt}$, over such sets $V$ (that depend  implicitly on $T$).

Now, to each $v\in V$, we will associate a periodic orbit and bound the above
sum in terms of $\calN_F(K, K_R, \alpha; T-\tau, T+\tau)$ for some constant
$\tau>0$.

Take $v\in V$. As it belongs to $\mathcal{A}_T$, both points $g^{T+T_0} v$
and $g^{-T-T_0}v$ belong to $T^1 A$. By the Connecting Lemma
(Proposition~\ref{lem:connecting}) applied to the sets $A$ and $K$, we deduce
the existence of a periodic vector $v_p$, and associated periodic orbit
$p(v)$, with $\abs{\ell(p(v))-2T}\le 3T_0$, and $d(g^t v_p,
g^tv)\le\varepsilon/3$ for all $-T\le t \le T$. Note that we have used a
longer orbit to start with since  Proposition~\ref{lem:connecting}  only
gives a good distance control $T_0$-away from the endpoints of the original
geodesic. Since the interior of $K$ intersects $\pi(\Omega)$, it also
follows from Proposition~\ref{lem:connecting} that we can also require that
the orbit $p(v)$ intersects $K$.

By Lemma~\ref{lm:hold-potential}, $\int_{0}^{\ell(p(v))} F(g^t v_p)\dd t$ is
close to $\int_{-T}^{T}F(g^t v)\dd t$ up to a constant depending only on
$A$ and $F$. Since $v\in \mathcal{A}_T$, the latter integral is close to $2T\int F\dd\mu$, up
to $2T\epsilon$. Altogether, we get
\begin{equation*}
\abs*{\int_{0}^{\ell(p(v))} F(g^t v_p)\dd t-\ell(p(v))\int F\dd\mu} \leq C_0 + \ell(p(v)) \epsilon,
\end{equation*}
for some $C_0$ depending only on $A$, $\varepsilon,F$. In particular, there exists $T_3$ such
that for $T\ge T_3$, $\ell(p(v))$ is also large, so that this inequality
becomes
\[
\abs*{\frac{1}{\ell(p(v))}\int_{0}^{\ell(p(v))} F(g^t v_p)\dd t- \int F\dd\mu} \le 2\varepsilon\,.
\]
Similarly, we obtain, for $T$ large enough,
\[
\frac{\ell(p(v)\cap T^1 K_{R/2})}{\ell(p(v))} \le \alpha+2\varepsilon,
\]
starting from the same properties for the orbit of $v$ due to the definition
of $\mathcal{A}_T$, and using the fact that the orbits of $v$ and $v_p$
remain close to each other up to $\epsilon$, so the orbit of $v_p$ can be in
$K_{R/2}$ only at times when the orbit of $v$ is in $K_R$, except for times
in a bounded interval.

Moreover, as the set $V$ is $(\varepsilon;-T,T)$ separated, and the periodic
orbit $p(v)$ associated with each $v\in V$ is $\varepsilon/3$-close to it, two
parameterized orbits $p(v)$ and $p(v')$ are separated by at least
$\epsilon/3$ when $v\neq v'$. Therefore, the multiplicity of $v \mapsto p(v)$ can only come
from different choices of basepoints on the resulting orbit, separated by at
least $\epsilon/3$, plus orbifold multiplicities bounded by the compactness of $A$. Hence, this multiplicity is bounded by some
multiplicative constant times $T$.

Therefore, up to a multiplicative constant, $\sum_{v\in V} e^{\int_{-T}^T
F\circ g^t v\dd t} $ is bounded by
\[
T\calN_F(K, K_{R/2}, \alpha+2\epsilon,2T-\tau,2T+\tau)\,,
\]
for some $\tau>0$ depending on $A,F,\varepsilon,T_0$ but independent of $T$. Let $\tilde K$ be a compact set such that $p_\Gamma(\tilde K) = K$. Applying
Theorem~\ref{th:CountExcursion} with $\eta/2$ and $R/2$, we
get that its exponential growth rate is bounded by
\begin{equation*}
  (1-\alpha-2\epsilon) \delta_{\Gamma_{\tilde{K}}}(F) + (\alpha+2\epsilon) \delta_\Gamma(F) + \eta/2
  + \psi\pare*{\frac{\alpha+2\epsilon}{R/2}}
\end{equation*}
where $\psi$ is a function tending to $0$ at $0$. If $\epsilon$ is small
enough, say $\epsilon \leq \epsilon_0$, then the error term $2\epsilon
(\delta_\Gamma(F)-\delta_{\Gamma_{\tilde{K}}}(F))$ is bounded by $\eta/2$,
and we get a bound
\begin{equation*}
  (1-\alpha)\delta_{\Gamma_{\tilde K}}(F) + \alpha\delta_\Gamma(F) + \eta + \psi\pare*{\frac{\alpha+2\epsilon}{R/2}}.
\end{equation*}
Finally, we choose $\epsilon = \alpha \epsilon_0$, so that
$(\alpha+2\epsilon)/(R/2)$ is a function of $\alpha/R$ that tends to $0$ when
$\alpha/R$ tends to $0$. This is the desired bound.
\end{proof}





\section{Strong positive recurrence}\label{sec:SPR}

In symbolic dynamics, the notion of strong positive recurrence appeared in
several works, as mentioned in the introduction, see for
example~\cite{Gurevic,Gurevic2,GS,Sa99,Sa01,Ruette,BBG06, Buzzi, BBG}. In our
geometric context, when $F=0$, the notion appeared in~\cite{ST19,CDST} under
the terminology of ``strongly positively recurrent manifold'' or ``strongly
positively recurrent action''. Independently, it appeared (still in the case
$F=0$) among people interested by geometric group theory, see for
example~\cite{ACT15,WY,WY19}, under the name of ``actions with a growth gap''
or later ``statistically convex-cocompact manifolds''. We follow the ergodic
terminology of {\em strong positive recurrence} below, extending the point of
view developed in~\cite{ST19}, in the spirit of the works of symbolic
dynamics.

\subsection{Different notions of recurrence}\label{ssec:SPR}

Recall some definitions which are classical in symbolic dynamics, and were
introduced for the geodesic flow in negative curvature in~\cite{PS16, ST19}.
Recall that $\mathcal{P}$ and $\mathcal{P}'_K$ have been defined in
Paragraph~\ref{subsubsec:Gurevic}, and $n_{\tilde K}$ after
Remark~\ref{rmk:closing_hard}.

\begin{defi} A Hölder-continuous potential $F:T^1M\to \bbR$ with finite topological pressure is said to be
\begin{enumerate}
\item {\em recurrent} if there exists a compact subset $K\subset M$ whose
    interior intersects the projection $\pi(\Omega)$ of the nonwandering
    set, with a compact lift $\tilde K$ to $\tilde M$ such that
\[
\sum_{p\in\mathcal{P}} n_{\tilde K}(p)e^{\int_p (F-\delta_\Gamma(F))}\,=+\infty\,;
\]
\item {\em positively recurrent} if it is recurrent with respect to some
    compact subset $K\subset  M$ whose interior intersects  $\pi(\Omega)$,
    with a lift $\tilde K$ to $\tilde M$, and for some $N\ge 1$,
\[
\sum_{p\in\mathcal{P}'_K,\,n_{\tilde K}(p)\le N} \ell(p)e^{\int_p (F-\delta_\Gamma(F))} <+\infty\,;
\]
\item {\em strongly positively recurrent} if its pressure at infinity
    satisfies
\[
\Ptop^\infty(F)<\Ptop(F)\,.
\]
\end{enumerate}
\end{defi}

Let us introduce another quantitative notion of recurrence,  which involves an invariant measure.

Let $K\subset M$ be a compact subset, $\tilde K\subset \tilde M$ a compact
subset such that $p_\Gamma(\tilde K)=K$. For all $T>0$ large enough, as
in~\cite{ST19}, we define\footnote{In~\cite{CDST}, the definition has been slightly modified to
guarantee that it remains open when $\tilde M$ is a Gromov-hyperbolic metric
space}  $U_T(\tilde K)\subset\tilde M$
as the open set
\[
U_T(\tilde K)=\{y\in\tilde M\cup\partial\tilde M,\,\,
\exists x\in\tilde K, \,\, d(x,y)>T\,\,\textrm{and}\,\,[x,y)_T\cap \Gamma\tilde K\subset \tilde K\,\}\,,
\]
where $[x,y)_T$ denotes the geodesic segment  of length $T$ starting from $x$
on $[x,y)$. When we write this, we implicitly require $d(x,y)>T$. In other
words, $y\in U_T(\tilde K)$ if there exists some geodesic $[x,y)$ starting in
$\tilde K$ and ending at $y$, which does not meet $\Gamma\tilde
K\setminus\tilde K$ until time $T$.

For technical reasons, we will need to work with the following slightly
larger sets:
\[
U_{T_0,T}(\tilde K)=\{y\in\tilde M\cup\partial\tilde M,\,\,
\exists x\in\tilde K,  \,\, d(x,y)>T\,\,\textrm{and}\,\,[x,y)_{[T_0,T]}\cap \Gamma\tilde K\subset \tilde K\,\}\,,
\]
where $[x,y)_{[T_0,T]}$ denotes the geodesic segment  of length $T-T_0$
starting  at distance $T_0$ from $x$ on $[x,y)$ (assuming again $d(x,y)>T\ge
T_0$). In other words, $y\in U_{T_0,T}(\tilde K)$ if there exists some
geodesic $[x,y)$ starting in $\tilde K$ and ending at $y$, which does not
meet $\Gamma\tilde K\setminus\tilde K$ between times $T_0$ and $T$.

Let us define $\tilde V_T(\tilde K)\subset T^1\tilde K$ (resp.\ $\tilde
V_{T_0,T}(\tilde K)\subset T^1 \tilde K$) as the set of unit vectors tangent
at $x$ to a geodesic segment $[x,y)$, for some $y\in U_{T}(\tilde K)$ (resp.\
$U_{T_0,T}(\tilde K)$) and $x$ associated with $y$ as above. Finally, let
$V_T(\tilde K) = p_\Gamma(\tilde V_T(\tilde K))$ and $V_{T_0,T}(\tilde K) =
p_\Gamma(V_{T_0, T})(\tilde K))$.

All these sequences $(U_T(\tilde K))_{T>0}$, $(U_{T_0,T}(\tilde K))_{T>T_0}$,
$(\tilde V_T(\tilde K))_{T>0}$, $(\tilde V_{T_0,T}(\tilde K))_{T>T_0}$,
$(V_T(\tilde K))_{T>0}$ and $(V_{T_0,T}(\tilde K))_{T>T_0}$ are nonincreasing
when $T\to \infty$.

\begin{defi}
The geodesic flow $(g^t)$ is said to be {\em exponentially recurrent} with
respect to an invariant (non-necessarily finite) measure $m$ if there exist a
compact subset $K\subset M$ whose interior intersects $\pi(\Omega)$ and some
compact subset $\tilde K \subset\tilde M$ with $p_\Gamma(\tilde K)=K$ such that, for
all $T_0\ge 0$, there exist $C>0$ and $\alpha>0$ such that for $T\ge T_0$,
\[
m(V_{T_0,T}(\tilde K))\le C\exp (-\alpha T)\,.
\]
\end{defi}

In~\cite[Thms 1.2, 1.4 and 1.6]{PS16}, the following result, reformulated here
thanks to Theorem~\ref{theo:HTS}, is proven.

\begin{theo}[Pit-Schapira]\label{theo:Pit-Schap}
Let $F:T^1M\to \bbR$ be a Hölder-continuous map with finite topological pressure.
\begin{enumerate}
\item The potential $F$ is recurrent if and only if $(\Gamma,F)$ is divergent,  if and only if
    $m^F$ is ergodic and conservative
\item The potential $F$ is positively recurrent  if and only if  $m^F$ is finite.
\item The potential $F$ is positively recurrent  if and only if  it is
    recurrent and there exists a compact subset   $K\subset M$ whose
    interior intersects at least a closed geodesic, and   a compact subset
    $\tilde K\subset \tilde M$ with $p_\Gamma(\tilde K)=K$, such that
\[
\sum_{\gamma\in\Gamma_{\tilde K}} d(o,\gamma o)e^{-\delta_\Gamma(F)d(o,\gamma o)+\int_o^{\gamma o} \tilde F}< + \infty\,.
\]
\end{enumerate}
\end{theo}

In Section~\ref{sec:SPR-implique-PR}, we will prove the following result.

\begin{theo}\label{theo:SPR-implies-PR'}
Let $F:T^1M\to \bbR$ be a Hölder-continuous map with finite topological pressure. If $F:T^1M\to \bbR$ is
strongly positively recurrent, then it is positively recurrent.
\end{theo}
This Theorem has been proven in~\cite{ST19} in the case $F\equiv 0$, and the
proof is almost the same. We provide it here  for the sake of completion and the comfort of the reader.

The contrapositive reformulation is extremely useful:\\
\centerline{ \bf If the measure $m^F$ is infinite, then
$\delta_\Gamma^\infty(F)=\delta_\Gamma(F)$.}

It has the following corollary.

\begin{coro}\label{prop:GaloisCover}
Let $F:T^1M\to \bbR$ be a Hölder-continuous map with finite topological pressure. Let $p : \overline{M}\to M$
be an infinite Riemannian connected Galois cover of $M$, and
  $H = \pi_1(\bar{M}) \triangleleft \Gamma = \pi_1(M)$. Let $\overline{F}=F\circ dp:T^1\overline{M}\to \bbR$ be the lift of $F$ to $T^1\overline{M}$.
Then
\[
\delta_H^\infty(\overline{F}) = \delta_H(\overline{F}) \leq \delta_\Gamma(F).
\]
\end{coro}
\begin{proof}
The inequality $\delta_H(\overline{F}) \leq \delta_\Gamma(F)$ is immediate
since $H \subset \Gamma$. By contradiction, assume that
$\delta_H^\infty(\overline{F}) < \delta_H(\overline{F})$. Then the potential
$\overline{F}$ would be strongly positively recurrent. By
Theorems~\ref{theo:SPR-implies-PR'} and~\ref{theo:Gibbs}, once renormalized
into a probability measure, the associated equilibrium measure $m^F$ is
finite and unique. By uniqueness, the measure $m^F$ is invariant under the
action of the deck group $G = \Gamma/H$. As $G$ is infinite by hypothesis, it
is a contradiction with the finiteness of $m^F$.
\end{proof}

\begin{rema}This corollary  does not apply to non-regular cover, even
for the zero potential. For example, consider the following construction.
Given  $\Sigma_\Gamma = \mathbb H^2/\Gamma$   a compact genus 2 hyperbolic
surface, there exists $H<\Gamma $ a non-normal subgroup such that
$\Sigma_H=\mathbb H^2/H$ is a punctured torus with infinite volume. The
(non-regular) covering $p : \Sigma_H = \mathbb H^2/H \to \Sigma_\Gamma$ does
not satisfy the conclusion of the above corollary. Indeed,  $\Sigma_H$ is
convex cocompact, non elementary, with infinite volume. In particular, there
exists a large compact subset $\tilde K\subset\mathbb H^2$ such that
$H_{\tilde K}$ is finite, so that
 \[
 \delta_H(0) > 0 \text{ and } \delta_H^\infty(0) = - \infty.
 \]
\end{rema}

\begin{coro}\label{coro:exposant-infini}
There exists a complete hyperbolic surface $M$, with $\delta_\Gamma^\infty(0)
> 0$, and a Hölder-continuous potential $F: T^1M\to \bbR$ such that
$\delta_\Gamma^\infty(F) = -\infty$.
\end{coro}

Observe that if $\delta_\Gamma^\infty(0)>-\infty$, then it is nonnegative
and every Hölder-continuous potential $F$ which is bounded from below by some
constant $-K$ satisfies $\delta_\Gamma^\infty(F)\geq -K$. Therefore examples
satisfying Corollary~\ref{coro:exposant-infini} must be unbounded from below.

\begin{proof}
Let $M=\mathbb{H}^2/\Gamma$ be a $\bbZ$-cover of a compact hyperbolic
surface. By Corollary~\ref{prop:GaloisCover}, $\delta_\Gamma^\infty(0) =
\delta_\Gamma(0)>0$. It is well known that $\delta_\Gamma(0) = 1$ (it follows
for instance from~\cite{Bro85}, see for instance~\cite{CDST} for details on
critical exponents of covers).  Choose some compact fundamental domain
$D\subset M$ with piecewise smooth boundary for the action of the deck group
$G = < g^n \; ; \; n\in \bbZ>$. For all $n\in \bbZ$, set $D_n = g^n D$. Build
a Hölder-continuous map  $F: T^1M \to \bbR$   such that for all $n\in
\bbZ\backslash \{0\}$ and  $v\in T^1 D_n$, we have $-\abs{n} \leq F(v) \leq
-(\abs{n}-1)$. Considering compact subsets $\tilde K_N$ with $p_\Gamma(\tilde
K_N)=\bigcup_{\abs{n}\le N}D_n$, we have $\delta_{\Gamma_{\tilde
K_N}}(F)\le\delta_{\Gamma_{\tilde K_N}}(0)-N$, so that
$\delta_\Gamma^\infty(F) = -\infty$.
\end{proof}

In fact, this construction applies whenever the non-wandering set $\Omega$ is
non-compact, for the potential $F(v) = - d(\pi (v), p_\Gamma(o))$ for instance.

The next two theorems are proved respectively in Subsections~\ref{exp-rec}
and~\ref{sec:SPR-ind-compact}.

\begin{theo}\label{theo:SPR-equiv-exprec}
Let $F:T^1M\to \bbR$ be a Hölder-continuous map with finite topological pressure. The potential $F$ is
strongly positively recurrent if and only if  the geodesic flow is exponentially
recurrent with respect to the measure $m^F$ given by the
Patterson-Sullivan-Gibbs construction.
\end{theo}

The last result that we shall prove provides a very satisfying information on
strongly positively recurrent potentials. We will not use it in this paper.

\begin{theo}\label{theo:indep-compact}
Let $F:T^1M\to \bbR$ be a Hölder-continuous map with finite topological pressure. If $F:T^1M\to \bbR$ is
strongly positively recurrent, then for \emph{every} compact set $\tilde
K\subset\tilde M$, whose interior intersects $\pi(\Omega)$, we have
$\displaystyle \delta_{\Gamma_{\tilde K} }(F)<\delta_\Gamma(F)\,$.
\end{theo}

In fact, the proof of Theorem~\ref{theo:SPR-equiv-exprec} shows that, for any
compact subset $K$ with $\delta_{\Gamma_{\tilde K}}(F) < \delta_\Gamma(F)$,
the sets $V_{T_0,T}(\tilde K)$ have exponentially small $m^F$-measure for all
$T\ge T_0$. Together with Theorem~\ref{theo:indep-compact}, this implies the
following corollary.
\begin{coro}\label{coro:expo-rec-indep-compact}
Let $F:T^1M\to \bbR$ be a Hölder-continuous map  with finite topological pressure and with finite Gibbs measure
$m^F$. Assume that the geodesic flow is exponentially recurrent with respect
to $m^F$. Then, for \emph{any} compact subset $K\subset M$ whose interior
intersects $\pi(\Omega)$ and any compact subset $\tilde K$ of $\tilde M$ with
$p_\Gamma(\tilde K)=K$, for any $T_0\ge 0$, there exist $C>0$ and $\alpha>0$
such that for all $T\ge T_0$,
\[
m^F(V_{T_0,T}(\tilde K))\le C\exp (-\alpha T)\,.
\]
\end{coro}

Before proving these results about strong positive recurrence, we provide in
the next paragraph ways of constructing strongly positively recurrent
potentials.


\subsection{Strong positive recurrence through bumps and wells}

Adding a bump $\lambda A$ to a potential $F$,  with $A$ a nonnegative
compactly supported Hölder-continuous map and $\lambda \to +\infty$, we
already proved in Corollary~\ref{coro:existence-pot-SPR}  the existence of
strongly positively recurrent potentials.  We restate it below with this
terminology.

\begin{coro}\label{coro:existence-pot-SPRbis}
On any nonelementary complete connected Riemannian manifold with pinched
negative curvature and bounded first derivative of the curvature, there exist
Hölder-continuous maps that are strongly positively recurrent.
\end{coro}

It will be convenient to add to a given potential $F$ large bumps of
arbitrarily small height. It is what we do in the next proposition.

\begin{prop}\label{prop:small-bump}
Let $F: T^1M\to \bbR$ be a  Hölder-continuous map with finite topological pressure.  For
all $\varepsilon>0$, there exists a Hölder-continuous map $0\le A\le 1$
compactly supported on $T^1M$, such that
\[
\delta_\Gamma^\infty(F+\varepsilon A)=\delta_\Gamma^\infty(F)\le \delta_\Gamma(F)<\delta_\Gamma(F+\varepsilon A)\,.
\]
\end{prop}
\begin{proof} For a given
$\varepsilon>0$, by the variational principle for $\Ptop(F)$, there exists a
measure $m_\varepsilon\in\mathcal{M}_{1}^F$, such that
$\Ptop(F)=\delta_\Gamma(F)=\sup_{m\in\mathcal{M}_1^F}\left( h_{KS}(m)+\int F\dd
m\right)\le h_{KS}(m_\varepsilon)+\int F\dd m_\varepsilon+\frac{\varepsilon}{2}$.

Choose some compact subset $K_\varepsilon$ such that
$m_\varepsilon(T^1K_\varepsilon)\ge 1-\varepsilon$. Now, choose some Hölder-continuous
map $0\le A\le 1$ with compact support such that $A\equiv 1$ on
$T^1K_\varepsilon$. Observe that as soon as $0<\varepsilon<1/2$, we have
 \[
\delta_\Gamma(F+\varepsilon A) \ge  h_{KS}(m_\varepsilon)+\int F\dd m_\varepsilon +\varepsilon m(K_\varepsilon)
\ge \delta_\Gamma(F)-\frac{\varepsilon}{2}+\varepsilon(1-\varepsilon)>\delta_\Gamma(F)\,.
\]
The result follows.
\end{proof}

Adding a bump does not modify the  topological pressure at infinity, and increases the
 topological pressure to produce strongly positively recurrent potentials. On the other
hand, subtracting a bump, i.e., adding a well, does not modify the  topological pressure
at infinity and decreases the  topological pressure towards the  topological pressure at infinity, as
shown in the next statement.

\begin{prop}\label{th:PressureWell}
Let $F: T^1M\to \bbR$ be a  Hölder-continuous map. Then for all $\beta\ge 0$
satisfying $\delta_\Gamma^\infty(F) \le \delta_\Gamma(F)-\beta$ and all
$\eta>0$, there exists a compactly supported Hölder-continuous function $A$
on $T^1M$ taking values in $[0,\beta]$ such that
\[
\delta_\Gamma(F - A)\le \delta_\Gamma(F)-\beta+\eta\,.
\]
\end{prop}
When $\delta_\Gamma^\infty(F)$ is finite, one may take $\beta =
\delta_\Gamma(F) - \delta_\Gamma^\infty(F)$. Then the proposition says that,
by perturbing $F$ with a compactly supported potential taking values in $[0,
\delta_\Gamma(F) - \delta_\Gamma^\infty(F)]$, one may get a pressure which is
arbitrarily close to $\delta_\Gamma^\infty(F)$. The formulation we have given
also makes sense when $\delta_\Gamma^\infty(F) = -\infty$, and says in this
case that with a compact perturbation one can make the topological  pressure arbitrarily
negative.

\begin{proof}
When $\delta_\Gamma(F)=+\infty$, the conclusion is true for $A = 0$.
Therefore, we may assume that $F$ has finite topological  pressure. For every $n\in\bbN$, let $A_n$ be a
compactly supported Hölder-continuous map  taking values in $[0,\beta]$, equal
to $\beta$ on $p_\Gamma T^1B(o, n)$. We claim that $\limsup \delta_\Gamma(F- A_n) \leq
\delta_\Gamma(F)-\beta$. The proposition follows from this claim by taking
$A=A_n$ for large $n$. Let us prove it.

For every $n\ge 1$, choose an invariant measure $\mu_n$ with $h_{KS}(\mu_n) +
\int (F-A_n) \dd\mu_n \geq \delta_\Gamma(F-A_n) -1/n$. Extracting a
subsequence if necessary, we can assume that $\mu_n$ converges weakly to an
invariant measure $\mu$, with mass $\norm{\mu} \in [0,1]$. Let $\epsilon>0$.
Choose a large compact subset $K\subset M$ with $\mu(T^1 K) > \norm{\mu} -
\epsilon$ and $\mu(\partial T^1 K)=0$. For large enough $n$, we also have
$\mu_n(T^1 K) > \norm{\mu} - \epsilon$, and therefore $\int A_n \dd\mu_n \geq
\beta(\norm{\mu} - \epsilon)$ as $A_n$ is equal to $\beta$ on $T^1 K$. We get
\begin{multline*}
  \limsup \delta_\Gamma(F- A_n) = \limsup \left(h_{KS}(\mu_n) + \int (F-A_n) \dd\mu_n \right)
  \\
  \leq \limsup \left(h_{KS}(\mu_n) + \int F \dd\mu_n\right)
  - \beta(\norm{\mu} - \epsilon).
\end{multline*}
Apply now Corollary~\ref{coro:PressureMassInfty}, to get an upper bound
\begin{equation*}
  \limsup \delta_\Gamma(F- A_n) \leq (1-\norm{\mu}) \delta_\Gamma^\infty(F) + \norm{\mu} \delta_\Gamma(F) - \beta(\norm{\mu} - \epsilon).
\end{equation*}
With the inequality $\delta_\Gamma^\infty(F) \leq \delta_\Gamma(F) - \beta$,
this gives
\begin{equation*}
  \limsup \delta_\Gamma(F- A_n) \leq \delta_\Gamma(F) - \beta + \beta \epsilon.
\end{equation*}
As $\epsilon$ is arbitrary, this concludes the proof.
\end{proof}


\subsection{Strong positive  recurrence implies positive recurrence}\label{sec:SPR-implique-PR}

In this section, we shall prove Theorem~\ref{theo:SPR-implies-PR'}. We follow
the proof of~\cite{ST19} in the case $F=0$.

Assume that $F$ is strongly positively recurrent. By definition, there exists
a compact subset $K\subset M$ whose interior intersects at least a closed
geodesic, and a compact subset $\tilde K\subset\tilde M$ with $p_\Gamma(
\tilde K)=K$, such that
\[
\delta_{\Gamma_{\tilde K}}(F)<\delta_\Gamma(F)\,.
\]
Without loss of generality, we assume that $o\in\tilde K$.

An elementary computation shows that this strict inequality implies the
convergence of  the series $\ds \sum_{\gamma\in\Gamma_{\tilde K}} d(o,\gamma
o)e^{-\delta_\Gamma(F)d(o,\gamma o)+\int_o^{\gamma o} \tilde F}$. Therefore,
in order to prove that strong positive recurrence implies positive recurrence, by
Theorem~\ref{theo:Pit-Schap} (point 3), it is enough to show that $F$ is
recurrent.

The following inclusion is a variant of an observation of~\cite{ST19}:
\[
\Lambda_\Gamma\setminus\Lambda_{\Gamma}^{\mathrm{rad}}\subset
\underset{T_0>0}{\bigcup}\underset{T>T_0}{\bigcap}U_{T_0,T}(\tilde K)\,.
\]
Indeed, the set on the right represents points $y\in\partial\tilde M$ such
that for some $x\in\tilde K$, the geodesic $[x,y)$ stays a bounded amount of
time in $\Gamma.\tilde K$, whereas the set on the left is the set of
$y\in\Lambda_\Gamma$ such that the geodesic $[xy)$ eventually leaves every orbit $\Gamma.\tilde L$, for every compact subset $\tilde L\subset \tilde M$.

To prove Theorem~\ref{theo:SPR-implies-PR'}, it suffices to show that for all
$T_0>0$, we have $\ds\nu^F( \bigcap_{T>T_0} U_{T_0,T}(\tilde K))=0$: then it
follows by the above inclusion that $\nu^F$ gives zero measure to
$\Lambda_\Gamma\setminus\Lambda_{\Gamma}^{\mathrm{rad}}$, and therefore full
measure to $\Lambda_{\Gamma}^{\mathrm{rad}}$. Then Theorem~\ref{theo:HTS}
implies that $F$ is recurrent.

The following lemma, a variation around~\cite[Eq.(29)]{ST19}, is a key step
of the proof, and will be useful also in Section~\ref{exp-rec}. For
$\varepsilon>0$, let $\tilde K_\varepsilon$ be the $\varepsilon$-neighborhood
of $\tilde K$. Let $D$ be the diameter of $\tilde K$. We recall that $\tilde
K$ contains $o$.
\begin{lemm}
\label{lem:ombres-et-GammaK} For all $\varepsilon>0$, there exists $T_0>0$
such that for all $T> T_0+2D+\varepsilon$, we have
\[
 \underset{\gamma\in\Gamma_{\tilde K_\varepsilon},\, d(o,\gamma o)\ge T+D+T_0  }
 {\bigcup}\mathcal{O}_o(\gamma \tilde K)\subset
 U_{T_0,T} (\tilde K).
\]
Moreover, for all $T_1\geq 0$, there exists a finite set $\{g_1,
\dotsc,g_N\}$ of elements of $\Gamma$ such that for all $T> T_1+2D$, we have
\[
U_{T_1,T} (\tilde K) \cap \Gamma \tilde K
\subset \bigcup_{i=1}^N \underset{ \gamma\in\Gamma_{\tilde K },d(o,\gamma o)\ge
T-2D-T_1} {\bigcup} g_i\cdot\mathcal{O}_{\tilde K}(\gamma \tilde
K)\,.
\]
\end{lemm}

\begin{proof}
The first inclusion uses the same kind of arguments as
for~\cite[Eq.(29)]{ST19}. If $\gamma\in\Gamma_{\tilde K_\varepsilon}$, there
exist $x,y \in \tilde K_\epsilon$ such that the geodesic segment $[x,\gamma
y]$ does not intersect $\Gamma\tilde K_\varepsilon$ outside $\{x, \gamma
y\}$. Consider $z \in \mathcal{O}_o(\gamma \tilde K)$. The geodesic $[o, z]$
intersects $\gamma \tilde K$ at a point $\gamma z_0$. By
Lemma~\ref{lm:NegCurv4Points}, there exists $T_0>0$ such that the geodesics
$[x,\gamma y]$ and $[o, \gamma z_0]$ follow each other up to $\epsilon$,
except in the $T_0$ neighborhood of the beginning and of the end of these
segments. Moreover, $d(o,\gamma o)\ge T+D+T_0$, so that $d(o,\gamma z_0)\ge T+T_0$.  Therefore, $[o, \gamma z_0]$ avoids $\Gamma \tilde K$ along $[T_0,
T]$, and so does $[o, z]$.

For the second inclusion, let $T_1\geq 0$. Introduce a finite family
$(g_i)_{1\le i\le N}$ of isometries of $\Gamma$ such that the
$T_1$-neighborhood $\tilde K_{T_1}$ of $\tilde K$ satisfies $\tilde
K_{T_1}\cap \Gamma\tilde K\subset\bigcup_i g_i\tilde K$. Consider a point
$y\in U_{T_1,T} (\tilde K) \cap \Gamma \tilde K$. Consider the last copy
$g_i\tilde K$ intersected by the segment $[o,y]_{T_1}$, and the first copy
$h\tilde K$ intersected by the segment $[o,y]_{T_1+D,T}$. By definition,
$\gamma = g_i^{-1}h\in\Gamma_{\tilde K}$. Moreover, it satisfies $d(o, \gamma
o) \geq T-2D-T_1$, and $y \in g_i\cdot\mathcal{O}_{\tilde K}(\gamma \tilde
K)$ by construction. The desired inclusion follows.
\end{proof}

Lemmas~\ref{lem:ombres-et-GammaK} and~\ref{lem:orbital-shadow-lemma} have the
following corollary, from which Theorem~\ref{theo:SPR-implies-PR'} follows.
\begin{coro}
For all $T_0 \geq 0$, for all $0<\eta<\delta_\Gamma(F)-\delta_{\Gamma_{\tilde
K}}(F)$, there exist $T_1,C>0$ such that for $T\ge T_1$, we have
\begin{equation}
\label{Eq:exp-rec}
\nu^F(U_{T_0,T}(\tilde K))\le Ce^{-(\delta_\Gamma(F)-\delta_{\Gamma_{\tilde K}}(F)-\eta)T}.
\end{equation}
In particular
\[
\nu^F(\bigcap_{T>T_0} U_{T_0,T}(\tilde K))=0\,.
\]
\end{coro}
Similar statements appeared in~\cite{ST19} and~\cite{CDST}, but it appears
that some details are welcome on the limit process. We include therefore a
detailed (short) argument.
\begin{proof}  Let $\eta$ be as above. By
Lemmas~\ref{lem:ombres-et-GammaK} and~\ref{lem:orbital-shadow-lemma} and
conformality of $\nu^{F,s_n}$, for all $s_n>\delta_\Gamma(F)$ close enough to
$\delta_\Gamma(F)$, and $T>T_0$ large enough, we have
\begin{align*}
\nu^{F,s_n}(U_{T_0,T}(\tilde K))&=\nu^{F,s_n}(\Gamma o\cap U_{T_0,T}(\tilde K))\le
 \sum_{i=1}^N\sum_{ \gamma\in\Gamma_{\tilde K },d(o,\gamma o)\ge T-2D-T_0  }
   \nu^{F,s_n}(g_i\cdot\mathcal{O}_{\tilde K}(\gamma \tilde K))\\
&\le  N\times C\times \sum_{ \gamma\in\Gamma_{\tilde K },d(o,\gamma o)\ge T-2D-T_0  }
  e^{-(s_n-\eta/2) d(o,\gamma o)+\int_o^{\gamma o} \tilde F}.
\end{align*}
As the exponential growth rate of $\sum_{ \gamma\in\Gamma_{\tilde K },i \leq
d(o,\gamma o) \leq i+1} e^{\int_o^{\gamma o} \tilde F}$ is
$\delta_{\Gamma_{\tilde K}}$ by definition, we get for large enough $i$
\begin{equation*}
\sum_{ \gamma\in\Gamma_{\tilde K },i \leq d(o,\gamma o) \leq i+1}
e^{\int_o^{\gamma o} \tilde F} \leq e^{(\delta_{\Gamma_{\tilde K}}+\eta/2)i}.
\end{equation*}
Together with the previous equation, this gives for large enough $T$, for some constant $C>0$, the
inequality
\begin{equation*}
\nu^{F,s_n}(U_{T_0,T}(\tilde K)) \leq C\times e^{(\delta_{\Gamma_{\tilde
K}}(F)+\eta-s_n)T}.
\end{equation*}

Now, $\nu^F$ is the  weak-$*$ limit $\nu^F=\lim_{n\to \infty} \nu^{F,s_n}$.
Therefore, for any open set $U$, one has $\nu^F(U) \leq \liminf
\nu^{F,s_n}(U)$. We obtain
\begin{equation*}
\nu^F(U_{T_0,T}(\tilde K)) \le  C \times e^{(\delta_{\Gamma_{\tilde K}}(F)+\eta-\delta_\Gamma(F))T }\,.
\end{equation*}
The result follows.
\end{proof}


\subsection{Strong positive recurrence and exponential recurrence}\label{exp-rec}

Let us prove Theorem~\ref{theo:SPR-equiv-exprec}.
\begin{proof}
The implication ``strong positive recurrence  of $F$  implies exponential
recurrence   of $(g^t)$ with respect to $m^F$'' was essentially shown in the
above proof of Theorem~\ref{theo:SPR-implies-PR'}, and in particular
Equation~\eqref{Eq:exp-rec}. Indeed, let $K$ and $\tilde K$ be as in the
proof of this theorem. On $T^1\tilde M$, the product structure $\tilde
m^F\sim\nu^F\times \nu^F\times \dd t$ (see Equation~\eqref{Gibbs-product}),
in the Hopf coordinates (see Equation~\eqref{Hopf}) shows that up to some
constant $c$,
\[
m^F(V_{T_0,T}(\tilde K))\leq \tilde m^F(\tilde V_{T_0,T}(\tilde K))\le c\nu^F(\partial\tilde M)\times\nu^{F}(U_{T_0,T}(\tilde K))\,.
\]
Equation~\eqref{Eq:exp-rec} concludes. Note that this proof, combined with
Theorem~\ref{theo:indep-compact}, implies
Corollary~\ref{coro:expo-rec-indep-compact}.

\medskip

Conversely, suppose that $(g^t)$ is exponentially recurrent with respect to
$m^F$, so that for some compact subset $\tilde K\subset \tilde M$ whose interior intersects
$\pi(\tilde \Omega)$, for every $T_0>0$, there exist $\alpha=\alpha(T_0)>0$ and
$C=C(T_0)>0$ such that, for all $T>T_0$,
\begin{equation}
\label{eq:mFexp}
m^F(V_{T_0,T}(\tilde K))\le C\exp(-\alpha T)\,.
\end{equation}

The first step consists in showing that there exists a constant $C'$ such
that, for all $T\ge T_0$, we have
\begin{equation}\label{exp-decay-bord}
 \nu^F(U_{T_0,T}(\tilde K))\le C' e^{-\alpha T}\,.
\end{equation}

Since the projection $p_\Gamma$ is finite-to-one on the compact set $T^1
\tilde K$, one deduces from~\eqref{eq:mFexp} that $\tilde m^F(\tilde
V_{T_0,T}(\tilde K) \leq C_0 \exp(-\alpha T)$ for some constant $C_0$. By
definition, if $v\in \tilde V_{T_0,T}(\tilde K)$, then $v^+\in
U_{T_0,T}(\tilde K)$, and $v^-\in \mathcal{O}_{v^+}(\tilde K)$. Recall that
$m^F$ is supported in $\Omega$. As above, Equations~\eqref{Gibbs-product}
and~\eqref{Hopf} show that up to some constant $c$,
\[
C_0 \exp(-\alpha T) \geq \tilde m^F(\tilde V_{T_0,T}(\tilde K))\ge
\frac{1}{c}\inf_{v\in\tilde \Omega\cap T^1\tilde K}\nu^F(\mathcal{O}_{v^+}(\tilde K)) \times\nu^{F}(U_{T_0,T}(\tilde K))\,.
\]
In the above infimum, the vector $v$ varies in  the compact set  $\tilde
\Omega\cap T^1\tilde K$,  and $\nu^F$ has full support in the limit set, so
that this infimum is positive. Therefore,~\eqref{exp-decay-bord} is proven.

In what follows, we will need to consider a compact set $\tilde L$ large enough
to satisfy the lower bound in Lemma~\ref{lem:orbital-shadow-lemma}. By a
standard use of Lemma~\ref{lm:NegCurv4Points}, for all $\varepsilon >0$ there
exists $\tau>0$, such that  if $\tilde L\supset\tilde
K_\varepsilon\supset\tilde K$ contains an $\varepsilon$-neighborhood of
$\tilde K$, uniformly in $T\ge T_0+2\tau$, we have
\[
\overline{U_{T_0,T}(\tilde L)}\subset
U_{T_0+\tau,T-\tau}(\tilde K).
 \]
In particular, it follows from~\eqref{exp-decay-bord} that $\nu^F(\overline{U_{T_0,T}(\tilde L)}) \leq C' e^{-\alpha T}$ for
some $C'>0$ and $\alpha>0$. Until now, $T_0$ was arbitrary.  We choose now $T_0$ given by the first
item in Lemma~\ref{lem:ombres-et-GammaK}.

As $\nu^F=\lim_{s_n\to \delta_\Gamma(F)} \nu^{F,s_n}$, we have
\begin{equation*}
  \limsup \nu^{F,s_n}(U_{T_0, T}(\tilde L)) \leq \nu^F(\overline{U_{T_0, T}(\tilde L)})
  \leq C' e^{-\alpha T}.
\end{equation*}
Therefore, for all $s_n$ close enough to $\delta_\Gamma(F)$, we have $
\nu^{F,s_n}(U_{T_0,T}(\tilde L))\le C' e^{-\beta  T}$ for any $\beta <
\alpha$, for instance $\beta= \alpha/2$. Fix some $\epsilon>0$. Then
Lemma~\ref{lem:ombres-et-GammaK} gives for some $D=D(\tilde L)$
\[
 \underset{\gamma\in\Gamma_{\tilde L_\varepsilon},\, d(o,\gamma o)\ge T+D+T_0}{\bigcup}\mathcal{O}_o(\gamma \tilde L)\subset
 U_{T_0,T} (\tilde L)\,,
\]
so that, as $\nu^{F,s_n}$ is supported on $\Gamma o$,
\[
\nu^{F,s_n}\left(
\Gamma o\cap\underset{\gamma\in\Gamma_{\tilde L_\varepsilon},\, d(o,\gamma o)\ge T+D+T_0  }{\bigcup}\mathcal{O}_o(\gamma \tilde L)\right)
\le C' e^{-\beta T}\,.
\]
In particular, there exists $C''$ such that, for any large enough $k$, we have
\[
\nu^{F,s_n}\left(
\Gamma o\cap\underset{\gamma\in\Gamma_{\tilde L_\epsilon},\, d(o,\gamma o)\in [k, k+1]}{\bigcup}\mathcal{O}_o(\gamma \tilde L)\right)
\le C'' e^{-\beta k}\,.
\]

As the group $\Gamma$ acts properly discontinuously on $\tilde M$ and $\tilde
L$ is compact, the intersections of shadows in the above union have a bounded
multiplicity. Therefore, we deduce that  there exists some constant $c>0$
such that
\[
\sum_{\gamma\in\Gamma_{\tilde L_\epsilon},\, d(o,\gamma o)\in [k, k+1]  } \nu^{F,s_n}(\mathcal{O}_o(\gamma \tilde L))\le ce^{-\beta k}\,.
\]
Together with the Orbital Shadow Lemma~\ref{lem:orbital-shadow-lemma}, this
implies that up to some multiplicative constant, uniformly in $s_n$, for some $c'>0$, for all
$k$ large enough, we have
\begin{equation*}
\sum_{\gamma\in\Gamma_{\tilde L _\varepsilon},\, d(o,\gamma o)\in [k, k+1]  }  e^{-s_n d(o,\gamma o)+\int_{o}^{\gamma o}\tilde F}\le c' e^{-\beta k}\,.
\end{equation*}

Let $s_n$ converge to $\delta_\Gamma(F)$. As $d(o, \gamma o) \in [k, k+1]$,
the previous inequality gives, for some $c''>0$, for large enough $k$,
\begin{equation*}
\sum_{\gamma\in\Gamma_{\tilde L _\varepsilon},\, d(o,\gamma o)\in [k, k+1]  }  e^{\int_{o}^{\gamma o}\tilde F}\le c'' e^{\delta_\Gamma(F) k -\beta k}\,.
\end{equation*}
Since $\delta_{\Gamma_{\tilde L _\varepsilon}}(F)$ is the exponential growth
rate of the left hand term, we get $\delta_{\Gamma_{\tilde L
_\varepsilon}}(F) \leq \delta_\Gamma(F) -\beta < \delta_\Gamma(F)$. In
particular, $\delta_\Gamma^\infty(F) < \delta_\Gamma(F)$, proving the strong
positive recurrence of $F$.
\end{proof}



\subsection{SPR is independent of the compact set}\label{sec:SPR-ind-compact}

This paragraph is devoted to the proof of Theorem~\ref{theo:indep-compact}.
Let $F : T^1M \to \bbR$ be a  strongly positively recurrent Hölder-continuous
potential. Let $K\subset M$ be a compact subset whose interior $\inter{K}$
intersects $\pi(\Omega)$, and $\tilde K\subset  \tilde M$  a compact subset
such that $p_\Gamma(\tilde K) = K$. Our proof relies on the following
proposition, which provides a convenient upper bound for the growth of
$\Gamma_{\tilde K}$.

\begin{prop}\label{prop:trou-lisse}
Let $A : T^1M\to [0, +\infty)$ be a nonnegative Hölder-continuous map whose support is contained in the interior of $T^1K$. Then
\[
\delta_{\Gamma_{\tilde K}}(F) \leq \delta_\Gamma(F-A).
\]
\end{prop}

\begin{proof}
By definition, $\delta_{\Gamma_{\tilde K}}(F)$ is the exponential growth rate
of
\begin{equation}
\label{eq:wxclkvjmlkwc}
\sum_{\gamma \in \Gamma_{\tilde K}, t-1 \leq d(o, \gamma o) \leq t}
e^{\int_o^{\gamma o} \tilde F}\,.
\end{equation}
Consider $\gamma \in \Gamma_{\tilde K}$. By definition, there exist $x,y\in
\tilde K$ such that $[x, \gamma y] \cap \Gamma \cdot \tilde K = \{x, \gamma
y\}$.

Let $\epsilon>0$ be small enough so that $A$ vanishes on the
$\epsilon$-neighborhood of the complement of $\tilde K$.
Lemma~\ref{lm:NegCurv4Points} provides a constant $T_0$ such that the
segments $[o, \gamma o]$ and $[x, \gamma y]$ follow each other up to
$\epsilon$, except in the $T_0$ neighborhood of their endpoints. In
particular, the function $\tilde A=A\circ p_\Gamma$ vanishes along $[o,
\gamma o]$ except in the $T_0$-neighborhood of its endpoints, where it is
uniformly bounded as $A$ has compact support.   It follows that
\begin{equation*}
  \abs*{\int_o^{\gamma o} \tilde F - \int_o^{\gamma o} (\tilde F - \tilde A)} \leq C,
\end{equation*}
where $C$ does not depend on $\gamma$. In particular,~\eqref{eq:wxclkvjmlkwc}
is bounded by
\begin{equation*}
e^C \sum_{\gamma \in \Gamma_{\tilde K}, t-1 \leq d(o, \gamma o) \leq t}
e^{\int_o^{\gamma o} \tilde F-\tilde A}
\leq e^C \sum_{\gamma \in \Gamma, t-1 \leq d(o, \gamma o) \leq t}
e^{\int_o^{\gamma o} \tilde F-\tilde A}\,.
\end{equation*}
As the exponential growth rate of the latter sum is $\delta_\Gamma(F-A)$,
this concludes the proof.
\end{proof}

We will also need the following proposition.

\begin{prop}\label{prop:trou-Gibbs}
Let $F_1, F_2 : T^1M\to \bbR$ be two Hölder-continuous potentials with finite
topological pressure that satisfy $F_2\le F_1$ and $F_2(w)<F_1(w)$ for some
$w\in \Omega$. Assume that $F_2$ admits a finite Gibbs measure  $m_{F_2}$.
Then their topological pressures satisfy
\[
\Ptop(F_2) < \Ptop(F_1)\,.
\]
\end{prop}

\begin{proof} For $i = 1,2$, $\Ptop(F_i)$ coincides with $\Pvar(F_i)$, i.e.,
\[
\sup \left\{ \int F_i \dd m + h_{KS}(m) \; ; \; m \text{ invariant probability measure with } \int F_i^- \dd\mu_i < +\infty\right\}.
\]
As $F_2 \leq F_1$, we have $\int F_1^- \dd m \leq \int F_2^- \dd m$ for any
invariant probability measure $m$. Therefore, when $m=m_{F_2}$,
\[
\Pvar(F_2)= \int F_2 \dd m_{F_2} + h_{KS}(m_{F_2})\le \int F_1 \dd m_{F_2} + h_{KS}(m_{F_2})\le \Pvar(F_1)  .
\]

Assume by contradiction that $\Pvar(F_1) = \Pvar(F_2)$. Then by the previous inequalities,
\[
\int F_1 \dd m_{F_2} =\int F_2 \dd m_{F_2}\,.
\]
It implies that $F_1 = F_2$  $m_{F_2}$-almost surely. As $F_2\le F_1$ and
$F_2<F_1$ on  a neighborhood of $w$, this contradicts the fact that $m_{F_2}$
has full support in $\Omega$. Therefore $\Pvar(F_2)< \Pvar(F_1)$.
\end{proof}

Let us conclude the proof of Theorem~\ref{theo:indep-compact}.

\begin{proof}[Proof of Theorem~\ref{theo:indep-compact}]
Choose some $w\in \Omega \cap T^1K$ and $\epsilon>0$  such that $B(w,
2\epsilon)\subset T^1K$. Let $A : T^1M \to [0, +\infty)$ be a nonnegative
Hölder-continuous map supported in $B(w, \epsilon)$ with $A(w)>0$. By
Proposition~\ref{prop:CompactPerturbPotential},  for all $\eta>0$,
$\delta_\Gamma^\infty(F - \eta A) = \delta_\Gamma^\infty(F)$. Moreover, the
map $\eta \mapsto \delta_\Gamma(F - \eta A)$ is Lipschitz-continuous by
Proposition~\ref{prop:PressureBump}. As $F$ is strongly positively recurrent,
for $\eta>0$ small enough, the map $F - \eta A$ is still strongly positively
recurrent. In particular, by Theorem~\ref{theo:SPR-implies-PR}, it admits a
finite Gibbs measure. Therefore,  Propositions~\ref{prop:trou-lisse}
and~\ref{prop:trou-Gibbs} give the inequalities
\[
\delta_{\Gamma_{\tilde K}}(F) \leq \delta_\Gamma(F-\eta A) < \delta_\Gamma(F)\,.
\]
Theorem~\ref{theo:indep-compact} follows.
\end{proof}


\appendix

\section{Entropies for geodesic flows, by Felipe Riquelme}

In this appendix, we prove that three important notions of entropies of an
invariant probability measure for the dynamics of the geodesic flow on negatively curved manifolds coincide,
namely the Kolmogorov-Sinai, the Katok and the Brin-Katok entropies. These
equalities were first proved for dynamical systems defined on compact metric
spaces in~\cite{Katok80} and~\cite{BK83}, and generalized for Lipschitz maps
on noncompact manifolds in~\cite{Riq-Ruelle-geod} taking only in
consideration ergodic measures. This appendix treats the case of non-ergodic
measures as well as the equality with Katok and local (Brin-Katok) entropies
relative to small dynamical balls. The extension of this appendix to the orbifold setting is open, as discussed in our last paragraph.


\subsection{Different notions of entropy}

Let $(\tilde M,g)$ be a smooth complete connected Riemannian manifold with pinched negative sectional
curvature $-b^2\leq K_g\leq -a^2$, for some $0<a\leq b$. Let $M=\tilde M/\Gamma$ be a quotient manifold, with  $\Gamma=\pi_1(M)$ a discrete group, and
$p_\Gamma:T^1\tilde{M}\to T^1M$ the differential of the quotient map
$\tilde{M}\to M$. We will denote by $(g^t)$ both geodesic flows
on $T^1\tilde{M}$ and  $T^1M=T^1\tilde M/\Gamma$.

For all definitions of entropy, the entropy of the geodesic flow $(g^t)$ with
respect to an invariant probability measure $\mu$ on $T^1M$ is defined as the
entropy of its time-one-map $ g^1$ with respect to $\mu$. If $\mu$ is ergodic
with respect to the flow, it is not necessarily ergodic with respect to this
time-one-map $g^1$. However, in this case, a.e.~time $\tau\in\mathbb{R}$ is
ergodic, so that the relation $h(g^\tau)=\abs{\tau} h(g^1)$ allows us to
assume, without loss of generality, that $\mu$ is ergodic w.r.t.~$g^1$.

\subsubsection{The Kolmogorov-Sinai entropy}
Let $\mathcal{M}_1$ be the set of $g^1$ invariant probability measures on $T^1M$ and let $\mu\in\mathcal{M}_1$.  In this appendix, the word {\em partition} always denotes a finite or countable measurable partition of $T^1M$.  Let
$\mathcal{P}$ be such a partition. The
entropy of $\mathcal{P}$ is defined by
\[
H(\mu,\mathcal{P})=-\sum_{P\in\mathcal{P}}\mu(P)\log\mu(P)\,.
\]
The join $\mathcal{P}^n=\bigvee_{i=0}^ng^{-i}\mathcal{P}$ is the partition
whose atoms are the nonempty subsets of the form $P_0\cap g^{-1}P_1\cap\dotsb g^{-n}P_n$, where
the sets $P_i$ are in $\mathcal{P}$. The entropy of $\mu$ w.r.t.
$\mathcal{P}$ is the limit
\[
h(\mu,\mathcal{P})=\lim_{n\to \infty}  \frac{1}{n}H(\mu,\mathcal{P}^n)\,.
\]
The Kolmogorov-Sinai entropy of $\mu$ is the  least upper bound
\[
h_{KS}(\mu)\coloneqq \sup_{\mathcal{P}} h(\mu,\mathcal{P})
\]
over all   partitions $\mathcal{P}$ with finite
entropy.

For any $v\in T^1M$ denote by $\mathcal{P}(v)$ the atom of $\mathcal{P}$
containing $v$. Shannon-McMillan-Breiman Theorem (see for
instance~\cite{AC88})  asserts that whenever $\mu$ is ergodic, then for
$\mu$-a.e.\ $v\in T^1M$, we have
\[
h(\mu, \mathcal P)=\lim_{n\to\infty} -\frac{1}{n}\log \mu(\mathcal{P}^n(v))\,.
\]
Moreover, when $\mu$ is not ergodic, we have
\[
\int_{T^1M} \lim_{n\to\infty} -\frac{1}{n}\log \mu(\mathcal{P}^n(v)) \dd\mu(v) = h(\mu, \mathcal P).
\]

\subsubsection{The Katok entropies}

For completeness, let us recall the following definitions. Let $d$ be any
metric on $T^1\tilde{M}$, bi-Lipschitz equivalent to the Sasaki metric. By an
abuse of notation, we will denote by $d$ the corresponding induced metric on
$T^1M$ and by $B_d(v, r)$ the corresponding metric ball centered at $v$ with
radius $r>0$.

Let $\tilde{v}\in T^1\tilde M$ and $\epsilon,T>0$. The \emph{dynamical ball}
$B(\tilde{v}, \epsilon;  T)$ on the universal cover is defined by
\[
B(\tilde{v}, \epsilon; T) = \{\tilde{w}\in T^1\tilde M \; ; \; \forall t\in [0, T], d(g^t \tilde{v}, g^t\tilde{w}) \leq \epsilon\}.
\]
As in~\cite[Rem 3.1]{ST19}, we consider on $T^1M$ the small dynamical ball
$B_\Gamma(v,\varepsilon; T)=p_\Gamma(B(\tilde v,\varepsilon; T))$ and the big
dynamical ball
\begin{equation}\label{eqn:dyn-ball2}
B_{\mathrm{dyn}}(v,\varepsilon; T)=\{w\in T^1 M \; ; \; \forall t\in [0, T], d(g^t v, g^tw) \leq \epsilon\}\supset B_\Gamma(v,\varepsilon; T).
\end{equation}
Both balls coincide as soon as the injectivity radius of $M$ is bounded from
below away from zero uniformly on $T^1M$ and $\varepsilon$ is small enough.
More generally, if along the orbit $(g^t v)_{0\le t\le T}$, the injectivity
radius at the point $\pi(g^t v)$ is larger than $\varepsilon$, then
\begin{equation}\label{eq:one-sided-balls-coincide}
B_{\mathrm{dyn}}(v,\varepsilon; T)= B_\Gamma(v,\varepsilon; T)\,.
\end{equation}

Given a probability measure $\mu$ on $T^1M$, $\delta\in (0,1)$ and
$\epsilon,T>0$,
 a set $V \subset T^1M$ is \emph{$(\mu, \delta, \epsilon;  T)$-spanning}
(respectively \emph{dynamically-$(\mu, \delta, \epsilon;  T)$-spanning})  if
\[
\mu\left(\bigcup_{v\in V} B_\Gamma(v, \epsilon;  T) \right)\geq \delta\,,\quad{\text{respectively}}\quad
\mu\left(\bigcup_{v\in V} B_{\mathrm{dyn}}(v, \epsilon;  T) \right)\geq \delta.
\]
Of course, a $(\mu, \delta, \epsilon; T)$-spanning set is also
dynamically-$(\mu, \delta, \epsilon;  T)$-spanning.

Let $S_\Gamma(\mu,\delta,\varepsilon;T)$ (resp.\
$S_{\mathrm{dyn}}(\mu,\delta,\varepsilon;T)$) be the minimal cardinality of a
$(\mu, \delta, \epsilon;  T)$-spanning set (resp.\ of a dynamically-$(\mu,
\delta, \epsilon;  T)$-spanning set).

The Katok entropy of $\mu$ with respect to the small (resp.\ big) dynamical
balls are defined respectively as
\[
h_{\mathrm{Katok}}^\Gamma(\mu)=\inf_{\delta>0} \sup_{\epsilon>0}\limsup_{T\to\infty}\frac{1}{T}\log S_\Gamma(\mu,\delta,\varepsilon;T)\,,
\]
and
\[
h_{\mathrm{Katok}}^{\mathrm{dyn}}(\mu)=\inf_{\delta>0}
\sup_{\epsilon>0}\limsup_{T\to\infty}\frac{1}{T}\log
S_{\mathrm{dyn}}(\mu,\delta,\varepsilon;T)\,.
\]


\subsubsection{The Brin-Katok entropies}

Given a nonempty compact subset $\mathcal{K}\subset T^1M$, we define the local entropies on $\mathcal{K}$
relative respectively to small and big dynamical balls as
\[
\bar{h}_{\mathrm{loc}}^\Gamma(\mu,\mathcal{K})=\sup_{\epsilon>0}\;\esssup_{v\in \mathcal{K}}\limsup_{T\to \infty,\,\, g^Tv\in \mathcal{K}} -\frac{1}{T}\log \mu(B_\Gamma(v,\varepsilon;T))\,,
\]
and
\[
\bar{h}_{\mathrm{loc}}^{\mathrm{dyn}}(\mu,\mathcal{K})=\sup_{\epsilon>0}\;\esssup_{v\in \mathcal{K}}\limsup_{T\to \infty,\,\, g^Tv\in \mathcal{K}} -\frac{1}{T}\log \mu(B_{\mathrm{dyn}}(v,\varepsilon;T))\,.
\]
Taking the  least upper bound  over nonempty compact subsets $\mathcal{K}$ leads to the definition of the
upper Brin-Katok local entropies
\[
\bar{h}_{BK}^\Gamma(\mu)=\sup_{\mathcal{K}} \;\bar{h}_{\mathrm{loc}}^\Gamma(\mu,\mathcal{K})\quad\text{and}\quad
\bar{h}_{BK}^{\mathrm{dyn}}(\mu)=\sup_{\mathcal{K}} \;\bar{h}_{\mathrm{loc}}^{\mathrm{dyn}}(\mu,\mathcal{K})\,.
\]

\subsection{All entropies coincide}

The main result of this appendix is stated below. Despite of being expected,
its relevance lies in its many potential applications. For example,
in~\cite[Theorem 1.4]{ST19}, a formula relating local entropies of invariant
measures through a change of the Riemannian metric has been established,
which brings as consequence such a formula for Kolmogorov-Sinai entropies. In
particular, it also gives a relationship between topological entropies of
geodesic flows coming from perturbations of a given Riemannian metric by the
use of measures of maximal entropies on the corresponding dynamics.

\begin{theo}\label{theo:entropies-coincide}
Let $(M,g)$ be a complete connected Riemannian manifold with pinched negative curvatures
$-b^2\leq K_g\leq -a^2<0$. Let $\mu\in\mathcal{M}_1$ be an ergodic invariant
probability measure for the geodesic flow on $T^1M$. Then
\[
h_{KS}(\mu)=\bar{h}_{BK}^\Gamma(\mu)=\bar{h}_{BK}^{\mathrm{dyn}}(\mu)=h_{\mathrm{Katok}}^\Gamma(\mu)=h_{\mathrm{Katok}}^{\mathrm{dyn}}(\mu)\,.
\]
\end{theo}

We will prove Theorem~\ref{theo:entropies-coincide} in two steps. The first
step is to prove that the Kolmogorov-Sinai entropy coincides with the local
Brin-Katok entropies, and the second one is the analogue with the Katok entropies.

\noindent\\
\textbf{Step 1.} The inequality $h_{KS}(\mu)\le
\bar{h}_{BK}^{\mathrm{dyn}}(\mu)$ is due to Brin-Katok~\cite{BK83}. In this
reference, equality is proved on a compact manifold, but the proof of this inequality does
not use compactness. The inequality $\bar{h}_{BK}^{\mathrm{dyn}}(\mu)\le
\bar{h}_{BK}^{\Gamma}(\mu)$ is immediate from~\eqref{eqn:dyn-ball2}.
Therefore, we just need to prove that $\bar{h}_{BK}^\Gamma(\mu)\leq
h_{KS}(\mu)$.

The proof relies on a crucial geometric property: as the curvature is bounded
from below, the injectivity radius along a geodesic decays at most
exponentially. More precisely, for every compact subset $C\subset M$, there
exists a positive constant $c>0$ such that for all vectors $w\in T^1C$, and
all $t\in\bbR$, we have
\begin{equation}\label{eqn:rayon-inj}
r_{\mathrm{inj}}(g^tw)\ge c^{-1} e^{-c\abs{t}}\,.
\end{equation}
This geometric inequality follows from~\cite[Thm~4.7]{CGT}, see
also~\cite[Prop~4.19]{CCGGIIKLKN}.

\medskip

For the next proposition we do not need the ergodicity of $\mu$. In
particular, the corollary stated after its proof is satisfied for any
invariant probability measure.

\begin{prop}\label{prop:part} For every compact subset
$\mathcal{K}\subset T^1M$ with $\mu(\mathcal{K})>0$, for every
$\varepsilon>0$, there exists a partition $\mathcal{P}_\mathcal{K}$ of
$\mathcal{K}$ with finite entropy such that, if
$\mathcal{P}=\mathcal{P}_\mathcal{K}\sqcup \left({T^1M\setminus
\mathcal{K}}\right)$, for $\mu$-a.e.\ $v\in \mathcal{K}$, the sequence
$n_k\to\infty$ of return times in $\mathcal K$ of $(g^nv)_{n\in\bbN}$
satisfies
\[
\mathcal{P}^{n_k}(v)\subset B_\Gamma(v,\varepsilon;n_k)\,.
\]
In particular, for every compact subset $\mathcal{K}\in T^1M$, for $\mu$-a.e.\
$v\in \mathcal{K}$,
\begin{equation}\label{eqn:part}
\limsup_{n\to \infty, g^n v\in \mathcal{K}} -\frac{1}{n}\log \mu(B_\Gamma(v,\varepsilon;n))\le
 \limsup_{n\to \infty, g^n v\in \mathcal{K} } -\frac{1}{n}\log \mu\left(\mathcal{P}^n(v)\right)\,.
\end{equation}
\end{prop}

\begin{proof}
By~\cite[Proposition 1.34]{Riq-these}, for every compact subset
$\mathcal{K}\subset T^1M$,  there exists $c_0>0$ such that for all
$\delta>0$, there exists a partition $\mathcal{P}_\delta$ of $\mathcal{K}$
whose atoms $\mathcal{P}_\delta(v)$ for any $v\in T^1M$ all have diameter at
most $\delta$, with $\mu(\partial \mathcal{P}_\delta(v))=0$, and $\#
\mathcal{P}_\delta\le c_0\delta^{-d}$ where $d$ is the dimension of $T^1 M$.
As $\mu(\mathcal{K})>0$, by Poincaré recurrence Theorem, we know that for
$\mu$-a.e.\ $v\in \mathcal{K}$, infinitely often $g^{n}v\in \mathcal{K}$.
Divide the set $\mathcal{K}$ (up to a measure $0$ set) into the return time
partition: for all $k\ge 1$, let
\[
 A_k=\{v\in \mathcal{K}, g^kv\in \mathcal{K},\,\text{and} \quad g^i v\notin \mathcal{K}\quad \text{for all}\quad 1\le i\le k-1\}\,.
\]
For all $k\ge 1$, set $\delta_k=\frac{\varepsilon}{(Le^{ c})^k}$, where $L$
is the Lipschitz-constant for the time-one-map $g^1$ of the geodesic flow,
and $c>0$ is the constant associated with the compact subset $C =
\pi(\mathcal{K})\subset M$ from equation~\eqref{eqn:rayon-inj}. For $v\in
A_k$, define $\mathcal{P}(v)$ as $\mathcal{P}(v)\coloneqq
\mathcal{P}_{\delta_k}(v)\cap A_k$. For $v\notin \mathcal{K}$, set
$\mathcal{P}(v)=T^1M\setminus\mathcal{K}$.

Thanks to the choice of $\delta_k$, an immediate verification shows that for
$v\in A_k$, we have $\mathcal{P}(v)\subset
B_{\mathrm{dyn}}(v,\frac{\varepsilon}{e^{ck}};k)\,.$  By
equations~\eqref{eq:one-sided-balls-coincide} and~\eqref{eqn:rayon-inj}, in fact,
we have in this case
\[
\mathcal{P}(v)\subset B_\Gamma(v,\frac{\varepsilon}{e^{ck}};k)=B_{\mathrm{dyn}}(v,\frac{\varepsilon}{e^{ck}};k)\,.
\]
Recall the notation
\[
\mathcal{P}^n(v)=\mathcal{P}(v)\cap g^{-1}\mathcal{P}(gv)\cap \dotsb \cap g^{-(n-1)}\mathcal{P}(g^{n-1}v)\,.
\]
Now, for almost all $v\in \mathcal{K}$, let $n_k\to \infty$ be the sequence of return times of $(g^nv)_{n\ge 0}$
inside $\mathcal{K}$ (with $n_0=0$).
Let $\tilde{v}$ be any lift of $v$ into $T^1\widetilde{M}$. By construction of $\mathcal{P}$, and by the above, for almost every $v\in T^1M$, we have
\begin{align*}
\mathcal{P}^{n_k}(v)&\subseteq \mathcal{P}(v)\cap g^{-n_1}\mathcal{P}(g^{n_1}v)\cap \dotsb \cap g^{-n_{k-1}}\mathcal{P}(g^{n_{k-1}}v)\\
&\subseteq  \bigcap_{i=0}^{k-1} g^{-n_i}B_{\mathrm{dyn}}\left(g^{n_i}v,\frac{\varepsilon}{e^{c(n_{i+1}-n_i)}};n_{i+1}-n_i\right)\\
&=  \bigcap_{i=0}^{k-1}  g^{-n_i}B_{\Gamma}\left(g^{n_i}v,\frac{\varepsilon}{e^{c(n_{i+1}-n_i)}};n_{i+1}-n_i\right)\\
&\subseteq  \bigcap_{i=0}^{k-1}  g^{-n_i}p_\Gamma\left(B_{\mathrm{dyn}} \left(g^{n_i}\tilde v,\varepsilon;n_{i+1}-n_i\right)\right)\\
&=  \bigcap_{i=0}^{k-1}  p_\Gamma \left(g^{-n_i}B_{\mathrm{dyn}} \left(g^{n_i}\tilde v,\varepsilon;n_{i+1}-n_i\right)\right).
\end{align*}
Note that we are strongly using the fact that dynamical balls for the
time-one-map coincide with dynamical balls for the flow at integer times.
Without loss of generality, we may assume that $\varepsilon \leq c^{-1}$. In
particular, the quotient map $p_\Gamma$ is an isometry restricted to each of
the dynamical balls involved in the last intersection, thanks
to~\eqref{eqn:rayon-inj}. Hence, we get
\begin{align*}
\mathcal{P}^{n_k}(v)&\subseteq  p_\Gamma\left(\bigcap_{i=0}^{k-1}  g^{-n_i}B_{\mathrm{dyn}} (g^{n_i}\tilde v,\varepsilon;n_{i+1}-n_i)\right)\\
&=  p_\Gamma\left( B_{\mathrm{dyn}} ( \tilde v,\varepsilon;n_{k })\right)\\
&= B_{\Gamma}(v,\varepsilon;n_{k })\,.
\end{align*}

It remains to prove that $\mathcal{P}$ is a partition of finite entropy.  By construction recall that
\[\#\{P\in\mathcal{P}: P\subseteq A_k\}\leq c_0\delta_k^{-d} = c_0\left(\frac{\varepsilon}{(Le^c)^k}\right)^{-d}\,.\]
We have
\begin{align*}
H_\mu(\mathcal{P})&=-\sum_{P\in\mathcal{P}}\mu(P)\log \mu(P)\\
&=-\mu(\mathcal{K}^c)\log \mu(\mathcal{K}^c)-\sum_{k=1}^\infty\sum_{P\in \mathcal{P}, P\subset A_k}\mu(P)\log \mu(P)\\
&\le -\mu(\mathcal{K}^c)\log \mu(\mathcal{K}^c)-\sum_{k=1}^\infty\mu(A_k)\log\mu(A_k)+\sum_{k=1}^\infty\mu(A_k)\log \#\{P\in\mathcal{P}: P\subset A_k\}\\
&\le -\mu(\mathcal{K}^c)\log \mu(\mathcal{K}^c)-\sum_{k=1}^\infty\mu(A_k)\log\mu(A_k)\\
&\quad -\left(\sum_{k=1}^\infty\mu(A_k)\right)\times\log \left(c_0 \epsilon^d\right)+\sum_{k=1}^\infty\mu(A_k)\times k\log (Le^c)^d\,.
\end{align*}
The first term is some finite constant.  The third term is bounded from above
by a constant times $\mu(\mathcal{K})$ and is therefore finite. By Kac lemma, the last
term, up to a constant, is equal to $\sum_{k=1}^\infty k\mu(A_k)=\mu(T^1M)\le 1$. The second term is finite since
Lemma~1.35 in~\cite{Riq-these} together with $\sum_k k\mu(A_k)<\infty$ imply
$\sum_k \mu(A_k)\log \mu(A_k)<\infty$.Therefore, $\mathcal{P}$ has finite entropy.
\end{proof}

Integrating~\eqref{eqn:part} over $v\in \mathcal{K}$ on the left, and over $v\in T^1M$ on the right,
and using the Shannon-McMillan-Breiman theorem,
 Proposition~\ref{prop:part} leads to the following corollary.
\begin{coro}Under the same assumptions, we have
\begin{equation*}
\int_{\mathcal{K}}
\limsup_{n\to \infty, \,g^nv\in \mathcal{K}}-\frac{1}{n}\log \mu(B_\Gamma(v,n,\varepsilon)) \,d\mu(v)\le
 \int_{T^1M}\limsup_{n\to \infty, g^n v\in \mathcal{K}} -\frac{1}{n}\log \mathcal{P}^n(v)\,d\mu(v)\le h_{KS}(\mu)\,.
\end{equation*}
\end{coro}

Assume now that $\mathcal K$ is large enough so that there exists $v\in \mathcal K$ with $\mu(B_{d}(v, 1))>0$ and $B_{d}(v, 2)\subset \mathcal K$. Let us
define
\[
\mathcal K_{-1}= \{v\in \mathcal K \; ; \; d(v, \mathcal K^c) \geq 1\} \subset \mathcal K.
\]

By our assumption, $\mu(\mathcal K_{-1})>0$. Note that for all $v\in \mathcal K_{-1}$, we have
\[
\limsup_{T\to \infty, \,g^Tv\in \mathcal{K}_{-1}, T\in \mathbb R}-\frac{1}{T}\log \mu(B_\Gamma(v,T,\varepsilon))\leq \limsup_{n\to \infty, \,g^nv\in \mathcal{K}, n\in \mathbb N}-\frac{1}{n}\log \mu(B_\Gamma(v,n,\varepsilon))
\]

If we consider the essential  least upper bound over $v\in\mathcal{K}$ on the left and on the right in~\eqref{eqn:part}, using the ergodicity of $\mu$ and the
Shannon-McMillan-Breiman Theorem, we get

\[
\bar{h}_{\mathrm{loc}}^\Gamma(\mu,\mathcal{K}_{-1})\leq h(\mu,\calP).
\]

This already implies $\bar{h}_{BK}^\Gamma(\mu)\leq h_{KS}(\mu)$ since the right hand side
of the inequality is less than $h_{KS}(\mu)$ and $\mathcal{K}\subset T^1M$ is
arbitrary.

\noindent\\
\textbf{Step 2.} The goal now is to prove the equality between the Katok
entropies and the Kolmogorov-Sinai entropy. The inequality $h_{KS}(\mu)\le
h_{\mathrm{Katok}}^{\mathrm{dyn}}(\mu)$ follows immediately from
Katok~\cite[Formula (1.4)]{Katok80}, where the author considers coverings
instead of spanning sets.  In this reference, equality is proved on a compact
manifold, but the proof of this inequality does not use compactness. The
inequality $h_{\mathrm{Katok}}^{\mathrm{dyn}}(\mu)\le
h_{\mathrm{Katok}}^{\Gamma}(\mu)$ is immediate from~\eqref{eqn:dyn-ball2}.
Hence, by Step 1 we just need to prove that
$h_{\mathrm{Katok}}^{\Gamma}(\mu)\leq \bar{h}_{BK}^\Gamma(\mu)$.

Let $h \coloneqq \bar{h}_{BK}^\Gamma(\mu)$. By definition of local entropy, for any $\rho>0$,
there exists a compact subset $\mathcal{K}\subset T^1M$  and $\varepsilon>0$ such that
$\mu(\mathcal{K})>4/5$ and for $\mu$-a.e.\ $v\in \mathcal{K}$, we have
\[
\limsup_{T\to \infty,\,\, g^Tv\in \mathcal{K}} -\frac{1}{T}\log \mu(B_\Gamma(v,\varepsilon/2;T))\leq h+\rho.
\]
For every $\tau>0$, set
\[
\mathcal{K}_\tau\coloneqq \{v\in \mathcal{K}: \mu(B_\Gamma(v,\varepsilon/2;T))\geq \exp(-T(h+2\rho)), \ \forall T\geq \tau, \ g^Tv\in \mathcal{K}\}.
\]
Then there exists $\tau_0>0$ such that $\mu(\mathcal{K}_{\tau_0})>3/4$. Note that
$\mu(Y_T)>1/2$ for every $T\geq\tau_0$, where $Y_T=\mathcal{K}_{\tau_0}\cap
g^{-T}\mathcal{K}_{\tau_0}$. Let $0<\delta<1/2$. Then
\begin{equation*}
h_{\mathrm{Katok}}^{\Gamma}(\mu) \leq  \limsup_{T\to\infty}\frac{1}{T}\log S_\Gamma(\mu,\delta,\varepsilon;T)
\leq \limsup_{T\to\infty}\frac{1}{T}\log S_\Gamma(Y_T,\varepsilon;T),
\end{equation*}
where $S_\Gamma(Y_T,\varepsilon,T)$ is the minimal cardinality of a
$(\varepsilon,T)$-spanning set of $Y_T$.

Choose a maximal $(\varepsilon/2,T)$-separated set $\mathcal{E}$ in $Y_T$, and denote by $\Sigma_\Gamma(Y_T,\varepsilon/2,T)$
its cardinality. By maximality, $\mathcal{E}$ is also $(\varepsilon,T)$-spanning, so that
$S_\Gamma(Y_T,\varepsilon,T)\le \Sigma_\Gamma(Y_T,\varepsilon/2,T)$.
By construction, we have
\[
e^{-T(h+\rho)}\Sigma_\Gamma(Y_T,\varepsilon/2,T)\le \sum_{y\in \mathcal{E}}\mu(B_\Gamma(y,\varepsilon/2;T))\le 1\,.
\]
With the above inequalities, we deduce that
\[
h_{\mathrm{Katok}}^{\Gamma}(\mu) \leq h+\rho\,.
\]
As $\rho$ is arbitrary, the result follows.

\subsection{Comparison between entropies for orbifolds}

A Riemannian orbifold is said to be \emph{good} when it is the quotient of a
simply connected manifold $\tilde M$ by a discrete group of isometries
$\Gamma$: it is the setting to which all the results in the article apply
except this appendix which assumes moreover that the action of $\Gamma$ is
free, i.e., $\tilde M/\Gamma$ is a manifold. A good orbifold  $M = \tilde
M/\Gamma$ is said to be \emph{very good} when it has a subgroup $\Gamma'<
\Gamma$ of \emph{finite index} acting on $\tilde M$ without fixed point,
i.e., if $M$ has a finite covering which is a manifold.

\medskip

Theorem~\ref{theo:entropies-coincide} extends immediately to very good
orbifolds since the entropies which we consider are invariant by finite
coverings. Unfortunately, it does not extend yet to general good orbifolds
for the following reason.

We have crucially used in the proof of Step 1.\ (which is used for Step 2.)
the fact that the injectivity radius cannot decrease more than exponentially
fast along geodesics. The notion of injectivity radius on orbifold is
delicate: note that the length of the shortest geodesic loop based at $x$
goes to $0$ as $x$ approaches a singularity. There are however notions of
injectivity radius adapted to orbifolds which are automatically positive on
compact sets, such as the \emph{cone injectivity radius} considered
in~\cite[Chapter 9]{BMP03}. Nevertheless, it is unknown whether such
injectivity radius can decrease faster than exponentially along the geodesics
of an orbifolds with bounded sectional curvature. The proof
of~\eqref{eqn:rayon-inj} given for manifolds in~\cite{CGT} is based on the
study of the Riemannian heat kernel. Therefore, its adaptation to orbifolds
is delicate.
%
%
%
%


\addcontentsline{toc}{section}{References}
\bibliography{biblio}
\bibliographystyle{amsalpha}

\setlength{\parindent}{0pt}

Sébastien Gouëzel. Univ Rennes, CNRS, IRMAR UMR 6625, F-35000 Rennes, France.
\\
\href{mailto:sebastien.gouezel@univ-rennes1.fr}{sebastien.gouezel@univ-rennes1.fr}

\medskip

Felipe Riquelme. IMA, Pontificia Universidad Católica de Valpara\'iso,
Blanco Viel 596, Valpara\'iso, Chile.
\\
\href{mailto:felipe.riquelme@pucv.cl}{felipe.riquelme@pucv.cl}

\medskip

Barbara Schapira. Univ Rennes, IRMAR UMR 6625, F-35000 Rennes, France.
\\
\href{mailto:barbara.schapira@univ-rennes1.fr}{barbara.schapira@univ-rennes1.fr}

\medskip

Samuel Tapie. Université de Lorraine, Institut Elie Cartan de Lorraine.
\\
\href{mailto:samuel.tapie@univ-lorraine.fr}{samuel.tapie@univ-lorraine.fr}

\end{document}